\documentclass[12pt]{article}
\usepackage[utf8]{inputenc}
\usepackage{amsmath,amssymb,amsthm,geometry}
\usepackage{xcolor}
\usepackage{tikz,tikz-3dplot}
\usepackage{fullpage}
\usepackage[colorlinks=true,linkcolor=blue]{hyperref}
\hypersetup{
   citecolor=blue,
   linkcolor=blue,
}
\usepackage[psamsfonts]{eucal}
\usepackage{microtype}

\usepackage{color}

 \usepackage{url}
 \usepackage{framed}
\usepackage{latexsym}
 \usepackage{graphicx, epstopdf, verbatim}
\usepackage{psfrag}

\newtheorem{theorem}{Theorem}[section]
\newtheorem{proposition}{Proposition}[section]
\newtheorem{lemma}[theorem]{Lemma}
\newtheorem{corollary}[theorem]{Corollary}
\newtheorem{remark}[theorem]{Remark}

\newtheorem{definition}{Definition}[section]
\newtheorem{notation}{Notation}[section]
\newtheorem{obs}{Observation}[section]

\newtheorem{claim}{Claim}[section]

\newcommand{\E}{\ensuremath{\mathbb E}}
\newcommand{\R}{\ensuremath{\mathbb R}}

\newcommand{\F}{\ensuremath{\mathcal F}}
\newcommand{\FF}{\ensuremath{\mathcal F}}

\newcommand{\reach}{\mathrm{reach}}

\newcommand{\lab}{\label}  \newcommand{\ra}{\ensuremath{\rightarrow}}  \def\a{{\mathbf{\alpha}}} \def\de{{\mathbf{\delta}}} \def\De{{{\Delta}}}  
  \def\beq{\begin{eqnarray}} \def\eeq{\end{eqnarray}} \def\ben{\begin{enumerate}}
\def\een{\end{enumerate}}
 \def\bit{\begin{itemize}}
\def\eit{\end{itemize}}
 \def\beqs{\begin{eqnarray*}} \def\eeqs{\end{eqnarray*}} \def\bel{\begin{lemma}} \def\eel{\end{lemma}}
\newcommand{\N}{\mathbb{N}} \newcommand{\Z}{\mathbb{Z}}  \newcommand{\C}{\mathcal{C}} \newcommand{\CC}{\mathcal{C}}
    
 \newcommand{\I}{\mathbb{I}}   \newcommand{\p}{\mathbb{P}}
\newcommand{\PP}{\mathcal P}  \newcommand{\HH}{\mathcal H}  \newcommand{\one}{\mathrm{1}}
\newcommand{\LL}{\mathcal L} \newcommand{\MM}{\mathcal M}\newcommand{\NN}{\mathcal N} \newcommand{\la}{\lambda}  
 \def\a{{\mathbf{\alpha}}}  \def\eps{{\epsilon}}  \def\ie{i.\,e.\,}

\renewcommand{\a}{\alpha}
\newcommand{\ta}{\tilde{\alpha}}

\newcommand{\tf}{\hat{F}}

\newcommand{\tmu}{\tilde{\mu}}
\newcommand{\tS}{\tilde{S}}

\newcommand{\RR}{\mathbb{R}}
\newcommand{\dist}{dist}
\newcommand{\dhaus}{\mathbf{d}_{\mathtt{haus}}}
\newcommand{\G}{\mathcal{G}}

\renewcommand{\H}{\mathbb{H}}

\newcommand{\D}{\mathcal{\bar{D}}}

\newcommand{\oc}{\overline{c}}

\newcommand{\beqn}{\begin{equation}}
\newcommand{\eeqn}{\end{equation}}
\newcommand{\f}{f^{n-D}}
\newcommand{\ey}{e^{n}_y}
\newcommand{\tbeta}{\theta}

\title{Fitting a manifold of large reach to noisy data}
\usepackage{times, verbatim}
\author{
Charles Fefferman\footnote{Princeton University, Mathematics Department,
Fine Hall, Washington Road,
Princeton NJ, 08544-1000, USA.} \and
 Sergei Ivanov\footnote{
St.~Petersburg Department of Steklov Institute of Mathematics, Russian Academy of Sciences, 
27 Fontanka, 191023 St.~Petersburg, Russia.} \and 
Matti Lassas\footnote{
University of Helsinki,
Department of Mathematics and Statistics, P.O. Box 68,
00014, Helsinki, Finland.} \and
Hariharan Narayanan\footnote{
School of Technology and Computer Science, 
Tata Institute for Fundamental Research, 
Mumbai 400005, India.}}

\begin{document}

\maketitle

{\centering To the memory of Yaroslav Kurylev.\par}

\begin{abstract}

Let $\MM\subset \R^n$ be
a $C^2$-smooth compact  submanifold of dimension $d$.
Assume that the
 volume of $\MM$ is at most $V$ and the reach (i.e.\ the normal injectivity radius) of $\MM$ is greater than $\tau$.
Moreover, let $\mu$ be a probability measure on $\MM$ whose  density on
$\MM$ is a strictly positive Lipschitz-smooth function. Let $x_j\in \MM$, $j=1,2,\dots,N$ be $N$ independent 
random samples from distribution $\mu$. Also, let
$\xi_j$, $j=1,2,\dots, N$ be independent random samples from a Gaussian random variable in $\R^n$ having covariance $\sigma^2I$, where  $\sigma$ is less than a certain specified function of $d, V$ and $\tau$.
We assume that we are given the data  points $y_j=x_j+\xi_j,$ $j=1,2,\dots,N$, modelling random points of $\MM$
with measurement noise. We develop an algorithm which  produces from these data, with high probability,  a $d$ dimensional submanifold $\MM_o\subset \R^n$ whose Hausdorff distance to $\MM$ is less than $\Delta$ for $\Delta > Cd\sigma^2/\tau$ and whose reach is greater than $c{\tau}/d^6$ with  universal constants $C,c > 0$.  The number $N$ of random samples required depends almost linearly on $n$, polynomially on $\Delta^{-1}$ and exponentially on $d$.

\end{abstract}

\tableofcontents

\section{Introduction}\label{sec:intro}
One of the main challenges in high dimensional data analysis is dealing with the exponential growth of the computational and sample complexity of generic inference tasks as a function of dimension, a phenomenon termed ``the curse of dimensionality". One intuition that has been put forward to diminish the impact of this curse is that high dimensional data tend to lie near a low dimensional submanifold of the ambient space. Algorithms and analyses that are based on this hypotheses constitute the  subfield of learning theory known as manifold learning.
In the present work,  we give a solution to the following  question from manifold learning. 
Suppose data is drawn independently, identically distributed (i.i.d) from a measure supported on a low dimensional twice  differentiable ($C^2$) manifold $\MM$ whose reach is $\geq \tau$, and corrupted by a small amount of (i.i.d) Gaussian noise.
How can can we produce a manifold $\MM_o$ whose Hausdorff distance to $\MM$ is small and whose reach is not much smaller than $\tau$? 

This question is an instantiation of the problem of understanding the geometry of data. To give a specific real-world example, the issue of denoising noisy Cryo-electron microscopy (Cryo-EM) images falls into this general category.
Cryo-EM images are X-ray images of three-dimensional macromolecules, e.g. viruses, possessing an arbitrary orientation. The space of orientations is in correspondence with the Lie group $SO_3(\R)$, which is only three dimensional. However, the ambient space of greyscale images on $[0, 1]^2$ can be identified with an infinite dimensional subspace of $\mathcal{L}^2([0, 1]^2)$, which gets projected down to a finite $n-$dimensional subspace indexed by $n = k \times k$ pixels, where $k$ is large. through the process of dividing $[0, 1]^2$ into pixels. Thus noisy Cryo-EM X-ray images lie approximately on an embedding of a compact $3-$dimensional manifold in a very high dimensional space. If the errors are modelled as being Gaussian, then fitting a manifold to the data can subsequently allow us to project the data onto this output manifold. Due to the large codimension and small dimension of the true  manifold,  the noise vectors are almost perpendicular to the true manifold and the projection would effectively denoise the data. The immediate rationale behind having a good lower bound on the reach is that this implies good generalization error bounds with respect to squared loss (See Theorem 1 in \cite{FMN}). Another reason why this is desirable is that the projection map onto such a manifold is Lipschitz within a tube of the manifold of radius equal to $c$ times the reach for any $c$ less than $1$. 

LiDAR (Light Detection and Ranging) also produces point cloud data for which the methods of this paper could be applied.

\subsection{A note on  constants}
In the following sections, we will denote positive absolute constants by $c, C, C_1, C_2, \oc_1$ etc. These constants are universal and positive, but their precise value may differ from occurrence to occurrence.
Also, for a natural number $n$, we will use $[n]$ to denote the set $\{1, \dots, n\}.$

\subsection{Model}
Let $\MM$ be a $d$ dimensional $\C^2$ submanifold of  $\R^n$. We assume $\MM$ has volume ($d-$dimensional Hausdorff measure) less or equal to $V$, reach (see (\ref{eq:reach}), also known as normal injectivity radius) greater or equal to $\tau$, and that $\MM$ has no boundary. Let $x_1, \dots, x_N$ be a sequence of points chosen i.i.d at random from a measure $\mu$ absolutely continuous with respect to the $d$-dimensional Hausdorff measure $\mathcal{H}^d_\MM = \la_\MM$ on $\MM$. More precisely, the Radon-Nikodym derivative $d\mu/d\la_\MM$ is bounded above and below by $\rho_{max}/V$ and $\rho_{min}/V$ respectively, where $\rho_{max}$ and $\rho_{min}$ lie in $[c, C]$ and $\ln\left(d\mu/d\la_\MM\right)$ is ${C}/{\tau}-$Lipschitz (as specified in (\ref{eq:form-2})). Thus, we assume that  \beq\label{eq:density 1} c <  \rho_{min} < V d\mu/d\la_\MM <  \rho_{max} < C,\eeq and also for all $ x, y \in \MM,$  \beq\label{eq:form-2}  \frac{\big|\ln \left(d\mu/d\la_\MM(x)\right)  - \ln\left(d\mu/d\la_\MM(y)\right)\big|}{|x - y|} \leq \frac{C}{\tau}.\eeq 

  Let $G_\sigma^{(n)}$ denote the Gaussian distribution supported on $\R^n$ whose density (Radon-Nikodym derivative with respect to the Lebesgue measure) at $x$ is 
\beq\label{def:gaussian} \pi_{G_\sigma^{(n)}}(x) = \left(\frac{1}{2 \pi \sigma^2}\right)^{\frac{n}{2}}  \exp \left(-\frac{\|x\|^2}{2 \sigma^2} \right).\eeq

Let $\zeta_1, \dots, \zeta_N$ be a sequence of i.i.d random variables independent of $x_1, \dots, x_N$ having the distribution $G_\sigma^{(n)}$. 
We observe 
$$
{ y_i = x_i + \zeta_i,\quad \hbox{ for }i = 1, 2, \dots,N,}
$$ 
and wish to construct a manifold $\MM_o$ close to  $\MM$  in Hausdorff  distance but at the same time having a reach not much less than $\tau$. Note that the distribution of $y_i$ (for each $i$), is the convolution of $\mu$ and $G_\sigma^{(n)}$.  This is denoted by $\mu*G_\sigma^{(n)}$. Let $\omega_d$ be the volume of a $d$ dimensional unit Euclidean ball.
{ Suppose that 
\beq\label{eq:sigma-May}
   \sigma< r_{c} D^{-1/2},\quad \hbox{where }
 r_{c} := cd^{-C}\tau, \  D = \min\left(n, \frac{V}{c^d\omega_d\beta^d}\right),\ \beta =  \tau\sqrt{\frac{c^d \omega_d \tau^d}{V}}.\hspace{-15mm}
\eeq 
and $\Delta \geq \frac{Cd\sigma^2}{\tau}.$
The quantities $r_c$, $D$, $\beta$, and $N_0$ have intuitive interpretations. $r_c$ is the radius of discs fit to the data in the initial stage. $D$ is the dimension of the Principal component Analysis (PCA) subspace, $\beta$ is roughly the maximum distance of a point on the manifold to the PCA subspace. $N_0$ is the number of random discs of radius $r_c$ that need to be chosen in order to cover the manifold.
We observe $y_1, y_2, \dots, y_N$ and for $k \geq 3$, will  produce a description of a $\C^k-$manifold $\MM_o$ such that the Hausdorff distance between $\MM_o$ and $\MM$ is at most $ \Delta$ and $\MM_o$ has reach that is bounded below by $\frac{c\tau}{d^6}$ with probability at least $1 - \eta.$  Note that the required upper bound on $\sigma$ does not degrade to $0$ with as the ambient dimension $n \ra \infty$, but can be controlled by the intrinsic parameters $d$, $V$ and $\tau$.  
We shall be assuming that $\sigma$ is known exactly, however, all the arguments that we use go through if $\sigma^2$ is merely an upper bound on the true variance. This assumption, and how it can be made practical will be discussed at the end of Section~\ref{sec:4}.
The following is our main theorem.

\begin{theorem}\label{thm main}
Let $\MM$ be a $d$ dimensional $\C^2$ submanifold of  $\R^n$. We assume $\MM$ has volume ($d-$dimensional Hausdorff measure) less or equal to $V$, reach (i.e. normal injectivity radius) greater or equal to $\tau$, and that $\MM$ has no boundary. 

{ Let $\mu$ be a probability measure on $\MM$ which density with respect to Hausdorff measure of $\MM$  satisfies (\ref{eq:density 1}) and (\ref{eq:form-2}) with bounds $\rho_{max},\rho_{min}\in [c, C]$.
Let $k \in [3, C]$ be a fixed integer.
Let $x_j\in \MM$, $j=1,2,\dots,N$ be $N$ independent 
random samples from distribution $\mu$. Also, let
$\xi_j$, $j=1,2,\dots, N$ be independent random samples from a Gaussian random variable in $\R^n$ having covariance $\sigma^2I$, see (\ref{def:gaussian}). 
Suppose that $
      \sigma< r_{c} D^{-1/2},
$
where $r_c$ and $D$  are given in (\ref{eq:sigma-May}) and $\Delta \geq \frac{Cd\sigma^2}{\tau}.$
 Let $0 < \eta < 1.$
 \beqs & &\tilde{N}_0 = \frac{CV}{\omega_d r_{c}^d}\ln\left(\frac{CV}{\omega_d r_{c}^d}\right),\\
 & & N =  \left(\frac{n\tau \Delta}{d} + \tau^2\right) \left(\frac{d}{\Delta}\right)^2 \left(\frac{ r_{c}\sqrt{d}}{\sqrt{\tau \Delta}}\right)^d\tilde{N}_0 \left(\log \tilde{N}_0\right)^3\log(\eta^{-1}).\eeqs

Suppose we observe the data points 
$$
y_i = x_i + \xi_i\quad\hbox{for }i = 1, 2, \dots, N.
$$
 Then with probability at least $ 1 - \eta$, using these data 
we can construct a $\C^k-$manifold $\MM_o\subset \R^n$ for any fixed $k\geq 3$ such that  the Hausdorff distance between $\MM_o$ and $\MM$ is at most $\Delta$ and $\MM_o$ has reach that is bounded below by ${c\tau}/{d^6}$.}
\end{theorem}

\begin{remark} Theorem~\ref{thm main} can be formulated using the following requirements for the parameters: Assume that $n, d, V$ and $\tau$ are given and define $r_c, D$ and $\beta$ as in (\ref{eq:sigma-May}). The claim of Theoreem~\ref{thm main} holds if $0 < \eta < 1$ and $\Delta > 0$ are arbitrary and $\sigma$ satisfies $$\sigma < \min(r_c D^{-1/2}, c(\De \tau/d)^{1/2}).$$ The above Theorem has implications even if there is zero noise. In such a scenario, if we desire a Hausdorff distance of $\De$, we can simply set $\sigma^2$ to $\frac{\tau \De}{Cd}$, and synthetically add independent Gaussian noise of this variance to the $N$ samples. The resulting data set can be processed by the algorithm to produce a manifold within a Hausdorff distance of $\De$ with high probability.
\end{remark}

\begin{remark} \lab{rem:1.3}When $\De < \frac{c\tau}{d^6},$ the output manifold $\MM_o$ is $C^1-$diffeomorphic to $\MM$, as proved in Proposition~\ref{prop:final-diff}.
\end{remark}

The above theorem may seem counterintuitive, in that the smaller $\sigma$ is, the larger $N$ is. However note that the Hausdorff distance that we are achieving is $O(\sigma^2)$, which itself decreases quadratically as $\sigma$ tends to zero.  This is the reason for the anomaly. We believe that going below a Hausdorff distance of $O(\sigma^2)$ in the case of $\CC^2$ manfolds would take different techniques and a significantly larger number of samples. Indeed that this is the case for sufficiently small Hausdorff distances was shown in \cite{Wasserman}.
In Proposition~\ref{thm:C3}, we obtain an explicit bound on the magnitude of the third derivatives of $\MM_0$.
{ We emphasize that in Theorem \ref{thm main} the Hausdorff distance of the constructed manifold
$\MM_o$ and the original manifold
$\MM$ as well as the reach of $\MM_o$ do not depend on the dimension $n$ of the ambient space.}
To prove Theorem~\ref{thm main} we develop an algorithm,  using a number of analytic tools, which ensures that the degradation of the reach is polynomial and not exponential in the dimension of the manifold, $d$. We believe that this is the first time this has been achieved. Secondly the number of samples required depends almost linearly on the ambient dimension $n$. This is the second novel feature of our algorithm. A detailed comparison to earlier results is given in Subsection~\ref{ssec:survey}.

\subsection{A survey of related work}\label{ssec:survey}

Let $f : K \rightarrow  \R$ be a function defined on a given (arbitrary) set $K \subset  \R^n$, and let $m \geq 1$ be
a given integer. The classical Whitney problem is the question  
 whether $f$ extends to a function $F \in C^m(\R^n)$ and  if such an
$F$ exists, what is the optimal $C^m$ norm of the extension. Furthermore, one is interested in
the questions
if the derivatives
of $F$, up to order $m$, at a given point can be estimated, or if 
one can construct extension $F$ so that it depends linearly on $f$.

{\color{black}
These questions go back to the work of H. Whitney \cite{W1,W11,W12} in 1934.
In the decades 
since
Whitney's seminal work, fundamental progress was made by G. Glaeser \cite{G}, Y. Brudnyi and
P. Shvartsman \cite{Br1,Br2,Br3,Br4,BS1,BS2} and \cite{Shv1,Shv2,Shv3}, and E. Bierstone-P. Milman-W. Pawluski \cite{BMP}. (See also N.
Zobin \cite{ZO1,ZO2} for the solution of a closely related problem.)

The above questions have been answered in the last few years, thanks to work of E. Bierstone,
Y. Brudnyi, C. Fefferman, P. Milman, W. Pawluski, P. Shvartsman and others, (see
\cite{BMP, Brom, Br1,Br3,Br4,BS2,F1,F2,F3,F5,F6}.) 
  Along the way,
the analogous problems with $C^m(\R^n)$ replaced by $C^{m,\omega }(\R^n)$, the space of functions whose
$m^{th}$ derivatives have a given modulus of continuity $\omega $, (see \cite{F5,F6}), were also solved.

The solution of Whitney's problems has led to a new algorithm for interpolation
of data, due to C. Fefferman and B. Klartag \cite{FK1,FK2}, where the authors show how to compute efficiently an interpolant $F(x),$ whose $C^m$ norm
lies within a factor $C$ of least possible,  where $C$ is a constant depending only on $m$ and $n.$ }

In traditional manifold learning, for instance, by using the ISOMAP algorithm introduced in the seminal paper  \cite{TSL}, one often aims to map points $X_j$ to points $Y_j=F(X_j)$ in an Euclidean space $\R^m$, where $m\geq n$ is as small as possible
so that the Euclidean distances $\|Y_j-Y_k\|_{\R^m}$ are close to the intrinsic  distances $d_M(X_j,X_k)$ and find a submanifold $\tilde M\subset \R^m$ that is close
to the points $Y_j$. { This method has turned out to be very useful, in particular in finding the topological manifold structure of the manifold $(M,g)$. 
It has been shown that when the original manifold  $(M,g)$  has a vanishing  Riemann curvature and satisfies certain 
convexity conditions, the manifold reconstructed by the ISOMAP approaches the original manifold  as the number of
the sample points tends to infinity (see the results in \cite{Bernstien,Donoho1,Donoho} for ISOMAP and \cite{Zha} for the continuum version of ISOMAP). We note that for a general Riemannian manifold, the construction  of}
a map $F:M\to \R^m$, for which { the intrinsic metric of} the embedded manifold $F(M)=\tilde M\subset \R^m$ is isometric to $(M,g)$  is  a very difficult 
task numerically as it means finding a map, the existence  of which is proved by the  Nash embedding theorem
(see \cite{Nash1,Nash2} and \cite{Verma1} on numerical techniques based on the Nash embedding theorem). 
{ We emphasize that the construction of an isometric embedding $f:M\to \R^n$ is outside of the context of the paper.}

%
%
One can  overcome the difficulties related {to the construction of the Nash embedding} by formulating the problem in a coordinate invariant way:
Given the geodesic distances of points sampled from a Riemannian manifold $(M,g)$, construct a manifold $M^*$ with an intrinsic metric tensor $g^*$ so that the Lipschitz distance of $(M^*,g^*)$  to the original manifold $(M,g)$ is small. 
{ The construction of abstract manifolds from the distances of sampled data points  has also been considered by Coifman and Lafon \cite{CoifmanLafon} and Coifman et al.\ \cite{diffusion1,diffusion2}
 using ``Diffusion Maps'',  and by Belkin and Niyogi \cite{BN} using ``EigenMaps'', where 
the data points are mapped to the values of the approximate eigenfunctions or diffusion kernels at the sample points. 
These methods construct 
a non-isometric embedding of the manifold $M$ into $\R^m$ with a sufficiently large $m$.
This construction  is continued in \cite{Meila} by computing an approximation the metric tensor $g$  by
using finite differences to find the Laplacian of  the products of the local coordinate functions. 
In \cite{FIKLN2}, we extend the results of \cite{FIKLN} that deals with the question how a smooth manifold, that approximates a manifold $(M,g)$,  can be constructed, when one is given the
 distances of the points of in a discrete subset $X$ of $M$ with small deterministic errors. In this paper we extend
 these results to two directions. First, the discrete set is randomly sampled and the distances have (possibly large) random errors. Second, we consider the case when some distance information is missing. }

The question of fitting a manifold to data is of interest to data analysts and statisticians \cite{aamari2019,  AizenbudSober, chen2015, Hein, Wasserman, Genovese:2012:MME:2188385.2343687,  kim2015, Sober,  zhigang}.  We will focus our attention on results that provide an algorithm for describing a manifold to fit the data  together with upper bounds on the sample complexity. 

A  work in this direction  \cite{ridge}, building over \cite{Ozertem11} provides an upper bound on the Hausdorff distance between the output manifold and the true manifold equal to $O((\frac{\log N}{N})^{\frac{2}{n+8}}) + {O}(\sigma^2\log (\sigma^{-1}))$. Note that in order to obtain a Hausdorff distance of $c\eps$, one needs more than $\eps^{-n/2}$ samples, where $n$ is the ambient dimension. 
This bound is exponential in $n$ and thus differs significantly from our results.

The  results of the present work  guarantee (for $\sigma$ satisfying (\ref{eq:sigma-May})) that the Hausdorff distance between the output manifold and the true manifold $\MM$ is less than $$\frac{C\sigma^2 d}{\tau} = O(\sigma^2)$$ with probability at least $ 1- \xi$ (with less than  $N$ samples).  Thus our bound on the Hausdorff distance is $O(\sigma^2)$ which is an improvement over  $O(\sigma^2\log(\sigma^{-1}))$, and also, the number of samples needed to get there depends exponentially on the intrinsic dimension $d$, but linearly on $n$. 
The upper bound on the number of samples depends polynomially on $\sigma^{-1}$, the exponent being $d + 4$. Moreover, if the ambient dimension $n$ increases while $\sigma$ decreases, in such a manner as to have \beq\lab{eq:point-out}n\sigma^2 > c \tau^2,\eeq we have the exponent of $\sigma^{-1}$ to be $d+2$.\\
{\bf Comparison with the results of \cite{putative}.}
The  results of \cite{putative}  guarantee (for sufficiently small $\sigma$) a Hausdorff distance of $$Cd^{7} (\sigma \sqrt{n})$$ with less than $$\frac{CV}{\omega_d( \sigma \sqrt{n})^{d}} = O(\sigma^{-d})$$ samples, where $d$ is the dimension of the submanifold, $V$ is in upper bound in the $d$ dimensional volume, and $\sigma$ is the standard deviation of the noise projected in one dimension. The present work improves the results of \cite{putative} in the following two ways. Firstly, the upper bound on the standard deviation $\sigma$ of the permissible noise is independent of the ambient dimension, while in \cite{putative} this upper bound depended inversely on the square root of the ambient dimension.  Secondly, the bound on the Hausdorff distance between the output manifold and the true manifold is less than $\frac{Cd\sigma^2}{\tau}$ rather than $Cd^{7} \sigma \sqrt{n}$ as was the case in \cite{putative}, which for permissible values of $\sigma$ is significantly smaller.
In terms of new methodology, firstly, the present work has a preliminary dimension reduction on to a Principle Component Analysis (PCA) Subspace, which is why,  in the present work, we can bound from above, the  Hausdorff distance between the output and true manifolds by a quantity independent of $n$.
Secondly, there is a two stage process of disc fitting, first with discs of size $\sqrt{\sigma\sqrt{D}\tau}$, (where $D$ is the PCA dimension)  and then with discs of size $\sqrt{d}\sigma.$ In \cite{putative} there was only one stage. Finally, in the present work, we provide an analysis of the third derivative of the output manifold, which is absent from \cite{putative}.

As shown in \cite{Wasserman} the question of manifold estimation with additive noise, in certain cases can be viewed as a question of regression with errors in variables \cite{Fan-Truong}. The asymptotic rates that can be achieved in the latter question are extremely slow. The results of \cite{Wasserman} imply among other things the following. Suppose that in a manifold that is the graph of $y =  \sin(x + \phi)$  we wish to identify with constant probability, the phase $\phi$ to within an additive  error of at most $\eps$, from samples of the form $(x + \eta_1, y + \eta_2)$ where $x$, $\eta_1$ and $\eta_2$ are standard Gaussians. Then the number of samples needed is at least $\exp(C/\eps).$ 

{\bf A heuristic lower bound.} To see why the bound on the number of samples $N = \tilde{O}\left(\left(\frac{\tau}{\sigma}\right)^{d+2} \frac{V}{\tau^d}\right)$ is likely to be a lower bound,\footnote{By $\tilde{O}(f(n)),$ we  mean $O(f(n)\ln^C(f(n))),$ for a universal constant $C.$}  set $\tau = 1$ after scaling.  Consider $S^d \subset \R^{d+1} \subset \R^{d+2},$ where $S^d$ is the unit $d-$dimensional sphere. Suppose $\sigma << 1$, and we look at Gaussian noise that is only in the $e_{n+2}$ direction. This can be seen to be an easier problem, since the original problem is obtained by convolving this distribution with another. Now consider a continuous family of manifolds $ \FF$, characterised by their being identical to $S^d$, except for having a bump $W_\MM \subseteq \MM \in \FF$ in the $e_{d+2}$ direction around $(1, 0, \dots, 0)$ of height $\eta\sigma^2,$ and radial width $O(\sigma)$ where $\eta$ is an unknown parameter in $(0, 1).$ Only samples from the bump carry information about the height of the bump. The probability of getting a sample from the bump is proportional to $\frac{\sigma^d}{V}.$ Even to estimate a real valued parameter $\eta$, that is the mean of a Gaussian of standard deviation $\sigma$,  to within $\sigma^2$ with a constant probability, takes at least $c \sigma^{-2}$ samples from the Gaussian.  It is natural to expect this lower bound to continue to hold in the more complicated setting that we are in, leading to a lower bound of $\tilde{O}\left(\left(\frac{\tau}{\sigma}\right)^{d+2} \frac{V}{\tau^d}\right)$ on the number of samples needed. 
In general we have an upper bound with a worse dependence on $\sigma^{-1}$, namely $N = \tilde{O}\left(\left(\frac{\tau}{\sigma}\right)^{d+4} \frac{V}{\tau^d}\right)$.
It is desirable to prove an optimal lower bound rigorously, but we leave this to future work.

\begin{definition}\label{dhaus}
Given two subsets $X$ and $Y$ of a metric space $(M, d_M)$, we denote by $dist(X, Y)$, the one-sided distance from $X$ to $Y$ which equals $\sup_{x\in X} \inf_{y\in Y} d_M(x, y).$ We denote the Hausdorff distance between $X$ and $Y$, which equals\\ $\max(dist(X, Y), dist(Y, X))$ by $\dhaus(X, Y).$ When $X$ is a singleton $\{x\}$, we abbreviate $dist(\{x\}, Y)$ to $dist(x, Y)$.
\end{definition}
It follows that it is not possible to provide  sample complexity bounds with a inverse polynomial dependence on $\dhaus(\MM_o, \MM)$, where the Hausdorff distance is arbitrarily small.

\begin{definition}[Reach]
The reach of a closed set $A \subseteq \R^m$, denoted $reach(A)$, is the supremum of all $r$ satisfying the following property. If $dist(p, A) \leq r$, then there exists a unique $q \in A$ such that $|p-q| = dist(p, A)$.
\end{definition}
For a smooth submanifold, the reach is the size of the largest neighborhood where the tubular coordinates near the submanifold are defined.

Finally, we mention that there is an interesting body of literature \cite{Boissonnat, Cheng} in computational geometry that deals with fitting piecewise linear manifolds (as opposed to $C^2-$smooth manifolds) to data. The paper \cite{Cheng} presented the first  algorithm for arbitrary $d$, that takes samples from a smooth $d-$dimensional
manifold $\MM$ embedded in an Euclidean space and outputs a simplicial manifold that is
homeomorphic and close in Hausdorff distance to $\MM.$ 
\subsection{Overview of sections of this paper.}
\ben
\item In Section~\ref{sec:GeomPrelim} we discuss preliminaries needed for working with $\C^2$ submanifolds of positive reach, as well as record some analytic and probabilistic facts needed subsequently. 
\item In Section~\ref{sec:4}, we start with projecting the raw data points on to a $D$ dimensional linear subspace $S$ which is such that it minimizes the sum of the squares of the distances to the points. The span of the eigenvectors corresponding to the top $D$ eigenvalues of the covariance matrix of the points, is such a linear subspace (the probability of ties in the eigenvalues is $0$). Projection on to this subspace reduces the variance of the noise by a multiplicative factor of $\frac{D}{n}$, while the reach of the projected manifold is almost as large as the reach of the original manifold (see Lemma~\ref{lem:alpha}).

\begin{figure}
\centering
\includegraphics[scale=0.40]{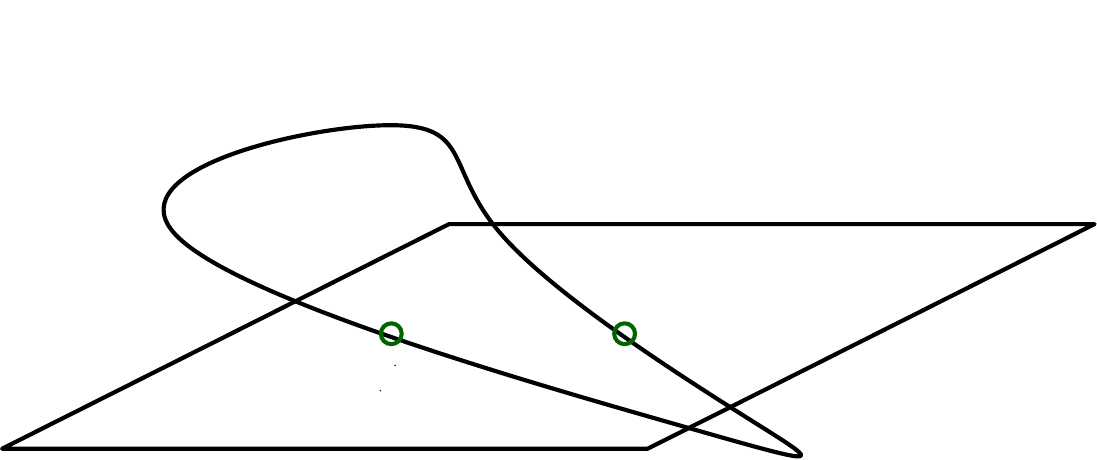}
\caption{We project the manifold on to a PCA subspace (see 2. below). In this picture, the two dimensional  PCA plane cuts the one dimensional manifold embedded in $\R^3$ into two pieces, the points of intersection being marked by green circles.}
\label{fig:PCA}
\end{figure}

\item In Section~\ref{sec:follows}, we find a family of $d$ dimensional putative discs $D_i$  of radius roughly $\sqrt{\sigma \sqrt{D} \tau}$ that approximates the set of projected points to within a Hausdorff distance of order roughly $\sigma \sqrt{D}$.
 We then view the projected data as points in a fiber bundle over a base space. The base space is the disjoint union of the discs in the family mentioned in (2). Each fiber is a disc of dimension $D-d$ and radius $\sigma\sqrt{D}$ centered at its basepoint (see Figure~\ref{fig:firstscale}). 
 
\begin{figure}
\centering
\includegraphics[scale=0.70]{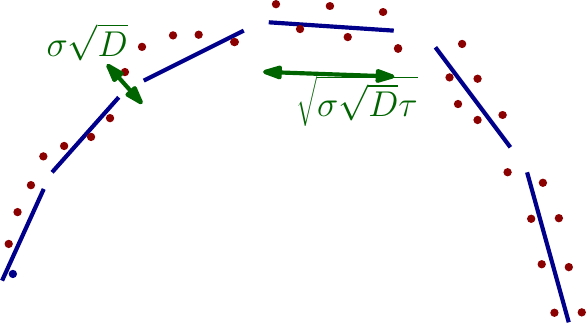}
\caption{After projecting  the data on to a PCA subspace, the projected points are mostly concentrated to a neighborhood of the finite number of $d-$dimensional discs $D_i$ that have radius $\sqrt{\sigma \sqrt{D}\tau}$. This neighborhood can be considered to be a bundle where fibers have diameter $\sigma \sqrt D$, over $d-$dimensional discs of radius  $\sqrt{\sigma \sqrt{D} \tau}$ (See step 3. and Figure~\ref{fig:PCA}).}
\label{fig:firstscale}
\end{figure}

\begin{figure}
\centering
\includegraphics[scale=0.70]{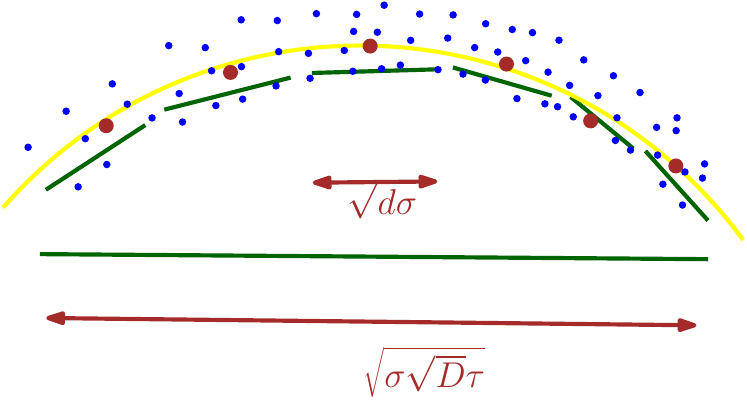}
\caption{After viewing the data as points in a bundle over discs of radius $\sqrt{\sigma \sqrt{D} \tau}$ in step 3., it is locally averaged to obtain the brown points. We then fit discs at a smaller scale of $\sqrt{d} \sigma$ to the brown points (See step 4.).}
\label{fig:secondscale}
\end{figure}

\begin{figure}

\begin{picture}(320,200)
 \put(80,00){\includegraphics[width=8.0cm]{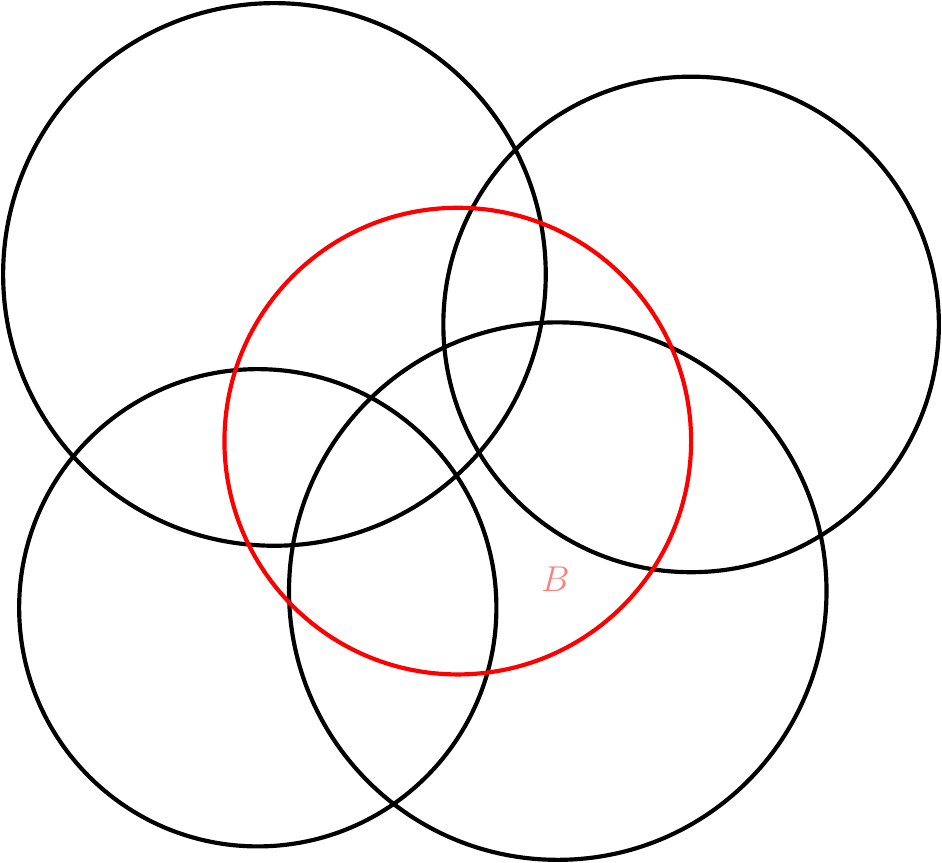}}
\end{picture}

\caption{Exponentially many (in intrinsic dimension $d$) unit balls whose centers are outside a given unit ball $B$ (in red) are needed to cover $B$ (See step 5.).}

 \lab{fig:circles}
\end{figure}
\begin{figure}

\begin{picture}(320,200)
 \put(80,00){\includegraphics[width=7.0cm]{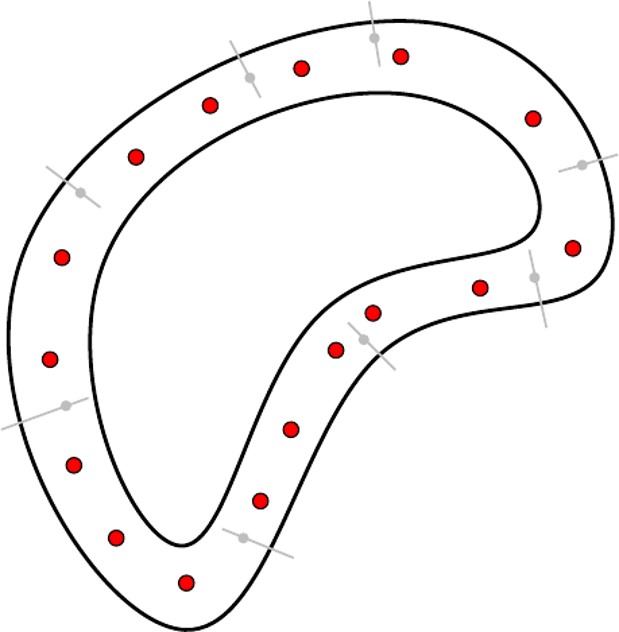}}

\end{picture}

\caption{Using data, we construct an output manifold (See step 6.). The grey lines represent fibers of a certain perturbation of the normal bundle of $\MM$, that is used to construct the output manifold. The manifold itself is defined by an equation of the form $\Pi_x F(x) = 0$, where $\Pi_x$ is an orthogonal projection on to the fiber at $x$, and $F$ is a certain vector valued function.}
\end{figure}

\item In Section~\ref{sec:ref-net},  in each disc $D_i$, we consider a set of lattice points and for each lattice point $x \in D_i$ we consider the Voronoi region in $D_i$ of points closer to $x$ than to the other lattice points in $D_i$ (see Figure~\ref{fig:H0}). This region is a $d-$dimensional cube unless $x$ is close to the boundary of $D_i$. We compute the average of all points whose orthogonal projection to $D_i$ is a point in the Voronoi region just mentioned.
The distance of individual points in the set $\mathrm{Rnet}$  of averages so obtained is within roughly $\frac{d\sigma^2}\tau$ of the original manifold, but the points are now contained in $\R^n$ and not necessarily $S$.

 Using this ``refined net" $\mathrm{Rnet}$ as a data set that replaces the original sample points, we now, for a second time, find a family of discs $\{D'_i\}_{i \in [N_3]}$ that approximate the new data in Hausdorff distance. However this time, the discs have radius roughly $\sqrt{d}\sigma$.
and the Hausdorff distance between the new data set and the union of the discs is of the order of 
$\frac{d\sigma^2}\tau$. Further we prove that the Hausdorff distance between the union of the discs and the manifold $\MM$ is also of the order of $\frac{d\sigma^2}\tau$ (see Lemma~\ref{lem:18} and Figure~\ref{fig:circles}).
\item In Section~\ref{sec:weights} (in the appendix), we design a set of weights $\{\a_i\}_{i \in [N_3]}$ associated with the set of discs $\{D'_i\}_{i \in [N_3]}$ to be used in the next step in the definition of a partition of unity. These weights play a crucial role in obtained a lower bound on the reach of the output manifold $\MM_o$ that differs from the reach of $\MM$ by a factor that is polynomially bounded in $d$. 
Since a point on the submanifold $\MM$ can be contained inside exponential in $d$ many Euclidean balls of dimension $n$ and radius less than $c\tau$ in any cover of $\MM$ with such balls, if one does not take special care their interaction leads to the reach being potentially smaller by  a multiplicative factor that is exponential in $d$. We  will make repeated use of H\"older's inequality for $\ell_p$ norms and $\ell_q$ norms inequalities to circumvent this difficulty.

\item
 In Section~\ref{sec:calc}, we  consider $n-$dimensional balls $U_i \subseteq S$ containing the respective discs $D'_i$ and having the same radius.
We use the union of these balls to construct the output manifold as follows.
We construct a vector bundle in which the base space is $\bigcup_i U_i$ (which is an $n$ dimensional neighborhood of $\MM$), and the fiber at a point $x$ is a $n-d$ dimensional affine subspace that is roughly orthogonal to the affine span of the disc $D'_i$ closest to $x$. This step uses partitions of unity for defining a subspace as a kind of weighted average of subspaces (See Definition~\ref{def:6.9}), with specially designed weights discussed in Section~\ref{sec:weights}.
 For each disc $D'_i$, we consider a bump function supported on $U_i$ corresponding to a partition of unity for $\bigcup_{i \in [N_3]} U_i$, which we use to generate a vector valued function $F$  that  approximates the gradient of the squared distance to the manifold from the individual squared distances to the discs. Finally, the output manifold is defined to be the set of all points $x$ at which $F$ is orthogonal to the fiber at $x$ (see Definition~\ref{def:7}). 
 This manifold is $\C^k$ and a lower bound on the reach of this manifold is obtained that has a polynomial rather than exponential dependence on $d$. A concrete upper bound on the third derivatives of the manifold viewed as the graph of a function is obtained in Proposition~\ref{thm:C3}.
\een

\section{Geometric Preliminaries}\label{sec:GeomPrelim}



We need the following form of the Gaussian concentration inequality, which may be found in Proposition 1.5.7 of \cite{Tao2011AnEO}. Recall from (\ref{def:gaussian}) that $G_\sigma^{(m)}$, is the centered Gaussian distribution with variance $\sigma^2$ supported on $\R^m$
whose density at $x$ is given by 
$ \rho_{G_\sigma}(x) := \left(\frac{1}{2 \pi \sigma^2}\right)^{\frac{m}{2}}  \exp \left(-\frac{\|x \|^2}{2\sigma^2} \right).$ 

\begin{lemma}[Gaussian concentration]\label{lem:Gaussian}

Let $g:\R^m \ra \R$ be a $1-$Lipschitz function and $a = \E g(X)$ where $X$ is a random variable having distribution respect to $G_\sigma^{(m)}$. Then, for $t > 0$,
\beq  G_\sigma^{(m)}\{x:|g(x) - a| \geq t\sigma\} \leq C\exp\left({-ct^2}\right).\eeq
 for some absolute constants $c, C$.
\end{lemma}

\begin{definition}
For a closed $A \subseteq \R^m$, and $a \in A$, let the tangent space (in the sense of Federer \cite{federer_paper}) $Tan^0(a, A)$ denote the set of all vectors $v \in \R^D$ such that for all $\eps > 0$, there exists $b \in A$ such that $0 < |a - b| < \eps$ and $$\bigg|v/{|v|} - \frac{b-a}{|b-a|}\bigg| < \eps.$$
 Let the normal space $Nor^0(a, A)$ denote the set of all $v$ such that for all $w \in Tan^0(a, A)$, we have $\langle v, w\rangle = 0$.
 Let $Tan(a, A)$ (or $Tan(a)$ when $A$ is clear from context) denote the set of all $x$ such that $ x -a \in Tan^0(a, A)$. For a set $X\subseteq \R^m$ and a point $p\in \R^m$,  let $dist(p, X)$ denote the Euclidean distance of the nearest point in $X$ to $p$. Let $Nor(a, A)$  (or $Nor(a)$ when $A$ is clear from context) denote the set of all $x$ such that $ x -a \in Nor^0(a, A)$.
\end{definition}

The following result of Federer (Theorem 4.18, \cite{federer_paper}), gives an alternate characterization of the reach.
 \begin{proposition}[Federer's reach criterion]\label{thm:federer} Let $A$ be a closed subset of $\R^m$. Then,
 \beq\lab{eq:reach} \reach(A)^{-1} = \sup\left\{2|b-a|^{-2}dist(b, Tan(a))\big| \, a, b \in A, a \neq b\right\}.\eeq \end{proposition}

\begin{corollary}\label{cor:1}
Suppose $a, b \in A$ and  $|a-b| < reach(A) := \tau.$ Let $\Pi_a b$ denote the unique nearest point to $b$ in $Tan(a)$. Then,
\beq \frac{dist(b, Tan(a))}{\tau} \leq  \left(\frac{|a - \Pi_{a} b|}{\tau}\right)^2. \eeq
\end{corollary}

\begin{proof} We have
\beq \frac{dist(b, Tan(a))}{\tau} \leq  \frac{|a - b|^2}{2\tau^2} = \frac{|a - \Pi_a b|^2 + dist(b, Tan(a))^2}{2\tau^2}. \eeq
Solving the quadratic inequality in $dist(b, Tan(a))$, we get
\beq  \frac{dist(b, Tan(a))}{\tau} & \leq  & 1 - \sqrt{1 - \left(\frac{|a - \Pi_{a} b|}{\tau}\right)^2} \leq  \left(\frac{|a - \Pi_{a} b|}{\tau}\right)^2.\eeq

\end{proof}

\begin{definition}\label{def:2new}
We say that $\MM$ is a compact $d$ dimensional $C^2$ submanifold of $\R^m$ if $\MM$ is compact and the following is true.
Firstly, for any $p \in \MM,$ the tangent space at $p$ is a $d$ dimensional affine subspace of $\R^m$, and therefore, by an orthogonal transformation, one can choose Euclidean coordinates in $\R^m$ so that the tangent space at $p$ has
the form $Tan(p) = \{(z_1, z_2) \in \R^d\oplus \R^{m-d}| A z_1 + b= z_2\}$ for some matrix $A = A(p)$ and vector $b = b(p)$. Secondly,  there  exists a neighborhood $ V \subseteq \R^m$ of $p$, an open set $U \subset \R^d$ and a $C^2$-smooth function $F: U \ra \R^{m-d}$  such that in the above coordinates 
\beq\label{eq:ast2} \MM \cap V = \{(u, F(u))\big| u \in U \cap \R^d\}. \eeq
\end{definition}
Let $\G(d, m, V, \tau)$ be the set of all $d$ dimensional,  compact $\C^2$ manifolds embedded in $\R^m$ and having reach at least $\tau$ and $d$ dimensional Hausdorff measure less or equal to $V$. Let $\MM \in \G(d, m, V, \tau)$. In the remainder of this section, for $x \in \MM$ denote  the orthogonal projection from $\R^n$
 to the affine subspace tangent to $\MM$ at $x$,  $Tan(x)$ by $\Pi_{x}$. 

\section{Projecting the manifold on to a $D-$dimensional subspace}\label{sec:4}
We shall be assuming that $\sigma$ is known exactly, however, all the arguments that we use go through if $\sigma^2$ is merely an upper bound on the true variance. This assumption, and how it can be made practical will be discussed at the end of this section. Further, for the purposes of the proof, we may assume that $\Delta = \frac{Cd\sigma^2}{\tau},$ because if $\Delta$ is larger by more than $C$ than this, we can simply add i.i.d Gaussian noise of standard deviation $\sqrt{(\sigma')^2 - \sigma^2}$ to each sample, and then assume that we have samples where the standard deviation of the noise is $\sigma'$ rather than $\sigma$. Here, $\sigma'$ is chosen so that $\Delta = Cd(\sigma')^2/\tau$.

This section describes the effect of Principal Component Analysis (PCA), with sufficiently many components on the hidden manifold. It is a preprocessing step involving projecting data sampled from $\mu$ on to a suitable $D$ dimensional linear subspace. After this step, one may assume that the data is $D$ dimensional rather than $n$ dimensional, where $D$ is an integer  that depends only of $d, V$ and $\tau$ as given in (\ref{eq:sigma-May}).

Suppose that we are in the following setting: there is a manifold in the class $\G(d, n, V, \tau)$ and a probability measure $\mu$ supported on this manifold that has a  density with respect to the uniform measure on the manifold the logarithm of which is $\widetilde L$ Lipschitz. Let \beq\lab{eq:tmu}\widetilde{\mu} := \mu \ast G_\sigma^{(n)}.\eeq

Let $S$ be an affine subspace of $\R^n$. Let $\Pi_S$ denote orthogonal projection onto $S$. Let the span of the first $d$ canonical basis vectors be denoted $\R^d$ and the span of the last $n-d$ canonical basis vectors be denoted $\R^{n-d}$. Let $\omega_d$ be the $d$ dimensional Lebesgue measure of the unit Euclidean ball in $\R^d$. 
Given $\a \in (0, 1),$ let \beq \beta := \beta(\a) =  \sqrt{ (1/10) \left(\frac{\a^2\tau}{2}\right)^2\left(\frac{\a^2\tau}{4}\right)^d \left( \frac{\omega_d\rho_{min}}{V}\right)}.\label{eq:beta1}\eeq 

 Let \beq \label{eq:beta} D := \left\lfloor  \frac{ V}{\omega_d \beta^d}\right\rfloor + 1.\eeq
 Let 
$\eps < \beta^2/2$. Below, $\de > 0$ will be a small parameter that gives a bound on the probability that the conclusion $\sup_{x \in \MM} dist(x, S) < \alpha^2\tau$ in  Proposition~\ref{lem:5-May} fails. 
Choose 

\beq\label{cons:3} R = C\sigma\sqrt{n} +  C \sigma \sqrt{\log (Cn\sigma^2/(\eps\de))},\eeq and
\beq\label{cons:4} N_0=\lfloor C(n\sigma^2 + \sigma^2 \log(Cn\sigma^2/(\eps\de)))\sqrt{\log(C/\de)}(D/\eps^2)\rfloor,\eeq
where $C$ is a sufficiently large universal constant, and in doing so simultaneously satisfy (\ref{cons:1}) and (\ref{cons:2}).

Our parameters are chosen such that \beq\label{cons:2} N_0 \geq  C(R^2 D/\epsilon^2)\sqrt{\log( C/\de)}\eeq and \beq\label{cons:1} R \geq C\sigma\sqrt{n} +  \sigma \max\left(C \sqrt{\log {CN_0/\de}}, C \sqrt{\log(Cn\sigma^2/\eps)}\right),\eeq where $C$ is a sufficiently large universal constant.

Note that due to the slow growth of the the $\sqrt{\log N_0}$ term in (\ref{cons:1}), it is possible to set $N_0$ in (\ref{cons:2}) in a way that is consistant with the definition of $R$ in (\ref{cons:1}).
\begin{proposition} \label{lem:5-May}
Given $N_0$ data points $\{x_1, \dots, x_{N_0}\}$ drawn i.i.d from $\tmu$, let $S$ be a $D$ dimensional affine subspace that minimizes \beq \sum_{i=1}^{N_0} dist(x_i, \tS)^2, \eeq subject to the condition that $\tS$ is an affine subspace of dimension $D$,  and $\beta < c \tau$, where $\beta$ is given by (\ref{eq:beta1}).

Then, \beq \p[\sup_{x \in \MM} dist(x, S) < \alpha^2\tau] > 1 - \delta.\eeq

\end{proposition}
In order to prove Proposition~\ref{lem:5-May}, we need some tools, which we proceed to develop. We will present the proof of the above proposition after presenting the proof of Lemma~\ref{lem:kplanes}.
We will need the following form of Hoeffding's inequality.
\begin{lemma}[Hoeffding's Inequality]
Let $X_1, \dots, X_{N_0}$ be i.i.d copies of the random variable $X$ whose range is $[0, 1]$. Then,
\beq \p\left[\left|\frac{1}{N_0}\left(\sum_{i=1}^{N_0} X_i\right) - \E[X]\right| \leq \eps\right] \geq 1 - 2 \exp(-2N_0\eps^2). \eeq
\end{lemma}

Let $\PP$ be a probability distribution supported on $B :=\{x \in \RR^n \big|\, \|x\| \leq 1\}$.
Let $\H := \H_D$ be the set whose elements are affine subspaces $S \subseteq \R^n$ of dimension $ D$, each of which  intersects $B$.
Let $\H^0 = \H_D^0$ be the set of linear subspaces of dimension $D$.
Let $\FF_{D}$ be the set of all loss functions $F(x) =  \dist(x, H)^2$ for some $H \in \H$
(where $\dist(x, S) := \inf_{y \in S} \|x - y\|$).
Let $\FF_{D}^0$ be the set of all loss functions $F(x) =  \dist(x, H)^2$ for some $H \in \H^0$. 
We wish to obtain a probabilistic upper bound on
\beq \label{eq1}\sup_{F \in \FF_{D}} \Bigg | \frac{\sum_{i=1}^{N_0}  F(x_i)}{N_0} - \E_\PP F(x)\Bigg |, \eeq
where $\{x_i\}_1^s$ is the training set and $\E_\PP F(x)$ is the expected value of $F$ with respect to $\PP$.
 In our situation,  (\ref{eq1}) is measurable and hence a random variable because $\FF$ is a family of bounded piecewise quadratic functions, continuously parameterized by $\H_{d}$, which has a countable dense subset, for example, the subset of elements specified using rational data.
We obtain a probabilistic upper bound on (\ref{eq1}) that is independent of $n$, the ambient dimension.
\begin{lemma}\label{lem:kplanes}
 Let $x_1, \dots, x_{N_0}$ be i.i.d samples from $\PP$, a distribution supported on the ball of radius $1$ in $\R^m$.
 
 Then, firstly,
\beqs \p\left[\left\|\frac{\sum_{i=1}^{N_0} x_i}{N_0} - \E_\PP x\right\| \leq 2 \left(\sqrt{\frac{1}{N_0}}\right)\left(1 + \sqrt{2 \ln (4/\de)}\right)\right] > 1 - \delta. \eeqs
Secondly, 
$$\p\left[\sup\limits_{F \in \F_{ D}} \Bigg | \frac{\sum_{i=1}^{N_0} F(x_i)}{N_0} - \E_\PP F(x)\Bigg | \leq  2 \left({\frac{\sqrt{D}+2}{\sqrt{N_0}}}\right)\left(1 + \sqrt{2 \ln (4/\de)}\right)\right] > 1 - 2\de. $$
\end{lemma}
\begin{proof}
Any $F \in \FF_{D}$ can be expressed as $F(x) = \dist(x, H)^2$ where  $H$ is an affine subspace of dimension equal to $D$ that intersects the unit ball. 
We see that $ \dist(x, H)^2$ can be expressed as
\beqs  \left(\|x \|^2 - x^{\dag} A^{\dag} Ax \right), \eeqs
where $A$ is the orthogonal projection onto the linear subspace $H$.  Thus, $F$ is defined using $H \in \H^0$, where \beqs F(x) := \left(\|x\|^2 - x^{\dag}   A^{\dag}  A x\right). \eeqs

Now, define vector valued maps $\Phi$ and $\Psi$ whose respective domains are the space of $D$ dimensional affine subspaces and $B$ respectively.
\beqs \Phi(H) := \left(\frac{1}{\sqrt{D}}\right)A^{\dag}  A \eeqs
and \beqs \Psi(x) :=  x x^{\dag},\eeqs
where $A^{\dag}  A$ and $xx^{\dag} $ are interpreted as rows of $n^2$ real entries.

Thus, \beq F(x) & = & \left(\|x\|^2 -   x^{\dag}  A^{\dag}  Ax\right)\\
& = &  \|x\|^2 + \sqrt{D}\Phi(H) \cdot \Psi(x),\eeq where the dot product is the inner product corresponding to Frobenius norm.
We see that since $\|x\| \leq 1$, the Frobenius norm (which equals the operator norm in this case) of 
$\Psi(x)$ is $\|\Psi(x)\|\leq 1$. The Frobenius norm  $\|A^{\dag}  A\|_F^2$ is equal to $Tr(AA^{\dag} AA^{\dag} ),$ which is the rank of $A_i$ since $A_i$ is a projection. Therefore, $$d \|\Phi(H)\|^2 \leq   \|A^{\dag}  A\|^2  \leq D$$
{ and}
\beq \nonumber \p\left[\sup\limits_{F \in \F_{ D}^0} \Bigg | \frac{\sum_{i=1}^{N_0} F(x_i)}{N_0} - \E_\PP F(x)\Bigg | > \eps\right]\le p^{(1)}+ p^{(2)}
\eeq 
where
\beq 
 p^{(1)}=\p\left[\Bigg | \frac{\sum_{i=1}^{N_0} \|x_i\|^2}{N_0} - \E_\PP \|x\|^2\Bigg | > \eps/2\right]\label{eq:57.sep02-2020}\eeq 
 and    
 \beq  p^{(2)}= \p\left[\sup\limits_{H \in \H_{ D}^0} \Bigg | \frac{\sum_{i=1}^{N_0}    \Phi(H) \cdot \Psi(x_i)}{N_0} - \E_\PP \Phi(H) \cdot \Psi(x)\Bigg | > \frac{\eps}{2\sqrt{D}} \right]\label{eq:9} .\eeq

The first term, namely $p^{(1)}$ can be bounded above using Hoeffding's inequality as follows
\beq  \p\left[\Bigg | \frac{\sum_{i=1}^{N_0} \|x_i\|^2}{N_0} - \E_\PP \|x\|^2\Bigg | \geq \eps/2\right]  \leq 2 \exp(-N_0\eps^2/2).\eeq

In order to bound $p^{(2)}$, we will use the notion of Rademacher complexity described below.
\begin{definition}[Rademacher Complexity]
Given a class $\F$ of functions $f:X \ra \R$ a measure $\mu$ supported on $X$, and a natural number $s \in \N$, and an $s-$tuple of points $(x_1, \dots, x_s)$, where each $x_i \in X$ we define the empirical Rademacher complexity $R_s(\F, x)$
as follows.
Let $\sigma = (\sigma_1, \dots, \sigma_s)$ be a vector of $s$ independent Rademacher random variables (which take values $1$ and $-1$ with equal probability).
Then,
$$R_s(\F, x) := \E_{\sigma}\left(\frac{1}{s}\right)\left[\sup_{f\in \F}\left(\sum_{i=1}^s \sigma_i f(x_i)\right)\right].$$
\end{definition}

We will use Rademacher complexities to bound the sample complexity from above. Let $X = B_n$ be the unit ball in $\R^n$. Let $\mu$ be a measure supported on $X.$ Let $\F$ be a class of functions $f:X \ra \R$. In our context, the functions $f$ are indexed by elements $H$ in $H_d^0$ and $f(x) = \Phi(H) \cdot \Psi(x)$ for any $x \in X$. Let $\mu_{N_0}$ denote the uniform counting probability measure on $\{x_1, \dots, x_{N_0}\},$ where $x_1, \dots, x_{N_0}$ are $N_0$ i.i.d draws from $\mu.$ Thus $\E_{\mu_{N_0}} f$ is shorthand for $(1/N_0) \sum_{i} f(x_i)$. We know (see Theorem $3.2$, \cite{Boucheron2005})
that for all $\de >0$,
\beq \label{eq:fat_to_gen1}\p_{(x_i)\sim \PP^{N_0}} \left[\sup_{f \in \F} \bigg |\E_\mu f - \E_{\mu_{N_0}} f\bigg | \leq 2 R_{N_0} (\F, x) + \sqrt{\frac{2\log (2/\de)}{N_0}}\right] \geq 1- \de. \eeq
Applying this inequality to the term in (\ref{eq:9}) we see that 
\beq \p\left[\sup\limits_{H \in \H_{ d}^0} \Bigg | \frac{\sum_{i=1}^{N_0}   \Phi(H) \cdot \Psi(x_i)}{N_0} - \E_\PP \Phi(H) \cdot \Psi(x)\Bigg | <  \frac{\eps}{2\sqrt{D}}\right] > 1- \de, \eeq
where \beq \label{eq:13} \frac{\eps}{2\sqrt{d}} >  \E_{\sigma}\frac{1}{N_0}\left[\sup\limits_{H \in \H_{ D}^0}\left(\sum_{i=1}^{N_0} \sigma_i  \Phi(H) \cdot \Psi(x_i) \right)\right]+ \sqrt{\frac{2\log (2/\de)}{N_0}}.\eeq
In order for the last statement to be useful, we need a concrete upper bound on $$ \E_{\sigma}\frac{1}{N_0}\left[\sup\limits_{H \in \H_{ D}^0}\left(\sum_{i=1}^{N_0} \sigma_i  \Phi(H) \cdot \Psi(x_i) \right)\right],$$ which we proceed to obtain.
{ We have}
\beq 
 \E_{\sigma}\frac{1}{N_0}\left[\sup\limits_{H \in \H_{ D}^0}\left(\sum_{i=1}^{N_0} \sigma_i  \Phi(H) \cdot \Psi(x_i) \right)\right] & \leq & \E_{\sigma}\frac{1}{N_0}\left[\Bigg \| \sum_{i=1}^{N_0} \sigma_i  \Psi(x_i)\Bigg\| \right]\\
& \leq &   \E_{\sigma}\frac{1}{N_0}\left[\Bigg \| \sum_{i=1}^{N_0} \sigma_i  \Psi(x_i)\Bigg\|^2 \right]^{\frac{1}{2}}\\
& = & \frac{1}{N_0}\left[\sum_i  \|  \Psi(x_i)\|^2 \right]^{\frac{1}{2}}\\
& \leq & \frac{1}{\sqrt{N_0}}. \eeq

Plugging this into (\ref{eq:13}) we see that for $$\eps > 2 \left(\sqrt{\frac{D}{N_0}}\right)\left(1 + \sqrt{2 \ln (4/\de)}\right)$$
and
$$\p\bigg[\sup\limits_{F \in \F_{ d}} \Bigg | \frac{\sum_{i=1}^s F(x_i)}{N_0} - \E_\PP F(x)\Bigg | < \eps\bigg] > 1 - \de. $$

The first claim of the lemma similarly follows from (\ref{eq:fat_to_gen1}). Let $\PP_{N_0}$ be the uniform measure on $\{x_1, \dots, x_{N_0}\}$.
A direct calculation shows that if $H^0$ is the translate of $H$ containing the origin, and for any $x$, the foot of the perpendicular from $x$ to $H$ is $q_x$ and the foot of the perpendicular from $x$ to $H^0$ is $q_x^0$,  then $$ \E_{\PP_{N_0}} \dist(x, H^0)^2 - \E_{ \PP} \dist(x, H^0)^2 - (\E_{\PP_{N_0}} \dist(x, H)^2 - \E_{ \PP} \dist(x, H)^2) $$ can be expressed as \beqs (\E_{\PP_{N_0}} - \E_{\PP}) (\dist(x, H^0)^2 - \dist(x, H)^2)
& = & (\E_{\PP_{N_0}}  - \E_{\PP}) ( |x - q^0_x|^2 -  |x - q|^2)\\
& = & (\E_{\PP_{N_0}}  - \E_{\PP}) (2 \langle x, q^0_x- q\rangle + |q_x^0|^2- |q|^2).\eeqs

This is in magnitude less than $|2(\E_{\PP_{N_0}}  - \E_{\PP}) ( x)|$ which by the first claim of the lemma is bounded by 
$$4 \left(\sqrt{\frac{1}{N_0}}\right)\left(1 + \sqrt{2 \ln (4/\de)}\right)$$ with probability greater than  $1 - \de$. The second claim of the Lemma follows.

\end{proof}

\begin{proof}[Proof of Proposition~\ref{lem:5-May}]

Let $x_i = y_i + z_i$ where for each $i \in [N_0]$, $y_i$ is a random draw from $\mu$ supported on $\MM$ and $z_i$ is an independent Gaussian sampled from $G(0, \sigma^2)$, and the collection $ \{(y_i, z_i)\}$ is independent, \ie comes from the appropriate product distribution $(\mu \times G(0, \sigma^2))^{\times N_0}.$

For each $i$, let $\widehat z_i$ equal $z_i$ if $|z_i| < R$ and let $\widehat z_i =0$ otherwise. Let the distribution of $\widehat z_i$ be denoted $\widehat G$. 

We shall first establish the following claim.

\begin{claim}\label{cl:1may17} If $$\E_{y \sim \mu} dist(y, S)^2 < \left(\frac{\a^2\tau}{2}\right)^2\left(\frac{\a^2\tau}{4}\right)^d \omega_d  \rho_{min}.$$ then 
 $$ \sup_{x \in \MM} dist(x, S) < \alpha^2\tau. $$ 
\end{claim}
\begin{proof}
If $$ \sup_{x \in \MM} dist(x, S) \geq \alpha^2\tau, $$ since $\MM$ is compact, the supremum is achieved at some point $x_0.$ Thus, any point within $B_{\frac{\a^2\tau}{2}}(x_0)$ is at a Euclidean distance of at least $\a^2\tau/2$ from $S$. Observe that, 
for any $x \in \MM$, $B_{\frac{\a^2\tau}{2}}(x) \cap \MM$ is the graph of a function over the orthogonal projection of $B_{\frac{\a^2\tau}{2}}(x) \cap \MM$ onto $Tan(x)$, which, by Federer's reach criterion,  contains a $d$ dimensional ball of radius at least  $\frac{\a^2\tau}{4}.$
Consequently, 
\beq \E_{y \sim \mu} dist(y, S)^2 & = & \int_\MM dist(y, S)^2 \mu(dy)\\
                                                    & \geq &  \left(\frac{\a^2\tau}{2}\right)^2 \inf_{\widehat y \in \MM} \mu\left(B_{\frac{\a^2\tau}{2}}(\widehat y) \cap \MM\right)\\
& \geq &\left(\frac{\a^2\tau}{2}\right)^2\left(\frac{\a^2\tau}{4}\right)^d \omega_d  \rho_{min} .\eeq
\end{proof}

\begin{definition}[$\eps-$net]
Let $(X, dist)$ be a metric space. We say that $X_1$ is an $\eps-$net of $X$, if $X_1 \subseteq X$ and for every $x \in X$, there is an $x_1 \in X_1$ such that $dist(x, x_1) < \eps.$
\end{definition}
\begin{claim}\lab{cl:fed}
 The volume of the intersection of an $n$ dimensional ball of radius $3\beta/2$ centered at a point in $\MM$ with $\MM$ is greater than $\omega_d \beta^n$.
\end{claim}
\begin{proof}
Recall from (\ref{eq:beta1}) that \beqs \beta = \beta(\a) =  \sqrt{ \left(\frac{1}{10}\right) \left(\frac{\a^2\tau}{2}\right)^2\left(\frac{\a^2\tau}{4}\right)^d \left( \frac{\omega_d\rho_{min}}{V}\right)}.\eeqs By Lemma~\ref{cl:g1sept}, 
If \beqs U = \{y\in \R^m\big||y-\Pi_xy| \leq \tau/4\} \cap  \{ y\in \R^m\big||x-\Pi_xy| \leq \tau/4\},\eeqs 
then, $$\Pi_x(U \cap \MM) = \Pi_x(U).$$ It therefore suffices to show that \ben \item the intersection of an $n$ dimensional ball of radius $3\beta/2$ centered at a point $x$ in $\MM$ with $\MM$ is nowhere at a distance greater than $\frac{\tau}{4}$ from $\Pi_x(U)$, and 
\item the radius of $\Pi_x(U)$ is greater than $\beta.$
\een
 The first point above follows from Federer's reach criterion, \ie Proposition~\ref{thm:federer}, while the second follows from Pythogoras.
\end{proof}
Note that $\MM$ can be provided with $3\beta/2$ net  with respect to Euclidean distance of size $D$ because the volume of the intersection of an $n$ dimensional ball of radius $3\beta/2$ centered at a point in $\MM$ with $\MM$ is greater than $\omega_d \beta^n$ (by  Claim~\ref{cl:fed}),  and $\omega_d \beta^n$  is greater  than $\frac{V}{D}$ by (\ref{eq:beta}). Next, let $\widehat S \subset \R^n$ be the linear span of a minimal $3\beta$ net of $\MM$.  Then, \beq\label{eq:27.9}  \E_{y \sim \mu} dist(y, \widehat S)^2 \leq 9 \beta^2.\eeq
Let $\eps < \beta^2/2$. 
By the definition of $S$, 
\beq \label{eq:28} \sum_{i=1}^N dist(x_i, S)^2 \leq \sum_{i=1}^N dist(x_i, \widehat S)^2. \eeq

\begin{claim} By our choice of $R$ and $N_0$, with probability greater than $1-\de/2$, for all $i \in [N_0]$, $x_i = y_i + \widehat z_i$.
\end{claim}
\begin{proof}
It suffices to show that \beq I_R := \int_{|x|>R} (2\pi\sigma^2)^{n/2}\exp(-|x|^2/(2\sigma^2))dx < 1 - (1 - \de/2)^{1/N_0}.\eeq
The left hand side $I_R$ can be bounded above as follows. 
\beqs I_R \exp(R^2/(4\sigma^2)) & \leq &  \int_{\R^n} (2\pi\sigma^2)^{-n/2}\exp(-|x|^2/(2\sigma^2)) \exp(|x|^2/(4\sigma^2))dx\\
& = & 2^{n/2}.\eeqs
From (\ref{cons:1}), $$R \geq C\sigma\sqrt{n} +  C\sigma \sqrt{\log (CN/\de)},$$
and so \beq I_R \leq 2^{n/2} \exp(-R^2/(4\sigma^2)) & \leq & 2^{n/2} \exp \left(- Cn -  C\log (CN_0/\de)\right)\\
& \leq & \frac{\de}{CN_0}\\
& \leq &  1 - (1 - \de/2)^{1/N_0}.\eeq

\end{proof}
In Lemma~\ref{lem:kplanes}, note that $\FF_D$ is the set of quadratic functions, given by the squared distance to a $D-$dimensional subspace that is of at most a unit distance from the origin.  Thus, by Lemma~\ref{lem:kplanes}, with probability greater than $1 - \de/2$, we have 
\beq \sup_{\widetilde S} |(1/N_0)\sum_{i=1}^{N_0} dist(y_i + \widehat z_i, \widetilde S)^2 -  \E_{(y, \widehat z) \sim \mu \times \widehat G} dist(y+ \widehat z, \widetilde S)^2| <\eps/2.\eeq
By (\ref{eq:tmu}),  $\widetilde{\mu} = \mu \ast G_\sigma^{(n)}.$
\begin{claim} 
By our choice of $R$, \beq \sup_{\widetilde S} |  \E_{x \sim \widetilde \mu} dist(x, \widetilde S)^2 -\E_{(y, \widehat z) \sim \mu \times \widehat G} dist(y+ \widehat z, \widetilde S)^2| < \eps/2.\eeq
\end{claim}
\begin{proof}
For any fixed $\widetilde{S}$, 
\beq \nonumber  -  \E_{x \sim \widetilde \mu} dist(x, \widetilde S)^2 + \E_{(x, \widehat z) \sim \mu \times \widehat G} dist(x + \widehat z, \widetilde S)^2 = \E_{\widehat z \sim \widehat G} dist( \widehat z, \widetilde S)^2,\eeq because for the vector valued random variable $p = x - \Pi_{\widetilde S} x$, $|p| = dist(x, \widetilde S)$ and we have $$\E|p|^2 = \E|p- \E p|^2 + |\E p|^2.$$ Therefore,
\beqs \nonumber  & &\sup_{\widetilde S} |  \E_{x \sim \widetilde \mu} dist(x, \widetilde S)^2 -\E_{(y, \widehat z) \sim \mu \times \widehat G} dist(y + \widehat z, \widetilde S)^2| \\ & =  &  \sup_{\widetilde S} \E_{\widehat z \sim \widehat G} dist( \widehat z, \widetilde S)^2\\
& =& \int_{|(x_{D+1}, \dots, x_n)|>R} |x|^2(2\pi\sigma^2)^{n/2}\exp(-|x|^2/(2\sigma^2))dx\nonumber\\
& \leq & \int_{|x|>R} |x|^2(2\pi\sigma^2)^{n/2}\exp(-|x|^2/(2\sigma^2))dx\nonumber\\
& =: & J_R
.\eeqs
Proceeding as with the preceeding claim,
\beqs J_R \exp(R^2/(4\sigma^2))&  \leq & \int_{\R^n} |x|^2 (2\pi\sigma^2)^{-n/2}\exp(-|x|^2/(2\sigma^2)) \exp(|x|^2/(4\sigma^2))dx\\
& = & 2n\sigma^2 2^{n/2}.\eeqs Since by (\ref{cons:1}), $R \geq C\sigma\sqrt{n} +  C\sigma \sqrt{\log(Cn\sigma^2/\eps)}$, we have $J_R \leq \eps/2$ and the claim follows.
\end{proof}
Thus, with probability greater than $1 - \de$,
\beq \sup_{\widetilde S} |(1/N_0)\sum_{i=1}^{N_0} dist(x_i, \widetilde S)^2 -  \E_{x \sim \widetilde \mu} dist(x_i, \widetilde S)^2| < \eps.\eeq
Therefore, by (\ref{eq:28}) and the above, with probability greater than $1 - \de$,
\beq  \E_{x \sim \widetilde \mu} dist(x,  S)^2 \leq  \E_{x \sim \widetilde \mu} dist(x, \widehat S)^2 + 2\eps.\eeq
Therefore, expanding $x = y + z$, we have 

 \beq \E_{x \sim \widetilde \mu} dist(y,  S)^2 &\leq&  \E_{x \sim \widetilde \mu} dist(y, \widehat S)^2 + 2\eps\\
& \leq & 9\beta^2 + 2\eps \,\,\,\,\,\,\,\,\,\,\,\,\,\,\,\,\,\,\,\,\,\,\,\,\,\,\,\,\,\,\,\,\,(by\,\, (\ref{eq:27.9}))\\
& < & \left(\frac{\a^2\tau}{2}\right)^2\left(\frac{\a^2\tau}{4}\right)^d \left(\frac{\omega_d  \rho_{min}}{V}\right).\eeq

Therefore, by Claim~\ref{cl:1may17},  with probability greater than $1 - \de$, $$ \sup_{x \in \MM} dist(x, S) < \alpha^2\tau. $$ 
This proves  Proposition~\ref{lem:5-May}.
\end{proof}

\begin{lemma}\label{lem:alpha}
Suppose $\MM$ is a $C^2$ submanifold of $\RR^n$ having reach $\tau$ and $S$ is a ${D}$ dimensional linear subspace such that $\sup_{x \in \MM} dist(x, S) < \alpha^2 \tau$ where $\alpha < \frac{1}{4}$.  Then $\Pi_S(\MM)$ is a submanifold of $\R^n$ having reach at least  $(1- 4 \a^2)\tau$.
\end{lemma}
\begin{proof}
Without loss of generality, set $\tau = 1.$ Since $\MM$ is compact, and $\Pi_S$ is continuous, $\Pi_S(\MM)$ is compact. We will first show that the reach of $\Pi_S(\MM)$ is greater than $1 - 4\a^2.$ 
Note that for any $x', y' \in \MM$ such that $|x' - y'| > \a$, \beq \frac{|\Pi_S (x' - y')|}{|x'-y'|} & =  & \sqrt{1 - \frac{|\Pi_{S^\perp} (x'-y')|^2}{|x'-y'|^2}}
\label{eq:58alpha}  \geq  \sqrt{1 - 4 \alpha^2}.\eeq 
Let $|x'-y'| \leq \a$.
Let \beqs \widehat U:= \{y\in \R^n\big||y-\Pi_xy| \leq 1/4\} \cap  \{y\in \R^n\big||x-\Pi_xy| \leq 1/4\}.\eeqs 
As $\MM$ is a $C^2$ submanifold of $\R^n$, by Lemma~\ref{lem:6}, there exists a $C^{2}$ function $F_{x, \widehat U}$ from $\Pi_x( \widehat U)$ to $\Pi_x^{-1}(\Pi_x(0))$ such that
\beqs  \{ y + F_{x, \widehat U}(y) \in \R^n \big | y \in \Pi_x(\widehat U)\} = \MM \cap \widehat U.\eeqs 
By Corollary~\ref{cor:1}, 
\beq |F_{x, \widehat U}(y)|&  \leq & {|x - \Pi_x y|^2}
\label{eq:feb17-81} \leq  \a^2.\eeq
Therefore, the Hausdorff distance between $\MM \cap \widehat U$ and the disc $Tan(x) \cap \widehat U$ is less or equal to $\a^2$.
Consequently, \beq \label{eq:61} \sup_{\bar x \in Tan(x) \cap \widehat U} dist(\bar x, S) < \a^2 +  \alpha^2 = 2 \a^2.\eeq
In particular, this implies that the dimension of $S$ is at least $d$. We observe that 
\beq  \frac{|\Pi_S (x' - y')|}{|x'-y'|} & = & \frac{|\Pi_S((x' + \frac{\a(y'-x')}{|y'-x'|}) - (x' + \frac{\a(x'-y')}{|x'-y'|}))|}{|(x' + \frac{\a(y'-x')}{|y'-x'|}) - (x' + \frac{\a(x'-y')}{|x'-y'|})|}\\
& = & \sqrt{1 - \left(\frac{|\Pi_{S^\perp}((x' + \frac{\a(y'-x')}{|y'-x'|}) - (x' + \frac{\a(x'-y')}{|x'-y'|}))|}{|(x' + \frac{\a(y'-x')}{|y'-x'|}) - (x' + \frac{\a(x'-y')}{|x'-y'|})|}\right)^2}\\ & \geq & \sqrt{1 - \left(\frac{4\a^2}{2\a}\right)^2}
\label{eq:64}  =  \sqrt{ 1 - 4 \a^2}.\eeq
For  $\eps>0$ let \beq U_\eps := \{y \in \R^n \big||y-\Pi_xy| \leq 1/4\} \cap  \{y\in \R^n\big||x-\Pi_xy| \leq \eps\}.\eeq
Then, \beqs U_\eps \cap \{x\in \R^n|dist(x, S) < \a^2\} = \{y\in \R^n\big||x-\Pi_xy| \leq \eps\} \cap  \{x\in \R^n|dist(x, S) < \a^2\}. \eeqs
Therefore, by (\ref{eq:61}) and the Lipschitz nature of the gradient (by (\ref{eq:la})) of $F_{x, \widehat U}$ if we choose a frame where the origin is $\Pi_S(x)$ and $\R^d$ is $\Pi_S(Tan(x))$,  we see that for sufficiently small $\eps$, $\Pi_S(U_\eps \cap \MM)$ is the graph of a $C^2$ function. 
 Since $x \in \MM$ was arbitrary, this proves that $\Pi_S\MM$ is a submanifold of $\R^n$. Finally, we note that the reach of $\Pi_S\MM$ is given by 
\beq reach(\Pi_S \MM) & = &\inf_{\MM \ni x' \neq y' \in \MM}\frac{ |\Pi_S(x'-y')|^2}{2 dist(\Pi_S x', \Pi_S Tan(y'))}\\ & \geq & \inf_{\MM \ni x' \neq y' \in \MM}\frac{ (1 - 4 \a^2)|x'-y'|^2}{2 dist(\Pi_S x', \Pi_S Tan(y'))}\\
& \geq & (1 - 4\a^2) \inf_{\MM \ni x' \neq y' \in \MM}\frac{ |x'-y'|^2}{2 dist( x',  Tan(y'))}\\
& = & (1 - 4\a^2) reach(\MM). \eeq

\end{proof}

\subsection{Estimating $\sigma$.}

As we mentioned in the beginning of this section, we have assumed knowledge of $\sigma$.
However in an interesting regime, namely the regime where the values of $n, D, \tau$ and $\sigma$ satisfy  $$n  - D\geq \left(\frac{\tau^2D}{\sigma^2}\right),$$ we shall now show how an upper bound on $\sigma$ can be obtained which is good enough for our purposes. 
Let $S$ be the $D$ dimensional subspace defined in the statement of Proposition~\ref{lem:5-May}. Let 
$$\hat{\sigma} := \sqrt{\E_{z \sim \tilde{\mu}} \left(\frac{dist(z, S)^2}{n-D}\right)}.$$ It then follows from Proposition~\ref{lem:5-May} that with probability at least $1 - \de$, $$(n-D) \sigma^2 \leq (n- D)\hat{\sigma}^2 \leq \a^4 \tau^2 + (n - D) \sigma^2.$$
This implies that 
$$ \sigma^2 \leq \hat{\sigma}^2 \leq (\a^4 + 1)  \sigma^2.$$

The Monte-Carlo method can be used to estimate $\hat{\sigma}^2$ to within a prescribed error with high probability.
\section{Learning discs that approximate the data}\label{sec:follows}



\subsection{Properties of $\Pi_D X_0$.}\label{ssec:prop_proj}

In what follows, we shall, without loss of generality, identify $S$ constructed in Proposition~\ref{lem:5-May} with $\R^D \subseteq \R^n$. We will denote the orthogonal projection $\Pi_{\R^D}$ of $\R^n$ on to $\R^D$ by $\Pi_D$.

Let \beq r_p \in \left[\sqrt{\sigma \tau} D^{1/4}, \frac{\tau}{C{d^C}}\right].\eeq Here $r_p$ is a preliminary radius that will be subsequently be replaced by a smaller radius $r_c$.
Let $N_0$ be chosen to be an integer such that \beq N_0/\ln(N_0) > \frac{CV}{\rho_{min}\omega_d(r_p^2/\tau)^d},\eeq where $\omega_d$ is the volume of a Euclidean unit ball in $\R^d$.  We will assume that $D$ is large enough to that  we can choose $N_0$ such that \beq N_0 \leq e^D.\eeq

\begin{lemma}\label{lem:haus-May}
Let $\widetilde{X}_0$ be a set of $N_0$ i.i.d random samples from the distribution $\mu$. Let $X_0$ be a set of $N_0$ i.i.d samples from $\mu\ast G_\sigma^{(n)}$, obtained by adding i.i.d noise sampled from $G_\sigma^{(n)}$ to the points in $\widetilde{X}_0$.
With probability $1 - N_0^{-C}$, $\Pi_D X_0$ will be $C r_p^2/\tau-$close to $\Pi_D \MM$ in Hausdorff distance. 
\end{lemma}
\begin{proof}
 By the coupon collector problem applied to the Voronoi cells corresponding to a  $6r_p^2/\tau$ net of $\MM$ that is also $r_p^2/{2\tau}$ separated (such a net always exists, and can be constructed by a greedy procedure), we see that if we examine the set $\Pi_D \widetilde{X_0}$ of $N_0$ i.i.d random samples  from the push forward of $\mu$ under $\Pi_D$, with probability at least $1 - N_0^{-C}$, every  Voronoi cell has at least one random sample. Therefore the Hausdorff distance of $\Pi_D \widetilde{X_0}$ to $\Pi_D \MM$ is less than $12r_p^2/\tau$. Due to Gaussian concentration (see Lemma~\ref{lem:Gaussian}), the maximum distance of a point $\Pi_D y_i$ of $\Pi_D X_0$ to the corresponding $\Pi_D x_i$ in $\Pi_D \widetilde{X_0}$ is bounded above  by $$\sigma\left(\sqrt{D} + \sqrt{\ln(N_0^C)}\right) < C r_p^2/\tau,$$ with probability  at least $1 - N_0^{-C}$.  This is an upper bound on the Hausdorff distance between $\Pi_D X_0$ and $\Pi_D \widetilde{X_0}$.
Therefore, we have proved the lemma. 
\end{proof}

\subsection{Putative discs}
\label{subsec:algorithm-finddisc}

Let $X$ be a finite set of points in $E= \mathbb{R}^D$ and $X \cap B_1(x) := \{x, \widetilde{x}_1, \dots, \widetilde{x}_s\}$ be a set of points within a Hausdorff distance $\delta$ of some (unknown) unit $d$-dimensional disc $D_1(x)$ centered at $x$.  Here $B_1(x)$ is the set of points in ${\mathbb R}^D$ whose distance from $x$ is less or equal to $1$. We give below a simple algorithm that finds a unit $d$-disc (i.e. a ball in a $d-$ dimensional isometrically embedded Euclidean space) centered at $x$ within a Hausdorff distance $C{d}\delta$ of $X_0:= X \cap B_1(x)$, where $C$ is an absolute constant.

The basic idea is to choose a near orthonormal basis from $X_0$ where $x$ is taken to be the origin and let the span of this basis intersected with $B_1(x)$ be the desired disc.
This algorithm appeared previously in \cite{FIKLN} but has been included in the interest of readability.

\underline{Algorithm FindDisc:}
\begin{enumerate}
\item Let $x_1$ be a point that minimizes $ |1 - |x- x'||$ over all $x' \in X_0$.
\item Given $x_1, \dots x_m$ for $m \leq d-1$, choose $x_{m+1}$ such that $$\max(|1-|x- x'||, |\langle x_1/|x_1|, x'\rangle|, \dots, |\langle x_m/|x_m|, x'\rangle|)$$ is minimized among all $x' \in X_0$ for $x'= x_{m+1}$.
    \end{enumerate}
Let $\widetilde{A}_x$  be the affine $d$-dimensional subspace containing $x, x_1, \dots, x_n$, and the unit $d$-disc $\widetilde{D}_1(x)$ be $\widetilde{A}_x \cap B_1(x)$. Recall that for two subsets $A, B$ of $\R^D$, $d_H(A, B)$ represents the Hausdorff distance between the sets. We  will denote large absolute constants by $C$ and small absolute constants by $c$.
\begin{lemma}\lab{lem:find-disc}
Suppose there exists a $d$-dimensional affine subspace $A_x$ containing $x$ such that $D_1(x) = A_x \cap B_1(x)$ satisfies $d_H(X_0, D_1(x)) \leq \delta$.
Suppose $0 < \delta < \frac{1}{2d}$. Then $d_H(X_0, \widetilde{D}_1(x)) \leq  C{d}\delta$, where $C$ is an absolute constant.
\end{lemma}
The proof is in Section~\ref{sec:subsidiary} of the Appendix.
\subsection{Fine-tuning the discs}\label{ssec:fine}
In this subsection, we will assume that the discs are centered at the origin and have a radius $1$. 
Recall that $B_1$ is a ball centered at the origin with radius $1.$
We also assume that
we have constructed a disc $\widetilde{D}_1$ such that $d_H(\widetilde{D}_1, B_1 \cap X)  < c$ and we know that there exists some disc $D_1$ (which we have not constructed) such that  
\ben 
\item $d_H(D_1, B_1 \cap X) = \de_1 < c$ is the Hausdorff distance between a unit disc $D_1$ and $B_1 \cap X$. 
\item $\de_2 := \sup\limits_{z \in B_1 \cap X}\dist(z, D_1).$
\een

We will describe an algorithm that produces a disc $\widehat{D}$ such that 
\ben 
\item $\sup\limits_{z \in B_1 \cap X}\dist(z, \widehat{D}) \leq 2 \de_2.$
\item $d_H(\widehat{D}, D_1) \leq 12 \de_2.$
\een  

This algorithm will be applied to $X = \Pi_D X_0$, from Lemma~\ref{lem:haus-May}. Note that in this context, $\de_2 < \frac{Cr^2}{\tau}$, while we only need $\de_1 < cr.$
Consider the family $\D$ of all unit $d-$dimensional discs in  $E$ centered at the origin defined by $$\D := \{D|(\sup_{z \in B_1 \cap X} dist(z, D) \leq 2 \de_2)\text{ and }(d_H(B_1\cap X, D) \leq 2 d_H(\widetilde{D}_1, B_1 \cap X))  \}.$$ We will obtain an estimate $\overline \de$ of the diameter of this family in the Hausdorff metric.
 By the triangle inequality applied to the Hausdorff metric, there exist $D_2 \in \D$ and $x\in \partial D_1$ such that $$dist(x, D_2)= d_H(D_1, D_2) \in [\overline\de/2, \overline\de].$$
Let $\phi: Tan(0, D_1) \ra Nor(0, D_1)$ be a linear map from the tangent space of $D_1$ at the origin to the normal space at $D_1$ at the origin such that
$D_2 \subseteq \{(z, \phi(z))|z \in D_1\}$. Then, $\phi$ is $\overline \de/\sqrt{1 - \overline \de^2}-$Lipschitz   (measured with respect to the Euclidean metric on the range and domain of $\phi$). 

Let $y\in X\cap B_1$ be such that $|y-x| \leq 2 d_H(\widetilde{D}_1, B_1 \cap X)$. Since $2 d_H(\widetilde{D}_1, B_1 \cap X) < c$, we have $$ |x - \Pi_{D_1} y|< 2 d_H(\widetilde{D}_1, B_1 \cap X) < c,$$ implying that
$$dist(\Pi_{D_1} y, D_2) \geq \overline \de/4.$$
Also $|y - \Pi_{D_1} y| \leq \de_2.$
Therefore, $$2\de_2 \geq \dist( y, D_2) \geq \overline \de/4 -  \de_2 .$$ Therefore, $\overline\de/4 \leq 3 \de_2.$ Thus,
\beq \label{bar_de}\overline \de \leq 12 \de_2 < c.\eeq

Therefore, by the preceding bound on the Hausdorff diameter and  $D_1 \in \D$, 
it suffices to find any one element of $\D$ to obtain an additive approximation $\widehat{D}$ to $D_1$ to within $12 \de_2$ in the Hausdorff metric which also satisfies   $\sup\limits_{z \in B_1 \cap X}\dist(z, \widehat{D}) \leq 2 \de_2.$    .
This is done as follows. 
Let the linear span of $\widetilde D_1$ be identified with $\R^d$ with the canonical Euclidean metric. Then, we consider the set $\widetilde \D$ of discs in $E$ of the form
\beq \label{DfromA}\widetilde{D}_A = B_1 \cap \{(z'_1, z'_2)| Az'_1 - I z'_2 = 0\},\eeq where $A$ is a $d \times (D-d)$ matrix, such that the operator norm of $A$ is bounded above as follows:
$$\|A\| \leq 2 d_H(\widetilde{D}_1, B_1 \cap X).$$
By (\ref{bar_de}), $\widetilde \D \supseteq \D$. It suffices to find one element of $\D$ in $\widetilde\D$. We do this by solving the following convex program using Vaidya's algorthm \cite{Vaidya}:\\
Find $A$ such that 
\ben
\item $\|A\|_2 \leq 2d_H(\widetilde{D}_1, B_1 \cap X),$ and
\item  for all $x = (x_1, x_2) \in B_1 \cap X$, where $x_1 \in \R^d$ and $x_2 \in \R^{D-d}$, $$\|Ax_1 - I x_2\|_2 \leq 2 \de_2.$$
\een

Once such an $A$ is found, the corresponding $\widehat D$ given by (\ref{DfromA}) satisfies the following. For any $(x_1, x_2) = x \in B_1 \cap X$, 
\beq dist(x, \widehat D) & = &  \bigg\|(I + A A^T)^{-1/2}(Ax_1 - I x_2)\bigg\|_2\leq  2 \de_2.\eeq


\section{Obtaining a refined net of $\MM$}\label{sec:ref-net}

We recall the setup. $\MM$ is a submanifold of $\R^n$. $\R^D \subseteq \R^n$ is a coordinate subspace such that the orthogonal projection $\Pi_D$ of $\R^n$ onto $\R^D$ satisfies 
$$\sup_{x \in \MM}|x - \Pi_D x| \leq \alpha^2 \tau,$$ where we choose $\a^2$ to be  less than  $\frac{1}{Cd}$.
Let the set of all points within a distance of $2\sqrt{D}\sigma$ of $\Pi_D \MM$, be denoted $(\Pi_D \MM)_{2 \sqrt{D}\sigma}$. Let $x \in (\Pi_D \MM)_{2 \sqrt{D}\sigma}$.  Recall that $X_0$ is a set of $N_0$ independent, randomly sampled points from $\mu \ast G_\sigma^{(n)}$. Let $r_c$ satisfy \beq\frac{r_c^2}{\tau} = \frac{\tau}{Cd^C} > 4\sigma  D^{\frac{1}{2}}.\eeq 
Let $B_D(x, r_c)$ denote the Euclidean $D-$dimensional ball of radius $r_c$,  centered at $x$, contained in $\R^D$. 
 Let $x$ be an arbitrary point in $(\Pi_D \MM)_{2 \sqrt{D}\sigma}$. Let $D_0$ be a $d-$dimensional disc centered at $x$, having radius $r_c$ and having a Hausdorff distance to $\Pi_D(X_0)\cap B_D(x, r_c)$ that is less than $\de = \frac{Cr_c^2}{\tau},$ that is
\beq\label{eq:haus2}\dhaus(D_0,  \Pi_D(X_0)\cap B_D(x , r_c)) < \de = \frac{Cr_c^2}{\tau}.\eeq
Disc $D_0$  can be obtained using the algorithm in the Section~\ref{sec:follows}.
\begin{lemma}\label{lem:11-May}  The disc $D_0 \subseteq E$ satisfies
$$\p\left[\dhaus(D_0, B_D(x, r_c) \cap \Pi_D(\MM)) \leq 2\de \right] \geq 1 - cN_0^{-C}.$$
\end{lemma}
\begin{proof}

Since $ \sigma (\sqrt{D} + \sqrt{\ln (N_0^C)})$ is less than $\de = Cr_c^2/\tau < \frac{\tau}{Cd^C}$, by Gaussian concentration (see Lemma~\ref{lem:Gaussian}), with probability at least $ 1 - \frac{N_0^{-C}}{C}$, every point $y_i = x_i + \zeta_i$ satisfies 
$$|y_i - x_i| = |\zeta_i| <  \sigma (\sqrt{D} + \sqrt{\ln (N_0^C)}) < Cr_c^2/\tau.$$ 
By replacing $r_p$ by $r_c$ in Lemma~\ref{lem:haus-May}, we see that 
$$\p\left[\dhaus(\Pi_DX_0,  \Pi_D(\MM)) \leq \de/2 \right] \geq 1 - N_0^{-C}.$$ Therefore, in particular,
$$\p\left[\dhaus(\Pi_DX_0 \cap B_D(x, r_c),  \Pi_D(\MM)\cap B_D(x, r_c)) \leq \de \right] \geq 1 - N_0^{-C}.$$
By the triangle inequality with respect to $\dhaus,$
\beqs \dhaus(D_0, B_D(x, r_c) \cap \Pi_D(\MM)) & \leq &  \dhaus(D_0,  \Pi_D(X_0)\cap B_D(x , r_c))\\ & + & \dhaus(\Pi_D(X_0)\cap B_D(x , r_c), B_D(x, r_c) \cap \Pi_D(\MM)).\eeqs The Lemma now follows from the last two inequalities and (\ref{eq:haus2}).
\end{proof}

\subsection{Using discs to approximate $\MM$ at a coarse scale}
Let $X_2 = \{q_i\}$ be a minimal $cr_c/d-$net of $\Pi_D X_0$.  Such a net can be chosen greedily, ensuring at every step that no element included in the net thus far is within $\frac{cr_c}{2d}$ of the point currently chosen. The process continues while progress is possible. Let the size of $X_2$ be denoted $N_2$.

We introduce a family of $D$ dimensional balls of radius $r ,$  $\{U_i\}_{i \in [N_2]} \subseteq \R^D$ where the center of $U_i$ is $q_i$ and a family of $d-$dimensional embedded discs of radius $r_c$, denoted $\{D_i\}_{i \in [N_2]}$, $D_i \subseteq U_i$ where $D_i$ is centered at $q_i$. 
The $D_i$ are chosen by fitting a disc that approximately minimizes among all discs of radius $r_c$ centered at $q_i$  the Hausdorff distance to $U_i \cap X_0$ by a procedure described in Subsection~\ref{subsec:algorithm-finddisc}. 

We note that the following properties of  $(D_i, q_i)$, hold with probability at least $1 - N_0^{-C}$. Property (C1) is an immediate consequence of Lemma~\ref{lem:11-May}.
\ben
\item[(C1)] The Hausdorff distance between $\cup_i D_i$ and $ \Pi_D \MM$ is less than $\frac{Cr_c^2}{ \tau} = \de$.
\item[(C2)] For any $i \neq j$, $|q_i - q_j| > \frac{cr_c}{d}$.
\item[(C3)] For every $z \in  \MM$, there exists a point $q_i$ such that $|z - q_i| <  3 \inf_{i \neq j}, |q_i - q_j|.$ 
\een

Now fix one disc that we shall relabel to call $D_0$. \\We use a new coordinate system in which the center of $D_0$ is the origin. We denote by $\R^d$, the affine span of $D_0$.
Let $\Pi_d:  \R^n \ra \R^d$
 denote the map that projects a point  in $\R^n$ orthogonally on to $\R^d$. Let $S_0$ denote the cylindrical set given by \beq\label{eq:S0} S_0 := B_d(0, \frac{r_c}{2})\times B_{D-d}(0, \frac{r_c}{2}) \subset \R^D,\eeq where $B_d(0, \frac{r_c}{2})$ is the ball of radius $\frac{r_c}{2}$ contained in $\R^d$ centered at the origin, and $B_{D-d}(0, \frac{r_c}{2})$ is the ball of radius $\frac{r_c}{2}$ contained in $\R^{D-d}$ centered at the origin. Here $\R^{D-d}$ is the orthogonal compliment of $\R^d$ inside $\R^D$. Note that \beqs \Pi_D^{-1} S_0 \subseteq \R^n. \eeqs We will obtain a refined net of $\MM \cap \Pi_D^{-1} S_0$.
\begin{figure}
\centering
\includegraphics[scale=0.60]{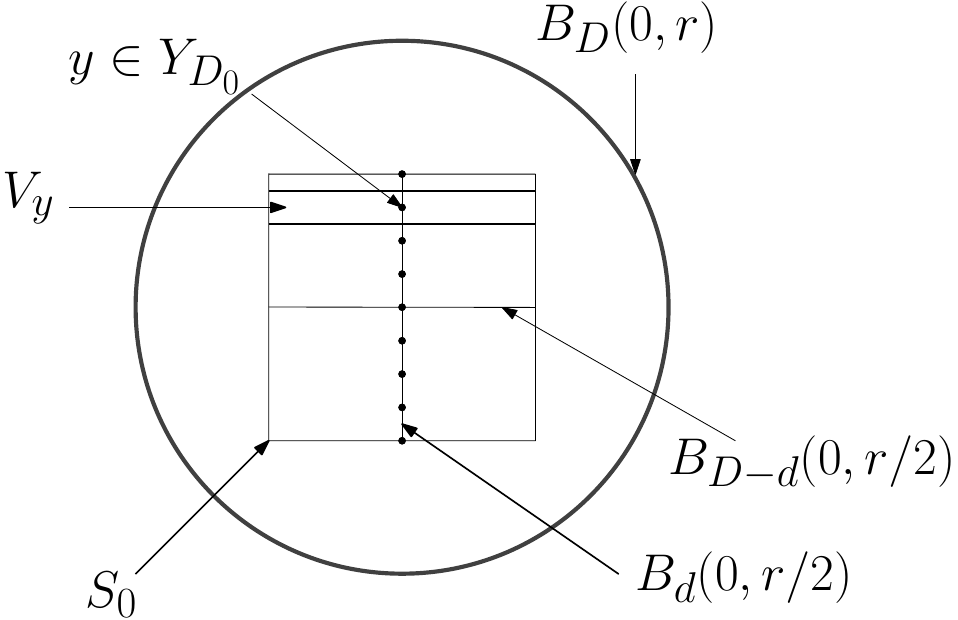}
\caption{The $r$ in the figure is set to $r_c$ in this section.}
\label{fig:H0}
\end{figure}
Let $Y_{D_0} = B_d(0, \frac{r_c}{2}) \cap 10 \sigma \Z^d$. 
 Due to volumetric considerations, \beq\label{eq:103-May}  |Y_{D_i}| \leq \left(\frac{r_c}{\sigma}\right)^d.\eeq

\begin{definition}
 For $y \in Y_{D_0}$, let $V_y$ denote the ``truncated Voronoi cell" defined by 
\beq V_y = \{y'\in S_0|\forall y'' \in Y_{D_0}, |y - y'| \leq |y' - y''|.\} \eeq 
\end{definition}
Thus, $V_y$ is the set of all points in $S_0$ 
that are at least as close to $y$ as they are to any other member of $Y_{D_0}.$ 


The points $q_i = p_i + \gamma_i$ of $X_0$ are i.i.d copies of the random variable $q = p + \gamma$, which has the form  $q_i = p_i + \gamma_i$, where $p_i$ has the same distribution as $p$, which is a sample from $\mu$ supported on $\MM$ and the $\gamma_i$ has the same distribution as $\gamma$ which is an i.i.d sample from the measure whose density is $G_\sigma^{(n)}$. 
Define the net $$Z_{S_0} := \{z_y|y \in Y_{D_0}\},$$ where \beq \lab{eq:formula_B} z_y = \E[q| q \in \Pi_D^{-1} V_y].\eeq
Let us first focus on $\Pi_D z_y$, and its distance to $\Pi_D \MM$. The push-forward of $G_\sigma^{(n)}$ under $\Pi_D$ is a Gaussian density $G_\sigma^{(D)}$ restricted to $\R^D$, the measure corresponding to which is 
$N_D(0, \sigma^2)$. Further $\Pi_D \MM$ has already been shown to be of reach at least $ (1- 4\a^2)\tau$ in Lemma~\ref{lem:alpha}. Secondly, by (\ref{eq:feb17-81}),  $\Pi_D$ applied to a unit tangent vector $v$ at a point on $\MM$ satisfies $$|\|\Pi_D v\| -1| \leq  \|(\Pi_D v) - v \| \leq 1 - \sqrt{1 - 4 \a^2} \leq 4\a^2 \leq \frac{1}{4d}.$$ Due to the consequences of this on the Jacobian of the restriction of $\Pi_D$ on the tangent space at a point on $\MM$, the push forward of $\mu$ under $\Pi_D$ has a Radon-Nikodym derivative with respect to the Hausdorff measure on $\Pi_D\MM$ that takes values in $[\frac{\rho_{min}}{2}, 2\rho_{max}].$

 {\it In this section, we will use the fact that that the logarithm of the Radon-Nikodym derivative of $\mu$ with respect to the Hausdorff measure on $\MM$ is $\frac{C}{\tau}-$Lipschitz.}
\begin{remark}\lab{rem:5.2}
As a consequence of Lemma~\ref{lem:6}, $\Pi_D(\MM) \cap \frac{\tau  S_0}{10 r_c}$ is the graph of a $C^2-$function $f$ from $B_d(0, \frac{\tau}{20})$ to the orthogonal compliment of $\R^d$ in $\R^D$, (which we henceforth denote by $\R^{D - d}$). This function $f$, when restricted to $B_d(0, \frac{r_c}{2})$ has a $C^0-$norm of at most $C\de$ with probability $ 1 - N_0^{-C}$ by Lemma~\ref{lem:11-May}. 
\end{remark}
Thus, we henceforth assume
\beq \label{eq:de0} \forall x \in B_d(0, r), | f(x)| \leq C\de. \eeq

The bound on the reach together with Lemma~\ref{lem:6} implies the following 
{ for all } $x \in B_d(0, \frac{\tau}{20}).$
Firstly, 
$$\forall_{v \in \R^d} \forall_{w \in \R^{n-d}} \langle \partial_v^2 f(x), w\rangle \leq 
\frac{C|v|^2|w|}{\tau}.$$ Secondly,
\beq\label{eq:230-Mar6}   \forall v \in \R^d,  |\partial_v f(x)| \leq  \frac{C\de|v|}{r_c} + \frac{C |v||x|}{\tau}. \eeq
This implies that
 $\forall x,y \in B_d(0, \frac{\tau}{20}),$ and $\forall w \in B_{n-d}(0, 1)$, denoting $x-y$ by $v,$
\beq\label{eq:231-Mar6}    | \langle f(x) - f(y) - \partial_v f(y), w\rangle | \leq   \frac{C |v|^2}{\tau}. \eeq
\begin{notation}\lab{not:X02} Let $X_0^{(2)}$ be a set of \beqs N_0^{(2)} & := & C^d n \left(\frac{\sigma}{\sigma^2/\tau}\right)^2N_0\max_i |Y_{D_i}| \log  N_0 \log\frac{N_0 r_c^d}{\sigma^d}\\ & \leq & C^d n  \left(\frac{\sigma}{\sigma^2/\tau}\right)^2 N_0\left(\frac{r_c}{\sigma}\right)^d  \log N_0 \log\frac{N_0 r_c^d}{\sigma^d} \eeqs independent random samples from $\mu$ (which are hence independent to $X_0$). 
\end{notation}
Let the $\#_y$  be defined by  \beq \lab{eq:form-A} \#_y := |V_y \cap \Pi_D X_0^{(2)}|.\eeq We will be interested in the case where $$\#_y > 100 \left(\frac{n\tau^2}{\sigma^2}\right) \log N_0,$$ since the complimentary event will be absorbed in the error probability.

\begin{lemma}\label{lem:14-May28} Let the number $\#_y$ of points, (see (\ref{eq:form-A})) satisfy
\beq \p\left[\left(\min_{D_i}\min_{y \in Y_{D_i}} \#_y\right) \leq 100 \left(\frac{n\tau^2}{\sigma^2}\right) \log N_0\right] <  N_0^{-100}. \eeq
\end{lemma}
\begin{proof}
By (\ref{eq:230-Mar6}), $\mu(\Pi_D^{-1}(V_y) \cap\MM) > \frac{\rho_{min}V}{N_0\max_i |Y_{D_i}|} > \frac{c \sigma^d}{r_c^d N_0}$. If $x \in \Pi_D^{-1}(V_y) \cap\MM$, and $\zeta$ is a random sample from the measure associated with $G_\sigma^{(n)}$, then $$\p[x + \zeta \in V_y]\geq c^d,$$ since $V_y \cap \R^d$ is a cube of side length more than $c\sigma.$ Therefore, if $x$ is a random sample from $\mu$ and $\zeta$ is a random sample from the measure associated with 
$G_\sigma^{(n)}$ and $z = x + \zeta$, then $\p[ z \in \Pi_D^{-1}V_y] > \frac{c^d \sigma^d}{r_c^d N_0}.$ The lemma follows from the solution to the coupon collector problem with $N_0\max_i |Y_{D_i}| < \frac{N_0 r_c^d}{\sigma^d}$ bins, corresponding to the different $V_y$.
\end{proof}
Recall from (\ref{eq:S0}) that $S_0 = B_d(0, \frac{r_c}{2})\times B_{D-d}(0, \frac{r_c}{2}).$

We split $ \Pi_D^{-1}(V_y \cap \Pi_D X_0^{(2)})$ (see Notation~\ref{not:X02}) into two sets.
Let $\mathfrak{X}_1$ consist of those points $p_i + \gamma_i$ such that $\Pi_D(p_i + \gamma_i) \in V_y$ and $\Pi_D (p_i) \in  \frac{\tau  S_0}{10 r_c}, $ where we recall that $p_i \in \MM$. Let $\mathfrak{X}_2$ consist of those points $p_i + \gamma_i$ such that $\Pi_D(p_i + \gamma_i)\in V_y $ and $\Pi_D (p_i) \not\in  \frac{\tau  S_0}{10 r_c}.$ Note that the number of points in $\mathfrak{X}_1 \cup \mathfrak{X}_2$ is $\#_y$.
Let 
\begin{eqnarray}\lab{eq:y1}
    \widehat{z}_{y,1} :=
    \begin{cases}
      \frac 1{\#_y}{\sum\limits_{z \in \mathfrak{X}_1} z},& \text{ if } \#_y \geq 1,  \\
      y,& \text{ if } \#_y = 0. \\
    \end{cases}
  \end{eqnarray} Let us define the following conditional expectation.
\beq \lab{eq:C1} z_{y,1} := \E\left[\widehat{z}_{y,1}| \{D_i\}_{i \in [N_2]}, \Pi_D\right]. \eeq 
Using a measure theoretic notation, the point $z_{y, 1}$ is the conditional expectation $$\E\left[\widehat{z}_{y,1}| \HH\right],$$ where $\HH$ is the $\sigma-$algebra generated by the random variables in the set $\{\{D_i\}_{i \in [N_2]}, \Pi_D\}.$

Let 
\beq \lab{eq:y2}
    \widehat{z}_{y,2} :=
    \begin{cases}
      \frac 1{\#_y}{\sum\limits_{z \in \mathfrak{X}_2} z},& \text{ if } \#_y \geq 1,  \\
      y,& \text{ if } \#_y = 0. \\
    \end{cases}
  \eeq
 Again, let us define a conditional expectation analogous to (\ref{eq:C1}).
Let \beq\lab{eq:C2} z_{y,2}:=
\E\left[\widehat{z}_{y,2}| \{D_i\}_{i \in [N_2]}, \Pi_D\right].\eeq
Let \beq\lab{eq:form-B} \widehat{z}_y :=  \widehat{z}_{y,1} + \widehat{z}_{y, 2}. \eeq
Assuming that $|\mathfrak{X}_1| \geq 1,$ and $|\mathfrak{X}_2| \geq 1,$ the above expression equals
$$\frac{1}{\#_y} \sum_{z \in \mathfrak{X}_1\cup\mathfrak{X}_2} z.$$
Let $\mu'$ be the  pushforward under $\Pi_d$ of the restriction of $\mu$ to $\Pi_D^{-1}\left(\frac{\tau  S_0}{10 r_c}\right)$, normalized to be a probability measure.   Thus $\mu'$ is the push forward of a measure (derived from $\mu$) supported in $\MM$. We will also need to work with an analogous  measure derived from $\mu \ast G_\sigma^{(n)},$ which we denote by $\mu''$ below.

Let $x$ be a random variable, whose distribution is such that $\p[x \in A] = \mu'(A)$ for every Borel set $A \subseteq B_d(0, \frac{\tau}{20})$. Let $f$ be the function referred to in Remark~\ref{rem:5.2}. Let $\Pi_{D-d}$ denote the orthogonal projection of $\R^n$ onto $\R^{D-d}$, which we define to be the orthocomplement of $\R^d$ in $\R^D$.
Let $\mu''$ be the probability measure supported on $B_d(0, \frac{\tau}{20})$ given by  the following conditional expectation:
\beqs \mu''(A) & := & \p[x \in A|(x' + x+ f(x) \in V_y)],\eeqs where $ x + f(x)$ is a sample from the pushforward of $\mu$ via $\Pi_D$ and $x'$ is the image via $\Pi_D$ of an independent sample from $G_\sigma^{(n)}$.
Thus \beqs \mu''(A) & = & \frac{\p[(x \in A) \land (x' + x+ f(x) \in V_y)]}{\p[(x' + x+ f(x) \in V_y)]}\\
& = & \frac{\p[(x \in A) \land (\Pi_d x' + x  \in \Pi_d V_y)\land (\Pi_{D-d} x' + f(x)\in \Pi_{D-d} V_y)]}{\p[(\Pi_d x' + x \in \Pi_d V_y)\land (\Pi_{D-d} x' + f(x)\in \Pi_{D-d} V_y)]}.\\
\eeqs
Let $\gamma_d$ and $\gamma_{D-d}$ denote Gaussian measures in $\R^d$ and $\R^{D-d}$ having covariances $\sigma^2 I_d$ and $\sigma^2 I_{D-d}$ respectively.
Note that the denominator in the above expression $$ \p[(\Pi_d x' + x \in \Pi_d V_y)\land (\Pi_{D-d} x' + f(x)\in \Pi_{D-d} V_y)]$$ equals  \beq \int_{\R^d}   \gamma_d(\Pi_d V_y - x) \gamma_{D-d}(\Pi_{D-d}V_y - f(x))\mu'(dx)
 =: \Gamma. \eeq

Let $B_d^\infty(x, R)$ denote $\ell_\infty$ ball  in $\R^d$ whose center is $x$ and side length is $2R.$
Then, the Radon-Nikodym derivative $\frac{d\mu''}{d\mu'}$ at $x \in B_d(0, \frac{\tau}{20})$ is given by 
\beq\label{eq:113-oct3-2020}
\Gamma^{-1} \gamma_d(\Pi_d V_y - x) \gamma_{D-d}(\Pi_{D-d}V_y - f(x)).
\eeq
{ Moreover,}
\beq\label{eq:114-oct3-2020} \gamma_d(\Pi_d V_y - x) = \int_{B_d^\infty(y, 5\sqrt{d}\sigma)}(\sqrt{2\pi}\sigma)^{-d}\exp(-\frac{(|x - z|)^2}{2 \sigma^2})\la_d(dz).
\eeq

Assuming the $D-d$ is larger than a sufficiently large universal constant $C$, we see by Gaussian concentration (Lemma~\ref{lem:Gaussian}), (\ref{eq:de0}) and the fact that $\Pi_{D-d}V_y \subseteq B_{D-d}(0, \frac{r}{2})$ that if $x \in B_d(y, \frac{r_c}{2}),$ then $f(x) < C\de < \frac{r_c}{6},$ implying that
\beq\label{eq:Den-Mar6} \gamma_{D-d}(\Pi_{D-d}V_y - f(x)) > 1 - \frac{\sigma^2}{\tau}.\eeq
It follows that $\Gamma^{-1} \leq C \la_d(B_d^\infty(y, 5\sqrt{d}\sigma))$.
We observe that due to the symmetries of the lattice $10\sigma \Z^d$, provided that
$ B_d^\infty(y, 5\sqrt{d}\sigma) \subseteq S_0,$
\beq \gamma_d(\Pi_d V_y - x) = \gamma_d(\Pi_d V_y - (2y - x)).\eeq

Let us first analyze $|z_{y,2}|.$
Since $V_y$ is contained in $S_0$, this is clearly in the interval $[- \frac{r_c}{2}, \frac{r_c}{2}]$.
We see that 
\beqs \big|z_{y,2}\big| & \leq & \frac{r_c}{2}\left((\sqrt{2\pi}\sigma)^{-D}\exp(- \frac{c\tau^2}{2\sigma^2})\rho_{max}\la_D(V_y)\right)
 \leq  \left(\frac{\sigma^2}{\tau}\right),\eeqs
by Gaussian concentration (Lemma~\ref{lem:Gaussian}), since $\frac{\tau}{Cd^C} > \sigma \sqrt{D}.$ 

Next, let us analyze $$e_y :=  \Pi_{D-d} (z_{y,1} + z_{y, 2}),$$ where $z_{y, 1}$ and $z_{y, 2}$ are defined in (\ref{eq:C1}) and (\ref{eq:C2}) respectively.
\subsection{Controlling the distance of $e_y$ to $\Pi_D \MM$}

Let $\widehat{G}_{D-d}^x$ be a random variable, whose density $\widehat{\gamma}_{D-d}^x$ at a point $z \in \R^{D-d}$ is given by $$\widehat{\gamma}_{D-d}^x(z) := \frac{{\gamma}_{D-d}(z)\I(\Pi_{D-d} V_y - f(x))}{\int\limits_{\Pi_{D-d} V_y - f(x)} \gamma_{D-d} (w)dw},$$ where $\I(\cdot)$ is the indicator function.

where $${\gamma}_{D-d}(z) = (2\pi\sigma^2)^{-\frac{D-d}{2}}\exp\left(-\frac{\|z\|^2}{\sigma^2}\right)$$ is the density of a mean $0$ Gaussian of covariance $\sigma^2I_{D-d}$. 

Let $g:B_d(0, \frac{\tau}{20}) \ra \R^{D-d}$ be given by $$g(x) := \E( \widehat{G}^{x}_{D-d}).$$
Then,
\beqs \Pi_{D-d} (z_{y, 1}) =  \int_{x \in B_d(0, \frac{\tau}{20})} (f(x) + g(x))  \mu''(dx),\eeqs and so 
\beq e_y = \Pi_{D-d} z_{y, 2}+ \int_{x \in B_d(0, \frac{\tau}{20})} (f(x) + g(x))  \mu''(dx). \eeq

By (\ref{eq:de0}), for $x \in B_d(0, r),$ we have
$ |f(x)|  \leq   C \de.$
Recall that $$\frac{r_c^2}{\tau}  = \frac{\de}{C}> 4  \sigma \sqrt{D}.$$

\begin{lemma}\lab{lem:5.3-diff-eq}
If $x \in B_d(0,  r_c)$, then $|g(x)| < \frac{\sigma^2}{\tau}.$
\end{lemma}
The proof is in Section~\ref{sec:subsidiary} of the Appendix.

The following lemma shows that $\dist(e_y, \Pi_D \MM) < \frac{Cd\sigma^2}{\tau}$. 
\begin{lemma}\label{lem:11-May27}The quantity $e_y$ satisfies 
\beqs \big|\Pi_{D-d} e_y -f(y) \big| < \frac{C d\sigma^2}{\tau}. \eeqs
\end{lemma}
The proof is in Section~\ref{sec:subsidiary} of the Appendix.

We define the ``refined net" of $\Pi_D \MM \cap B_D(0, \frac{r_c}{2})$ to be
$\{e_y | y\in Y_{D_0}\}$. To extend this to a net of $\Pi_D \MM$, we  take the union over all such refined nets corresponding to balls of radius $r_c$ with centers in $X_1$. 

\subsubsection{The $\psi_2$ norm}
A random variable $Z$ in $\R$ that satisfies for some positive real $K$, $$\E[\exp(|Z|^2/K^2)] \leq 2$$ is called subgaussian. 
\begin{definition}[$\psi_2$ norm]
We define $$\|Z\|_{\psi_2} = \inf\{ t > 0: \E[\exp(|Z|^2/t^2)] \leq 2\}.$$ 
\end{definition}
That this corresponds to a norm is known, see for example, Exercise 2.5.7 of \cite{vershynin_2018}.
For a subgaussian random variable, by Proposition 2.5.2 of \cite{vershynin_2018}, 
\beq \label{eq:windup}\p[Z \geq t] \leq 2\exp\left(\frac{-ct^2}{\|Z\|_{\psi_2}^2}\right),\eeq for all $t \geq 0$.


We appeal to Theorem 3.1.1 from \cite{vershynin_2018}, which implies the following.
\begin{proposition}\label{thm:versh}
Let $Z  = (Z_1, \dots, Z_n) \in \R^n$ be a random vector with independent Gaussian coordinates $Z_i$ that satisfy $\E Z_i^2 = \sigma^2$. Then 
$$\left \||Z| - \sqrt{n}\sigma \right\|_{\psi_2} \leq  C\sigma,$$ where  $C$ is an absolute constant.
\end{proposition}

\subsection{Controlling the distance of a point in the refined net to $\MM$} 
We now proceed to obtain a refined net of $\MM$.
We split $ (V_y \oplus \R^{n-D})\cap  X_0^{(2)}$ into two sets. 
Recall from (\ref{eq:S0}) that $S_0 = B_d(0, \frac{r_c}{2})\times B_{D-d}(0, \frac{r_c}{2}).$

Let $\mathfrak{X}^n_1$ consist of those points $p_i + \gamma_i$ such that $p_i + \gamma_i\in (V_y \oplus \R^{n-D}) \subset \R^n$ and $\Pi_D p_i \in  \frac{\tau  S_0}{10 r_c}.$

\begin{eqnarray}\lab{eq:y1n}
    \widehat{z}^n_{y,1} :=
    \begin{cases}
      \frac 1{\#_y}{\sum\limits_{z \in \mathfrak{X}_1^n} z},& \text{ if } \#_y \geq 1,  \\
      y,& \text{ if } \#_y = 0. \\
    \end{cases}
  \end{eqnarray}
Let us define a conditional expectation analogous to (\ref{eq:C1}), \beq\lab{eq:C1n} z^n_{y,1}:=
\E\left[\widehat{z}^n_{y,1}| \{D_i\}_{i \in [N_2]}, \Pi_D\right].\eeq

Let $\mathfrak{X}^n_2$ consist of those points $p_i + \gamma_i$ such that $\Pi_D(p_i + \gamma_i)\in V_y $ and $\Pi_D (p_i) \not\in  \frac{\tau  S_0}{10 r_c}.$ 

\begin{eqnarray}\lab{eq:y2n}
    \widehat{z}^n_{y,2} :=
    \begin{cases}
      \frac 1{\#_y}{\sum\limits_{z \in \mathfrak{X}_2^n} z},& \text{ if } \#_y \geq 1,  \\
      y,& \text{ if } \#_y = 0. \\
    \end{cases}
  \end{eqnarray}
   Again, analogous to (\ref{eq:C1}) we define,  \beq\lab{eq:C2n} z^n_{y,2}:=
\E\left[\widehat{z}^n_{y,2}| \{D_i\}_{i \in [N_2]}, \Pi_D\right].\eeq

Let $$\widehat{z}_{y}^n := \widehat{z}^n_{y, 1} + \widehat{z}^n_{y, 2}. $$


\begin{definition}\label{def:rnet}
We define the (random) net $\mathrm{Rnet}_0$ of $\MM \cap (S_0 + \R^{n-D})$ by
\beq\lab{eq:96new}\mathrm{Rnet}_0 := \{ \Pi_D \widehat{z}_{y'}  + \Pi_{n-D}\widehat{z}^n_{y'}  | y'\in Y_{D_0}\}.\eeq We analogously define the nets $\mathrm{Rnet}_i$, corresponding to disc $D_i$ as $i$ ranges over $[N_2]$.
\end{definition}
As a matter of fact, $$\Pi_D \widehat{z}_{y'} =  \widehat{z}_{y'},$$ for $y' \in Y_{D_i}$, where $i \in [N_2]$, but we have chosen to write (\ref{eq:96new}) thus, to emphasize that $\widehat{z}_{y'} \in \R^D.$
\begin{definition}\label{def:Rnet}
We  extend this to a net $\mathrm{Rnet}$ of $ \MM$, by taking the union over all nets
$\mathrm{Rnet}_i$. 
\end{definition}
We observe that $\Pi_{n-D} \widehat{z}^n_y$ can be expressed as the sum of two random variables, one that is a sample $\zeta_{y,1}^n$ from the push-forward of $\mu$ via the orthogonal projection onto $\R^{n-D}$ and another that is an independent  Gaussian $\zeta_{y, 2}^n$ belonging to $\R^{n-D}$ having the distribution $N_{n-D}(0, \sigma^2)$. Due to the bound on the distance of any point on $\MM$ to $\R^D$, we see that conditional on the correctness of the Principal Component Analysis step in Proposition~\ref{lem:5-May}, $|\zeta_{y, 1}^n| < \frac{\tau}{Cd}.$

Let us first analyze $ |z^n_{y,2}|.$
We see that there is a distribution $\gamma''_\tau$ supported on $[c\tau, \infty)$, such that
\beq\label{eq:251-Mar8.1} \big|z^n_{y,2}\big|  & \leq & \int_\R \tau\left((\sqrt{2\pi}\sigma)^{-D}\exp(- \frac{c\zeta^2}{2\sigma^2})\la_D(S_0)\right)\gamma''_\tau(d\zeta)
\leq  \left(\frac{\sigma^2}{\tau}\right).\eeq
The second inequality above is by virtue of the fact that the support of $\gamma''_\tau$ does not contain $(-c\tau, c\tau)$.

Let us next analyze $$\ey := e_y +  \Pi_{n-D}(z^n_{y,1} + z^n_{y, 2}).$$ 

Note that \beq \lab{eq:94.5}\E[\Pi_D \widehat{z}_{y} + \Pi_{n-D}\widehat{z}^n_{y}| \{D_i\}_{i \in [N_2]}, \Pi_D] = \ey.\eeq
As a consequence of Proposition~\ref{lem:5-May}, $\MM \cap \Pi_D^{-1}\left(\frac{\tau  S_0}{10 r_c}\right)$ is the graph of a $C^2-$function $f^{n-D}$ from $B_d(0, \frac{\tau}{20})$ to the orthogonal compliment of $\R^D$ in $\R^n$, (which we henceforth denote by $\R^{n - D}$). 
\begin{definition}\label{def:G-24sept}
Let $\widehat{G}_{n-D}$ denote a mean $0$ Gaussian of covariance $\sigma^2I_{n-D}$. Then,
\beqs \Pi_{n-D} (\ey  - z_{y, 2}^n) = \int_{x \in B_d(0, \frac{\tau}{20})} (f^{n-D}(x) + \E (\widehat{G}_{n-D}))  \mu''(dx)= \int_{x \in B_d(0, \frac{\tau}{20})} \f(x)  \mu''(dx). \eeqs
\end{definition}
In order to show that \beq\label{eq:dist-ey} \dist(\ey, \MM) < \frac{Cd\sigma^2}{\tau},\eeq it suffices to prove the following lemma. 
Note that the difference between Lemma~\ref{lem:11-May27} and the following Lemma~\ref{lem:12-May} is that in the latter we have the projection $\Pi_{n-D}$ from $\R^n$ to $\R^D$, while the former involves the projection $\Pi_{D-d}$ from $\R^D$ to $\R^d$.
\begin{lemma}\label{lem:12-May} The quantity $e^n_y$ satisfies 
\beqs \big|\Pi_{n-D} \ey - \f(y) \big| < \frac{C d\sigma^2}{\tau}. \eeqs
\end{lemma}
The proof is in Section~\ref{sec:subsidiary} of the Appendix.

Recall from Definition~\ref{dhaus} that for two subsets $X$ and $Y$ in a metric space $\mathbb M$, we define $dist(X, Y)$ to be $$\sup_{x \in X} \inf_{y \in Y} d_{\mathbb M}(x, y).$$ For use, the metric will be Euclidean, and $X$ and $Y$ will be subsets of a Euclidean space.
Thus $\dhaus(X, Y) = \max(dist(X, Y), dist(Y, X)).$

\begin{lemma}\label{lem:18}
\beqs \p\left[(dist(\mathrm{Rnet}, \MM) < \frac{Cd\sigma^2}{\tau}) \,\text{and}\, (dist(\MM, \mathrm{Rnet}) < C\sigma)\right] > 1 - N_0^{-50}.\eeqs
\end{lemma}
\begin{proof}
We know that assuming the Principal Component Analysis step does not produce an erroneous output, $\dhaus(\MM, \Pi_D \MM) < \a^2\tau,$ from Proposition~\ref{lem:5-May} and hence
the function $f^{n-D}$ (which was introduced in discussion immediately following Definition~\ref{def:G-24sept},)when restricted to $B_d(0, \frac{r}{2})$ has a $C^0-$norm of at most $\frac{\tau}{Cd}$. This, together with Lemma~\ref{lem:12-May} implies that
\beq \p[dist(\bigcup_i \{e_y|y \in Y_{D_i}\}, \MM) < \frac{Cd\sigma^2}{\tau}] > 1 - N_0^{-75}.\label{eq:107}\eeq It also implies that \beq \p[dist(\MM, \bigcup_i \{e_y|y \in Y_{D_i}\}) < C\sigma] > 1 - N_0^{-75}.\label{eq:108}\eeq We next show that with probability at least $1 - N_0^{-75}$, $\dhaus(\mathrm{Rnet}, \cup_i \{e_y|y \in Y_{D_i}\}) < \frac{C\sigma^2}{\tau}.$ This follows from Lemma~\ref{lem:14-May28}, (\ref{eq:windup}) and Proposition~\ref{thm:versh} applied separately to the random variables $\widehat z_y$ and $\widehat z_y^n.$ Indeed $\widehat z_y$ is the average of independent samples contained inside $\Pi_D V_y$ and hence $\|\widehat z_y\|_{\psi_2}$ is less than $\frac{C\tau}{\sqrt{\#y}}$. On the other hand $\widehat z_y^n$ is the average of $\#_y$ random points, each of which 
is the sum of two independent random variables, one that is a sample from $N_{n-D}(0, \sigma^2)$, and the other that is absolutely continuous with respect to the push forward of $\mu$ under $\Pi_{n-D}$. Since $\Pi_{n-D} \MM$ is contained in a ball of radius $\frac{\tau}{Cd}$ if the Principal Component Analysis step in Proposition~\ref{lem:5-May} executes correctly (which is a high probability event), this implies $\|\widehat z_y^n\|_{\psi_2}$ is less than $\frac{\tau + \sigma\sqrt{n}}{\sqrt{\#y}}$. We conclude that with probability at least $1 - N_0^{-75}$, $\dhaus(\mathrm{Rnet}, \cup_i \{e_y|y \in Y_{D_i}\}) < \frac{C\sigma^2}{\tau}.$ The lemma follows from (\ref{eq:107}) and (\ref{eq:108}).
\end{proof}
 
\subsection{Boosting the probability of correctness of ${\mathrm{Rnet}}$ to $1-\xi$.}
Now consider $G$ to be metric space whose elements are finite subsets of $\R^n$, and the metric is
 the Hausdorff distance. Let  $p = \cup_i \{e_y | y\in Y_{D_i}\}$, and $\eps = \frac{C \sigma^2}{\tau}$. We have a procedure (see Definition~\ref{def:Rnet}) by which we can produce independent $p_1, p_2, \dots$ in $G$ such that for each $i$, $\p[dist(p_i, p) < \eps] > \frac{2}{3}.$ Then, we may take $C\log (\xi^{-1})$ points $p_i$ and search for an index $j$ such that at least a $\frac{4}{7}$ fraction of all the points $p_i$ lie within a $2\eps$ ball of $p_j$. If no such $p_j$ exists we declare failure, but if such a $p_j$ is found, as will happen with probability at least $1-\frac{\xi}{10}$, this $p_j$ has the property that it is within $3\eps$ of $p$ with probability at least $1 - \frac{\xi}{10}.$ 



\subsection{Using discs to approximate $\MM$ at a fine scale}
We now set $r$ to a much smaller value than $r_c$, namely  \beq r:= C\sqrt{d}\sigma,\eeq and apply the algorithm in Section~\ref{sec:follows} to the refined net obtained in Definition~\ref{def:Rnet}.
Let $X_3 = \{p_i\}$ be a minimal $cr/d-$net of $\mathrm{Rnet}$.  Such a net can be chosen greedily, ensuring at every step that no element included in the net thus far is within $\frac{cr}{2d}$ of the point currently chosen. The process continues while progress is possible. Let the size of $X_3$ be denoted $N_3$.

\subsubsection{Remark on notation} Below is the definition of the discs $\{D_i\}_{i \in [N_3]}$ and the balls $\{U_i\}_{i \in [N_3]}$ which will be used henceforth in the following sections.
We introduce a family of $n$ dimensional balls of radius $r ,$  $\{U_i\}_{i \in [N_3]}$ where the center of $U_i$ is $p_i$ and a family of $d-$dimensional embedded discs of radius $r$, denoted $\{D_i\}_{i \in [N_3]}$, $D_i \subseteq U_i$ where $D_i$ is centered at $p_i$. 
The $D_i$ are chosen by fitting a disc that approximately minimizes among all discs of radius $r$ centered at $p_i$  the Hausdorff distance to $U_i \cap \mathrm{Rnet}$ by the procedure described in Subsection~\ref{subsec:algorithm-finddisc}. 
We will need the following properties of  $(D_i, p_i)$, which hold with probability at least $1 - N_0^{-C}$. Property (F1) follows from Lemma~\ref{lem:18}. 
\ben
\item[(F1)] The Hausdorff distance between $\cup_i D_i$ and $ \MM$ is less than $\frac{Cr^2}{ \tau} = \de$.
\item[(F2)] For any $i \neq j$, $|p_i - p_j| > \frac{cr}{{d}}$.
\item[(F3)] For every $z \in  \MM$, there exists a point $p_i$ such that $|z - p_i| <  3 \inf_{i \neq j}, |p_i - p_j|.$ 
\een 

\section{Computing weights used to define  output manifold $\MM_o$}\label{sec:weights}

For $v \in \R^n$, let $\tbeta(v) = \left(1 - |v|^2\right)^{d+k}$ for $|v| \leq 1$, and $\tbeta(v) = 0$ for $|v| > 1$.
Consider the bump function $\ta_i:\R^n \ra \R$  given by $$\ta_i(p_i + rv) = c_i\tbeta(v).$$ Here $k$ is some fixed integer greater or equal to $3$.
 Let $$\ta(x)  := \sum_{i\in [N_3]} \ta_i(x),\quad\hbox{ and }\a_i(x) = \frac{\ta_i(x)}{ \ta(x)},\hbox{ for each } i \in [N_3].$$
These weights will be used to construct partitions of unity. While we use them to construct a manifold, they have been also been used in the manifold learning literature, to learn measures on manifolds (e.g., see \cite{Divol}.)
\begin{lemma}\label{lem:weights}
It is possible to choose $c_i$  such that for any $z$ in a $\frac{r}{{4d}}$ neighborhood of $\MM$, $$c^{-1} >   \ta(z) > {c},$$ where $c$ is a small universal constant.  Further, such $c_i$ can be computed using no more than $N_0(Cd)^{2d}$ operations involving vectors of dimension $D$.
\end{lemma}
After appropriate scaling, we will assume that $r=1$.
 We will need the following claim.
\begin{claim} Let $\la_\MM$ denote the Hausdorff measure supported on $\MM$ and let $\la$ denote the Lebesgue measure on $\R^D$.
There exists  $\kappa \in \R$ such that the following is true. For all $z$ in a $\frac{r}{{4d}}$ neighborhood of $\MM$, $$c \kappa^{-1} < \frac{d(\la_\MM \ast \tbeta)}{d\la} (z)< c^{-1}\kappa^{-1}.$$
\end{claim}
\begin{proof}

We make the following claim.

\begin{claim} If $|v| < \frac{1}{\sqrt{2}}$, then $ \exp(-2(d+k)|v|^2) < \tbeta(v).$ Also
 \beq \forall v \in \R^D, \,\exp(-(d+k)|v|^2) > \tbeta(v). \eeq
\end{claim}
 \begin{proof}
To see the first inequality, note that \beq |v| & < & \frac{1}{\sqrt{2}}\\
\implies (-2)(1 - |v|^2) & < & -1\\
\implies (-2)(d+k)|v|^2 & < & (d+k) \left(- \frac{|v|^2}{1 - |v|^2}\right)\\
\implies (-2)(d+k)|v|^2 & < & (d+k) \left(\ln (1 - |v|^2)\right)\\
\implies  \exp((-2)(d+k)|v|^2) & < & \left(1 - |v|^2\right)^{d+k} = \tbeta(v).\eeq

To see the second inequality, exponentiate the following inequality for $|v| < 1$:
\beq  -(d+k)|v|^2 & > & (d+k) \left(\ln (1 - |v|^2)\right). \eeq

When $|v| \geq 1$, $\tbeta(v) = 0$, so the inequality holds.

\end{proof}

We next provide the proof of Lemma~\ref{lem:weights}.

We will need the following Proposition which follows from Theorem 3.2.3 in  \cite{federer_book}.

\begin{proposition}\label{prop:Jac}
Let $\mathcal{L}^m$ denote the $m-$dimensional Lebesgue measure and $\mathcal{H}^m$ denote the $m-$dimensional Hausdorff measure.
Suppose $f:A \ra \R^n$ be an injective $C^2$ function with $m \leq n$ where $A$ is a $\mathcal{L}^m-$measurable subset of $\R^m$ and $J_m f$ is the Jacobian of $f$.  If $u:A \ra \R$ is a $\mathcal{L}^m$ integrable function, then 
\beq \int_{A} u(x) J_m f(x) \mathcal{L}^m (dx) = \int_{f(A)} u(f^{-1}(y)) \mathcal{H}^m(dy). \eeq
\end{proposition}

\begin{proof}[Proof of Lemma~\ref{lem:weights}]

We will use the preceding claim to get upper and lower bounds on $$\int_{\R^d} \tbeta(v)\lambda(dv)  = : c_\tbeta^{-1},$$ where $\lambda$ corresponds to the $d-$dimensional Lebesgue measure.
\beq \int_{\R^d} \tbeta(v)\lambda(dv) & = &  \int_{B_d} \tbeta(v)\lambda(dv) \\
                                                         & \leq & \int_{B_d} \exp(-(d+k)v)\lambda(dv)
                                                          \leq  (\frac{\pi}{d+k})^{d/2}.\eeq
Also,
\beq \int_{B_d} \tbeta(v)\lambda(dv), & \geq & \int_{ \frac{B_d}{\sqrt{2}}} \exp(-2(d+k)v)\lambda(dv)
\geq  c (\frac{\pi}{2(d+k)})^{d/2} .\eeq

Using numerical integration the value of $c_\tbeta$ can be estimated to within a multiplicative factor of $2$ using $(Cd)^d$ operations on real numbers.

Next consider a unit disk $B_d \subseteq \R^n$ equipped with the measure $c_\tbeta \lambda$. We consider a point $q$ at a distance $\Delta$ from the projection of $q$ onto $B_d$, which we assume is the origin. As a warm-up, we will be interested in \beq \frac{ ((c_\tbeta \lambda \one_{B_d})* \tbeta)(q)}{ ((c_\tbeta \lambda \one_{B_d})* \tbeta)(0)} = \frac{\int_{B_d} \tbeta(q -v)(c_\tbeta \lambda(dv))}{\int_{B_d} \tbeta( -v)(c_\tbeta \lambda(dv))},\eeq as a function of $\Delta$.
We observe that  $v \in B_d \implies \tbeta(-v) \geq \tbeta(q - v)$, and so \beq \frac{ ((c_\tbeta \lambda \one_{B_d})* \tbeta)(q)}{ ((c_\tbeta \lambda \one_{B_d})* \tbeta)(0)} \leq 1. \eeq
Let $\Delta^2 \leq \frac{1}{8d^2}$. 
Suppose $|v|^2 < 1 - \frac{1}{2d}$, then \beq \De^2 \leq \left(\frac{1-|v|^2}{4d}\right).\eeq
Therefore, \beq\nonumber  \int_{B_d} \beta(q -v)(c_\tbeta \lambda(dv)) & = &  \int_{B_d} (1 - |v|^2 - \De^2)^{d+k}\one_{\{v| |v|^2 \leq 1 - \De^2\}}(c_\tbeta \lambda(dv))\\\nonumber
                                                                                              & \geq &  \int_{\sqrt{1 - \frac{1}{2d}}B_d} ((1 - |v|^2)(1 - \frac{1}{4d}))^{d+k}(c_\tbeta \lambda(dv))\\
                                                                                              & \geq &  \int_{\sqrt{1 - \frac{1}{2d}}B_d} c (1 - |v|^2)^{d+k}(c_\tbeta \lambda(dv))\\
                                                                                              & \geq & c \int_{B_d} (1 - |v|^2)^{d+k}(c_\tbeta \lambda(dv)).
\eeq
In the above sequence of inequalities the last step comes from dilating the disk $\sqrt{1 - \frac{1}{2d}}B_d$ to $B_d$ and observing that $\tbeta(v_1) \geq \tbeta(v_2)$ if $|v_1| < |v_2|$.

We thus have \beq c \leq \frac{ ((c_\tbeta \lambda \one_{B_d})* \tbeta)(q)}{ ((c_\tbeta \lambda \one_{B_d})* \tbeta)(0)} = \frac{\int_{B_d} \tbeta(q -v)(c_\tbeta \lambda(dv))}{\int_{B_d} \tbeta( -v)(c_\tbeta \lambda(dv))} \leq 1,\eeq for some absolute constant $c > 0$ provided $\Delta^2 \leq \frac{1}{8d^2}$.

Next consider a point $q$ at a distance $\leq 1/(2d)$ from $\MM$. We let $q$ be the origin. Consider a unit disk $B_d \subseteq \R^D$ that is parallel to the tangent plane to $\MM$ at the point nearest to $q$. We will be interested in \beq \frac{ ((c_\tbeta \mathcal{H}^d_\MM \one_{B_m})* \tbeta)(q)}{ ((c_\tbeta \lambda \one_{B_d})* \tbeta)(0)} = \frac{\int_{\MM \cap B_D} \tbeta(-v)(c_\tbeta  \mathcal{H}^d_\MM(dv))}{\int_{B_d} \tbeta( -v)(c_\tbeta \lambda(dv))},\eeq as a function of $\Delta$.
Let $\Pi_d$ denote the projection onto $B_d$.
Let \beq \sup_{x \in \MM \cap B_n} |x - \Pi_d x| = \Delta.\eeq Then, by Federer's criterion for the reach, $\Delta < 1/d$. Also, $\MM \cap B_n$ is the graph of a function $f(x)$ from $\Pi_d(\MM\cap B_n)$ to the $n-d$ dimensional normal space to $B_d$. For $v \in \MM \cap B_n$, let $w = \Pi_d v$, and by the definition of $f$, $v = w + f(w)$. 
{ Then,}
\beq \nonumber  \int_{\MM \cap B_n} \tbeta(-v)(c_\tbeta  \mathcal{H}^d_\MM(dv)) & = & \int_{\Pi_d(\MM\cap B_n)} \beta(-(w + f(w)))(c_\beta \mathcal{H}^d_\MM(dv))\\
                                                    							  & \leq & \int_{\Pi_d(\MM\cap B_n)} \tbeta(-w)(c_\tbeta \mathcal{H}^d_\MM(dv))\\
												   & \leq &  \int_{\Pi_d(\MM\cap B_n)} \tbeta(-w)(c_\tbeta J(w)\la(dw)).
\eeq
Since $\|Df\|$ is of the order of $\frac{1}{Cd^C}$ by Lemma~\ref{lem:6} and the upper bound on $r$, the Jacobian $$J(w) = \sqrt {\det(I + (Df(w)) (Df(w))^{T})}$$ is less or equal to an absolute constant $C.$ This, in view of Proposition~\ref{prop:Jac}, implies that 
\beq \int_{\Pi_d(\MM\cap B_n)} \tbeta(-w)(c_\tbeta J(w)\la(dw)) \leq C \int_{B_d} \tbeta( -v)(c_\tbeta \lambda(dv)).\eeq This in turn implies that
\beq c^{-1} > \frac{\int_{\MM \cap B_n} \tbeta(-v)(c_\tbeta  \mathcal{H}^d_\MM(dv))}{\int_{B_d} \tbeta( -v)(c_\beta \lambda(dv))} .\eeq for an appropriately small universal constant $c$.

We now proceed to the lower bound. As noted above, $\Delta < 1/d$. { Then,}
\beq \nonumber \int_{\MM \cap B_n} \tbeta(-v)(c_\tbeta  \mathcal{H}^d_\MM(dv)) & = & \int_{\Pi_d(\MM\cap B_n)} \tbeta(-(w + f(w)))(c_\tbeta \mathcal{H}^d_\MM(dv))\\
                                                    							  & \geq & \int_{B_d(1 - 1/d)} \tbeta(-(w + f(w)))(c_\tbeta \mathcal{H}^d_\MM(dv))\\\nonumber
												   & \geq &  \int_{B_d(1 - 1/d)} (1-|w|^2-\Delta^2)^{d+k}(c_\tbeta J(w)\la(dw))\\\nonumber
                                                                                                              & \geq & \int_{B_d(1 - 1/d)} ((1-|w|^2)(1-1/d))^{d+k}(c_t\beta \la(dw))\\
												& \geq & 	c \int_{B_d(1 - 1/d)} (1-|w|^2)^{d+k}(c_\tbeta \la(dw))\\
												& \geq & c^2 \int_{B_d} (1-|w|^2)^{d+k}(c_\tbeta \la(dw)).
\eeq

 The last step comes from dilating the disk $(1 - \frac{1}{d})B_d$ to $B_d$ and observing that $\tbeta(v_1) \geq \tbeta(v_2)$ if $|v_1| < |v_2|$. In dropping $J(w)$, we used the fact that $J(w) \geq 1$.
                                                                         
Relabelling $c^2$ by $c$, the above sequence of inequalities shows that                                      
\beq  \frac{\int_{\MM \cap B_n} \tbeta(-v)(c_\tbeta  \mathcal{H}^d_\MM(dv))}{\int_{B_d} \tbeta( -v)(c_\tbeta \lambda(dv))} > c.\eeq
\end{proof}

We next, using the fact that the Hausdorff distance of the set $\{p_i\}$ to $\MM$ is less than $\frac{cr}{d}$  show the following. 

\begin{claim} There exists a measure $\mu_P$ supported on $\{p_i\}$ such that 
$$c  < \frac{d(\mu_P * \tbeta)}{d\la}(z) < c^{-1},$$ for all $z$ in a $\frac{r}{4d}$ neighborhood of $\MM$ in $\R^n$.
\end{claim}
\begin{proof}
For any $\eps \in (0,1)$, let $\tbeta_\eps(\eps rv) =  c_{\eps,\tbeta}(1 - \|v\|^2)^{d+k}$ if $|v| \leq 1$, and $\tbeta_\eps(\eps rv) = 0$ if $|v| > 1$. Here $c_{\eps,\tbeta}$ is chosen so that $\tbeta_\eps$ integrates to $1$ over $\R^n$.
\begin{definition} \label{def:vor} For $i \in [N_3]$, let $\hbox{Vor}_i$ denote the open set of all points $p \in \R^n$ such that for all $j \neq i$, $|p - p_i| < |p - p_j|.$
Let $$\mu_P(p_i) = (c_\beta \mathcal{H}_{\MM}^d* \beta_\eps)(\hbox{Vor}_i).$$
\end{definition}
We note 
 $\frac{d(c_\tbeta \mathcal{H}_\MM^d * \tbeta)}{d\la}(z) $ is a $\frac{d}{cr}-$Lipschitz function of $z$, and 
 { so is also the function} 
$$\frac{d(c_\tbeta \mathcal{H}_\MM^d* \tbeta*\tbeta_\eps)}{d\la}(z),$$ for any $\eps \in (0, 1)$. Further, there exists an $\eps_0 \in (0, 1)$ such that 
$$\forall \eps \in (0, \eps_0), \left\|\frac{d(c_\tbeta \mathcal{H}_\MM^d * \tbeta*\tbeta_\eps)}{d\la} - \frac{d(c_\tbeta \mathcal{H}_\MM^d * \tbeta)}{d\la}\right\|_{\mathcal{L}^\infty(\R^n)} < c(\eps),$$ where $\lim_{\eps \ra 0} \frac{c(\eps)}{\eps}$ exists and is finite.
It thus suffices to prove that for all $\eps \in (0, \eps_0)$, 
$$ \left\|\frac{d(c_\tbeta \mathcal{H}_\MM^d *\tbeta*\tbeta_\eps)}{d\la} - \frac{d(\mu_P * \tbeta)}{d\la} \right\|_{\mathcal{L}^\infty(\R^d)} < \frac{c}{2} - c(\eps)$$ for all $z$ in a $\frac{r}{4d}-$neighborhood of $\MM$. 
For any $i$,  $$\hbox{diam(supp}(c_\tbeta \mathcal{H}_\MM^d *\tbeta_\eps)\cap \hbox{Vor}_i)
< \frac{cr}{d}. $$ Let $\pi$ denote the map defined on $supp(c_\tbeta \mathcal{H}_\MM^d *\tbeta_\eps)$ from $\hbox{Vor}_i$ to $p_i$.
Then, \beqs \left| \frac{d(c_\tbeta\mathcal{H}_\MM^d *\tbeta*\tbeta_\eps)}{d\la}(z) - \frac{d(\mu_P * \tbeta)}{d\la}(z)\right| & = & 
\left| \frac{d(((c_\tbeta \mathcal{H}_\MM^d *\tbeta_\eps) - \mu_P) * \tbeta)}{d\la}(z)\right| .\eeqs
                  For any $w \in supp(c_\tbeta \mathcal{H}_\MM^d *\tbeta_\eps)\cap \hbox{Vor}_i$, $|\pi(w) - w| < \frac{cr}{d}$. 
Let $c_\tbeta\mathcal{H}_\MM^d *\tbeta*\tbeta_\eps$ be denoted $\nu$.
Then, 
\beqs  \frac{(\nu - \mu_P) * \tbeta}{d\la}(z)
& = & \int_{z + supp(\tbeta)} \nu(dx)\tbeta(z-x) - \int_{z + supp(\tbeta)} \mu_P(dy) \tbeta(z-y)\\
& = & \int_{z + supp(\tbeta)} \nu(dx)\tbeta(z-x) - \int_{z + supp(\tbeta)} \nu(dx) \tbeta(z-\pi(x)).
\eeqs
The Lemma follows noting that $\tbeta$ is $\frac{d}{cr}-$Lipschitz.

\end{proof}
Let $\la_d^i$ denote the $d-$dimensional Lebesgue measure restricted to the disc $D_i$. 

Recall that from Definition~\ref{def:vor} that $$\mu_P(p_i) = (c_\tbeta \mathcal{H}_{\MM}^d* \tbeta_\eps)(\hbox{Vor}_i),$$ where the $\hbox{Vor}_i\subset \R^n.$
Let $$\widetilde{\mu_P}(p_i) = (c_\tbeta  \la_d^i)(\hbox{Vor}_i \cap D_i).$$
By making $\frac{r}{\tau} < \frac{1}{Cd^C}$ for a sufficiently large universal constant $C$, and $\epsilon$ a sufficiently small quantity, we see that for each $i$, $$c \leq \frac{\widetilde{\mu_P}(p_i)}{{\mu_P}(p_i)} \leq c^{-1}.$$ for a suitable universal constant $c$. We see that $(c_\tbeta  \la_d^i)(\hbox{Vor}_i \cap D_i)$ is the volume of the polytope $\hbox{Vor}_i \cap D_i$ multiplied by $c_\tbeta$, and membership of a point in $\hbox{Vor}_i \cap D_i$ can be answered in time $(Cd)^d$. Thus by placing a sufficiently fine grid, and counting the lattice points in $\hbox{Vor}_i \cap D_i$, $\widetilde{\mu_P}(p_i)$ can be computed using $(Cd)^{2d}$ deterministic steps. Even faster randomized algorithms exist for the task, which we choose not to delve into here. This concludes the proof of Lemma~\ref{lem:weights}.
\end{proof}

\section{The output manifold}\label{sec:calc}
For the course of this section, we consider the \underline{scaled setting} where $r=1$. Thus, in the new Euclidean metric, $\tau \geq Cd^C$.

Let $\Pi^i$ be the orthogonal projection of $\R^n$ onto the $n-d-$dimensional subspace containing the origin that is orthogonal to the affine span of $D_i$.
Recall that the $p_i$ are the centers of the discs $D_i$ as $i$ ranges over $[N_3]$.
We define the function $F_i:U_i \ra \R^n$ by $F_i(x) =\Pi^i (x - p_i)$. Let $\cup_i U_i = U$.
We define $$F:U \ra \R^n$$ by \beq F(x) = \sum_{i \in [N_3]} \a_i(x) F_i(x).\eeq

 Given a symmetric matrix $A$ such that $A$ has $n-d$ eigenvalues in $(1/2 , 3/2)$ and $d$ eigenvalues in $(-1/2, 1/2)$, let $\Pi_{hi}(A)$ denote the projection onto the span of the eigenvectors of $A$, corresponding to the largest $n-d$ eigenvalues.

\begin{definition} \label{def:6.9} For $x \in \cup_i U_i$, we define $\Pi_x = \Pi_{hi}(A_x)$ where $A_x = \sum_i \a_i(x) \Pi^i$. 
\end{definition}

Let $\widetilde{U}_i$ be defined as the $\frac{cr}{d}-$Eucidean neighborhood of $D_i$ intersected with $U_i$. Given a matrix $X$, its Frobenius norm $\|X\|_F$ is defined as the square root of the sum of the squares of all the entries of $X$. This norm is unchanged when $X$ is premultiplied  or postmultiplied by orthogonal matrices (of the appropriate order). Note that $\Pi_x$ is $\C^2$ when restricted to $\bigcup_i \widetilde U_i$, because the $\a_i(x)$ are $C^2$ and when $x$ is in this set, $c < \sum_i \widetilde\alpha_i(x) < c^{-1}$, and for any $i,j$ such that $\a_i(x) \neq 0 \neq \a_j(x)$, we have $\|\Pi^i - \Pi^j\|_F < Cd\de$.
\begin{definition}\label{def:7}
The output manifold $\MM_o$ is the set of all points $x \in \bigcup_i \widetilde{U}_i$  such that $\Pi_x F(x) = 0$.
\end{definition}


As stated above,  $\MM_o$ is the set of points $x \in \bigcup_i \widetilde{U_i}$ such that 
\beq \Pi_{hi}(\sum_{i\in [N_3]} \a_i(x)\Pi^i)(\sum_{i \in [N_3]} \a_i(x)\Pi^i(x - p_i)) = 0. \eeq
We see that \beqs  \Pi_{hi}(\sum_i \a_i(x)\Pi^i) = \frac{1}{2\pi i} \left[\oint_\gamma (zI - (\sum_i \a_i(x)\Pi^i))^{-1}dz\right]\eeqs
using diagonalization and Cauchy's integral formula, and so

\beq\label{MMo} \frac{1}{2\pi i} \left[\oint_\gamma (zI - (\sum_i \a_i(x)\Pi^i))^{-1}dz\right] \left(\sum_i \a_i(x)\Pi^i(x - p_i)\right) = 0\eeq
where $\gamma$ is the circle of  radius $1/2$ centered at $ 1$.

Let \beq \sum_{i \in [N_3]} \a_i(x) \Pi^i = M(x),\label{eq:M(x)}\eeq and as stated earlier, $\Pi^i(x - p_i) = F_i(x)$. Let $\Pi_{hi}(M(x))$ be denoted $\Pi_x$.

Then the left hand side of (\ref{MMo}) can be written as 
\beq \oint_\gamma \frac{dz}{2\pi i}\left(\sum_i \a_i(x) (zI - M(x))^{-1}F_i(x)\right). \eeq
for any $v \in \R^{\widehat n}$ and and $f:\R^{\widehat n} \ra \R^{\widetilde n}$ where $\widehat n, \widetilde n \in \N_+$ let $$\partial_v f(x) := \lim_{\a \ra 0} \frac{f(x + \a v) - f(x)}{\a}.$$

Then,
\beq \partial_v  \oint_\gamma \frac{dz}{2\pi i}\left(\sum_i \a_i(x) (zI - M(x))^{-1}F_i(x)\right) & = & \sum_i \a_i(x) \Pi_x (\partial_v F_i(x)) \label{eq:31}\\
& + &  \sum_i \a_i(x) (\partial_v \Pi_x)  F_i(x) \label{eq:32}\\
& + &   \sum_i (\partial_v \a_i(x)) \Pi_x  F_i(x). \label{eq:33}\eeq


 Let $v \in \R^n$ be such that $\|v\| = 1$. Let $\MM_\Pi^d$ denote the set of all projection matrices of rank $d$. This is an analytic submanifold of the space of $n\times n$ matrices.

\begin{claim}\label{cl:grassmann} The reach of $\MM_\Pi^d \subset \R^{n \times n}$ is greater or equal to $1/2.$
\end{claim}
\begin{proof}
Let $$\MM_\Pi := \bigcup_{\widehat d=0}^n \MM_\Pi^{\widehat d}.$$ The various connected components 
of $\MM_\Pi$ are the different $\MM_\Pi^d$ (whose dimensions are respectively $(n-d)d$), and by evaluating Frobenius norms, we see that the distance between any two points on distinct connected components is at least $1$. Since it suffices to show that a normal disc bundle of radius less than $1/2$ injectively embeds into the ambient space (which is $\R^{n(n-1)/2}$,)  it suffices to show that
$$reach(\MM_\Pi) = 1/2.$$ Let $x \in \MM_\Pi^d$. Let $z$ belong to the normal fiber at $x$ and let $\|x-z\|_F < 1/2$. Without loss of generality we may (after diagonalization if necessary) take $x = diag(1, \dots, 1, 0, \dots, 0)$ where the number of $1s$ is $d$ and the number of $0s$ is $n-d$. Further, (using block diagonalization if necessary), we may assume that $z$ is a diagonal matrix as well. All the eigenvalues of $z$ lie in $(1/2, 3/2)$ and further the span of the corresponding eigenvectors is the space of eigenvectors of $x$ corresponding to the eigenvalue $1$. Therefore $\Pi_{hi}(z)$ is well defined through Cauchy's integral formula and equals $x$. Thus the normal discs of radius $< 1/2$ do not intersect, and so $reach(\MM_\Pi^d) \geq 1/2.$ Conversely, $\MM_\Pi^0$ is the origin and $\MM_\Pi^1$ contains the point $diag(1, 0, \dots, 0)$. We see that $diag(1/2, 0, \dots, 0)$ is equidistant from $\MM_\Pi^0$ and $\MM_\Pi^1$ and the distance is $1/2$. Therefore $reach(\MM_\Pi^d) \leq 1/2.$ Therefore, 
$$reach(\MM_\Pi^d) \geq reach(\MM_\Pi) = 1/2.$$
\end{proof}

In what follows, we will make repeated use of H\"older's inequality for $\ell_p$ norms and $\ell_q$ norms: Let $p, q \in \R$ and $\frac{1}{p} + \frac{1}{q} = 1$, then,
$$\forall x, y \in \R^n, \langle x, y \rangle \leq \|x\|_p \|y\|_q.$$ 

Secondly, we will use the fact that for any ball $U_i$, the number of $j$ such that $U_i \cap U_j$ is nonempty is bounded above by $(Cd)^d$ because of the lower bound of $\frac{cr}{d}$ on the spacing between the $p_i$ and $p_j$ for any two distinct $i$ and $j$. A consequence of this is that any vector $w \in \R^{N_3}$ that is supported on the set of all $j$ such that $U_i \cap U_j \neq \emptyset$ will satisfy 
\beqs 
\|w\|_{d+k} \leq Cd \|w\|_\infty,\quad
\|w\|_{\frac{d+k}{2}} \leq Cd^2 \|w\|_\infty,\quad
 \|w\|_{\frac{d+k}{3}} \leq Cd^3 \|w\|_\infty.\eeqs
Thirdly, we will use bounds on the derivatives of the bump functions at points $x$ that are within a distance of $cr/d$ of $\MM$.
Recall that $\sum_i \ta_i(x)$ is denoted $\ta(x)$. Then we know that $c < \ta(x) < C$ if the distance of $x$ from $\MM$ is less than $cr/d$. Recall that $N_3$ is the total number of balls $U_i$.
Recall from Lemma~\ref{lem:weights} that  for any $z$ in a $\frac{r}{{4d}}$ neighborhood of $\MM$, $$c^{-1} >   \ta(z) > {c},$$ where $c$ is a small universal constant. Note that  For any $v \in \R^D$ such that $|v| = 1$, and any $x \in \R^D$ such that $dist(x, \MM) \leq \frac{cr}{d}$, $   \|(\partial_v \a_i(x))_{i \in[N_3]}$ is a vector in $\R^{N_3}$. For a vector $w \in \R^{N_3}$ we denote the $\ell_p$ norm of $w$ by $\|w\|_p.$
\begin{lemma} \label{lem:ta2} For any $v \in \R^D$ such that $|v| = 1$, and any $x \in \R^D$ such that $dist(x, \MM) \leq \frac{cr}{d}$, \beq   \|(\partial_v \a_i(x))_{i \in[N_3]}\|_{\frac{d+k}{d+ k -1}} \leq Cd^2. \eeq \end{lemma}
\begin{proof}
{ We have}
\beqs  \nonumber  \|(\partial_v \a_i(x))_{i \in[N_3]}\|_{\frac{d+k}{d+k-1}} & = &  \|(\partial_v \frac{\ta_i(x)}{\ta(x)})_{i \in[N_3]}\|_{\frac{d+k}{d+k-1}}\\ & \leq &
 \frac{ \|(\partial_v \ta_i(x))_{i \in[N_3]}\|_{\frac{d+k}{d+k-1}}}{\ta} +   \frac{\|( (\partial_v \ta(x)) \ta_{i}(x))_{i\in[N_3]}\|_{\frac{d+k}{d+k-1}}}{\ta^2}\\
& \leq & (c^{-1}) \|Cd(\ta_i(x))_{i \in[N_3]}\|^{\frac{d+k-1}{d+k}}_1 + (c^{-2}) \| (\ta_i(x))_{i \in[N_3]}\|_{\frac{d+k}{d+k-1}} |\partial_v \ta|\\
& \leq & Cd + C|\partial_v \ta|\\
& \leq & Cd + C\|(\partial_v \ta_i)_{i  \in [N_3]}\|_{\frac{d+2}{d+1}}\|(1)_{i \in [N_3]}\|_{d+2}\\
& \leq & C d^2 .\eeqs
\end{proof}
Recall that as the $F_i$ are affine maps, $\partial F_i(x) = \Pi^i.$
We first look at the right hand side of (\ref{eq:31}).  This can be rewritten as
	\beq \sum_{i \in [N_3]} \a_i(x)\Pi_x \Pi_i v = \Pi_x v + \Pi_x\left( {M(x)} - \Pi_x\right)v.\eeq
                    It follows from properties of the Frobenius norm that 
$$\Pi_{hi}(A_x) = \arg\min_{\Pi \in \MM^s_\Pi} \|A_x - \Pi\|_F.$$                         Thus,         recalling from (\ref{eq:M(x)}) that $$\sum_{i \in [N_3]} \a_i(x) \Pi^i = M(x),$$
\beqs  \| {M(x)} - \Pi_x\|_F & = & dist(M(x), Tan(\Pi_x, \MM_\Pi^d))\\
& \leq & \sup_i dist(\Pi_i, Tan(\Pi_x, \MM_\Pi^d))\\
& \leq & \sup_i \|\Pi_i - \Pi_x\|_F^2/(2 \, reach(\MM_\Pi^d))\\
& \leq & 4 \sup_{i,j} \|\Pi_i - \Pi_j\|_F^2\\
& \leq & 8d\de^2.\eeqs

In the above array of equations, $i, j$ are such that $\a_i(x)$ and $\a_j(x)$ are nonzero.
                                                                     
We look at (\ref{eq:32}) next. Observe that
\beq\nonumber  \left \| \oint_\gamma \frac{dz}{2\pi i} \left(\sum_{i \in [N_3]} \a_i(x) \left(\partial_v ((zI - M(x))^{-1})\right) F_i(x)\right)\right\| & \leq &  \left\| \partial_v \Pi_x \right\|\left\|\sum_{i \in [N_3]} \a_i(x)  F_i(x)\right\|.\eeq

\begin{lemma}\label{lem:ta} We have for any $v \in \R^n$ such that $|v| = 1$, and any $x \in \R^n$ such that $dist(x, \MM) \leq \frac{cr}{d}$, \beq   \|(\partial^2_v \a_i(x))_{i \in[N_3]}\|_{\frac{d+k}{d+k-2}} \leq Cd^4. \eeq \end{lemma}
\begin{proof} We have
\beqs   \|(\partial^2_v \a_i(x))_{i \in[N_3]}\|_{\frac{d+k}{d+k-2}} & = &  \|(\partial^2_v \frac{\ta_i(x)}{\ta(x)})_{i \in[N_3]}\|_{\frac{d+k}{d+k-2}}\eeqs
\beqs =   \|( \frac{\partial^2_v \ta_i(x)}{\ta(x)} + \frac{(-2)(\partial_v \ta_i(x))(\partial_v\ta(x))}{\ta(x)^2} +  \frac{\ta_i(x)}{\ta(x)^3}(2(\partial_v \ta)^2 - \partial_v^2 \ta(x) (\ta(x)) )_{i \in[N_3]}\|_{\frac{d+k}{d+k-2}}\nonumber. \eeqs

We use the triangle inequality on the above expression, and reduce the task of obtaining an upper bound to that of separately obtaining the following bounds.

\begin{claim} We have \beq \|( \frac{\partial^2_v \ta_i(x)}{\ta(x)})_{i \in[N_3]}\|_{\frac{d+k}{d+k-2}} \leq Cd^2, \eeq \end{claim}

\begin{proof} This follows from $c < \ta < C$, and the discussion below. Suppose $x$ belongs to the unit ball in $\R^D$. Then,
\beq \partial_v^2 ( 1 - \|x\|^2)^{k+d}  & = & \partial_v((k+d)(1 - \|x\|^2)^{k+d-1}(2 \langle x, v \rangle))\\
                                                             & = & (k+d)(k+d-1) (1 - \|x\|^2)^{k+d-2}(4 \langle x, v \rangle)^2)\\
						       & + &   (k+d)(1 - \|x\|^2)^{k+d-1}(2 \langle v, v \rangle)).
\eeq
Therefore, 
 \beq \|( \partial^2_v \ta_i(x))_{i \in[N_3]}\|_{\frac{d+k}{d+k-2}} & \leq & C\left(d^2\ta + d\|( \partial_v \ta_i(x))_{i \in[N_3]}\|_{\frac{d+k}{d+k-2}}\right)\\
& \leq & C\left(d^2 + d \|( \partial_v \ta_i(x))_{i \in[N_3]}\|_{\frac{d+k}{d+k-1}}\right)
 \leq  Cd^2.\eeq 
\end{proof}

\begin{claim} We have \beq \|(  \frac{(-2)(\partial_v \ta_i(x))(\partial_v \ta(x))}{\ta(x)^2}  )_{i \in[N_3]}\|_{\frac{d+k}{d+k-2}} \leq Cd^3. \eeq \end{claim}

\begin{proof}
We have seen that  $|\partial_v \ta(x)| < Cd^2$. Therefore, 
\beq \nonumber  \|(  \frac{(-2)(\partial_v \ta_i(x))(\partial_v \ta(x))}{\ta(x)^2}  )_{i \in[N_3]}\|_{\frac{d+k}{d+k-2}} & < & Cd^2 \|(  (\partial_v \ta_i(x))  )_{i \in[N_3]}\|_{\frac{d+k}{d+k-2}}\\
& \leq  & Cd^2 \|(  (\partial_v \ta_i(x))  )_{i \in[N_3]}\|_{\frac{d+k}{d+k-1}}\\
& \leq & Cd^3. \eeq
\end{proof}

\begin{claim} { We have} \beq \|(   \frac{\ta_i(x)}{\ta(x)^3}(2(\partial_v \ta)^2 - \partial_v^2 \ta(x) (\ta(x)) )_{i \in[N_3]}\|_{\frac{d+k}{d+k-2}} \leq Cd^4. \eeq
\end{claim}
\begin{proof}
The only term that we have not already bounded is $|\partial_v^2 \ta(x)|$. To bound this, we observe that
\beq  C|\partial^2_v \ta| & \leq & C\|(\partial^2_v \ta_i)_{i \in [N_3]}\|_{\frac{d+k}{d+k-2}}\|(1)_{i \in [N_3]}\|_{(d+k)/2}\\
& \leq & C d^4 .\eeq Therefore, the entire expression gets bounded by $Cd^4$ as well.
\end{proof}

{ This proves Lemma \ref{lem:ta}.}\end{proof}

Recall that $ F(x) = \sum\a_i(x) F_i(x).$
\subsection{A bound on the first derivative of $\Pi_xF(x)$}
 We proceed to obtain an upper bound on $ \|\partial_v \Pi_x \|$, for $ x \in \bigcup_i \tilde{U}_i.$ Recall that this implies that $c < \tilde{\a}(x) < C$.
Recall that the radius of the circle $\gamma$ is $\frac{1}{2}.$ Thus, 
 \beq \|\partial_v \Pi_x \| & \leq & \left(\frac{1}{2}\right) \left\|\partial_v ((zI - M(x))^{-1})\right\|\label{eq:41}\\
 & = & \left(\frac{1}{2}\right) \|(zI - M(x))^{-1}\partial_v M(x) (z I - M)^{-1}\|\\
                                                            & \leq & \left(\frac{1}{2}\right) \|(zI - M(x))\|^{-2}\|\partial_v M(x)\|\\
                                                             & \leq & 8 \|\partial_v M(x)\|\\
                                                             & = & 8 \|\sum_{i \in [N_3]} \partial_v \a_i(x) (\Pi^i-\Pi^1) + \partial_v \sum_i \a_i(x) \Pi_1\|\\
                                                              & \leq & 8 \sum_{i \in [N_3]} |\partial_v \a_i(x)| \de + 0\\
                                                                & \leq &    8   \|(\partial_v \a_i(x))_{i \in[N_3]}\|_{\frac{d+k}{d+k-1}}\|(\de)_{i\in[N_3]}\|_{{d+k}}\\
                                                                  & \leq & C d^3\de,  \label{eq:48}\eeq
 where $C$ is an absolute constant.

Therefore, 
\beq \left \| \oint_\gamma \frac{dz}{2\pi i} \left(\sum_{i \in [N_3]} \a_i(x) \left(\partial_v ((zI - M(x))^{-1})\right) F_i(x)\right)\right\| \leq  Cd^3 \de.\eeq
 Finally, we  bound  (\ref{eq:33}) from above, 
  \beqs
  \nonumber & & \hspace{-1cm}\left\| \oint_\gamma \frac{dz}{2\pi i} \left(\sum_i  \left(\partial_v \a_i(x)\right)  (zI - M(x))^{-1} F_i(x)\right)\right\| \\
 &\leq & \nonumber \left\|\Pi_x \left(\sum_i (\partial_v \a_i(x)) (F_i(x) - F_1(x))\right) \right\|    +  \left\| \left(\sum_i \partial_v \a_i(x)\right)F_1(x)\right\|\\
& \leq &  \|\Pi_x\| \sum_i |\partial_v \a_i(x)| \| F_i(x) -F_1(x)\| + 0\\
\nonumber & \leq &    \|(\partial_v \a_i(x))_{i\in[N_3]}\|_{\frac{d+k}{d+k-1}}\|(F_i(x) -F_1(x))_{i\in[N_3]}\|_{{d+k}}\\
& \leq & Cd^3 \de.
\eeqs 
Therefore, \beq\label{eq:201-Mar-10} \left\| \partial_v \left(  \Pi_x F(x) \right) -  \Pi_x v\right\| \leq C d^3 \de.\eeq
Note also by (\ref{eq:41})-(\ref{eq:48}) that \beq\|\partial_v \Pi_x\| \leq Cd^3\de .\label{eq:3-176}\eeq

\subsection{A bound on the second derivative of $\Pi_x F(x)$}

We now proceed to obtain an upper bound on $\|\partial_v^2 \left(\Pi_x F(x)\right)\|.$
{ To this end, we use that} \beq \|\partial_v^2 \left(\Pi_x F(x)\right)\| & \leq & \| (\partial_v^2 \Pi_x) F(x)\|\label{eq:57}\\
& + & \|2 (\partial_v \Pi_x) \partial_v F(x)\|\label{eq:58} \\ 
&+ & \|\Pi_x \partial_v^2 F(x)\|.\label{eq:59}\eeq 

We first bound from above the right side of (\ref{eq:57}). To this end, we observe that
\beqs  (\partial_v^2 \Pi_x)  & = & \partial_v^2\left[\frac{1}{2\pi i}\oint_\gamma [zI - M(x)]^{-1}dz\right]\\
                                                       & = & \partial_v\left[\frac{1}{2\pi i} \oint_\gamma (zI - M(x))^{-1}\partial_v M(x) (zI - M(x))^{-1}dz\right]\\
                                                       \nonumber & = & \frac{1}{2\pi i} \oint_\gamma 2 (zI - M(x))^{-1}\partial_v M(x) (zI - M(x))^{-1} \partial_v M(x) (zI - M(x))^{-1}dz\\ 
& + &  \oint_\gamma (zI - M(x))^{-1}\partial^2_v M(x) (zI - M(x))^{-1}dz.\eeqs
Therefore, \beqs  \nonumber \| (\partial_v^2 \Pi_x) F(x)\| & \leq & \sup_{z \in \gamma} \left(C \|zI - M(x)\|^{-3}\|(\partial_v M(x))^2\| + C \|zI - M(x)\|^{-2}\|(\partial^2_v M(x))\|\right)\\
                                                         & \leq & C(\|\partial_v M(x) \|^2 + \|\partial_v^2 M(x)\|)\\
                                                        & \leq & Cd^6\de^2 +  C\|\partial_v^2 M(x)\| \\
                                                        & = & Cd^6\de^2 + C\|\partial_v^2 \sum_{i \in [N_3]} \a_i(x) \Pi_i\|\\
                                                        & \leq & Cd^6\de^2 + C\sum_{i \in [N_3]} |\partial_v^2 \a_i(x)(\Pi_i - \Pi_1)|\\
                                                         & \leq & Cd^6\de^2 +  C\|(\partial^2_v \a_i(x))_i\|_{\frac{d+k}{d+k-2}}\|(\de)_i\|_{\frac{d+k}{2}}\\
                                                          & \leq & Cd^6\de^2 + Cd^6\de.\eeqs

Next, we bound  (\ref{eq:58}) from above.
Note that \beq \|\partial_v F(x)\| & \leq & \|(\sum_i (\partial_v \a_i(x)(F_i(x) - F_1(x)))+\Pi_xv\|\\
& \leq & 1 + C d^3\de\eeq
and
\beq \|(\partial_v \Pi_x) \partial_v F(x)\| & \leq &  \|(\partial_v \Pi_x)\| \|\partial_v F(x)\|\\
                                                                  & \leq &   (C d^3 \de ) (1 + d^3 \de)\\
                                                                    & = & C d^3\de + C d^6 \de^2.\eeq

Finally, we bound (\ref{eq:59}) from above by observing that
\beq  \|\Pi_x \partial_v^2 F(x)\| & \leq &  \|\partial_v^2 F(x)\|\\
& \leq & \|\partial_v^2 (F(x) - F_1(x))\|\\
                                                      & \leq & \sum_i |\partial_v^2 \a_i(x)|\| F_i(x) - F_1(x)\|\label{eq:75}\\
                                                        & + & \sum_i 2| \partial_v \a_i(x)| \|\partial_v F_i(x)- \partial_v F_1(x)\label{eq:76}\|\\
                                                        & + & \sum_i |\a_i(x)| \|\partial_v^2 F_i(x)\|.\label{eq:77}
\eeq

We first bound (\ref{eq:75}) from above.
\beq \nonumber &&\hspace{-1.5cm} \sum_i |\partial_v^2 \a_i(x)|\| F_i(x)- F_1(x)\|\\
 & \leq & \| (|\partial_v^2 \a_i(x)|)_{i\in[N_3]}\|_{\frac{d+k}{d+k-2}}\|(\| F_i(x)- F_1(x)\|)_{i \in [N_3]}\|_{\frac{d+k}{2}}\\
                                                                        & \leq & (Cd^4)(d^2 \de)\\
                                                                          & = & Cd^6 \de.   
\eeq

Next we bound (\ref{eq:76}) from above.
\beqs \nonumber
& &\hspace{-1cm}\sum_i | \partial_v \a_i(x)| \|\partial_v F_i(x) - \partial_v F_1(x)\| \\
& \leq & \|(|\partial_v \a_i(x)|)_{i\in [N_3]}\|_{\frac{d+k}{d+k-1}} \|(\|\partial_v F_i(x)- \partial_v F_1(x)\|)_{i\in[N_3]}\|_{d+k}\\
                                                                             & \leq & C d^2 (d\de)\\
                                                                                & = & C d^3\de. \eeqs

{ Observe, that the term (\ref{eq:77}) is  equal to} $0$.
  Therefore, \beq\label{eq:207} \left\| \partial^2_v \left(  \Pi_x F(x) \right) \right\| \leq C d^6 \de.\eeq

Recall that $\MM_o$ is the set of points $x \in \cup_i  \widetilde{U}_i$ such that 
\beq \Pi_{hi}(\sum_i \a_i(x)\Pi^i)(\sum_i \a_i(x)\Pi^i(x - p_i)) = 0. \eeq
In particular, $x \in \MM_o \cap U_i$ if and only if $ h(z) = \Pi_i \Pi_{hi}(\sum_i \a_i(z)\Pi^i)(\sum_i \a_i(z)\Pi^i(z - p_i)) = 0,$ where $\Pi^i$ is the orthogonal projection onto the subspace orthogonal to $D_i$, containing the center of $D_i$. We take $U_i$ to be the unit ball and the center of $U_i$ to be the origin and take the linear span of $D_i$ to be $\R^d$.  We split $z$ into its $x$ component (projection onto $\R^d$) and $y$ component (projection orthogonal to $\R^d$), thus $z = (x, y) \in \R^d \times \R^{n-d}.$ We define $g(x, y) = (x, h(x, y))$. This function is then substituted into the quantitative inverse function theorem of Subsection~\ref{sec:quant}.

\subsection{Hausdorff distance of $\MM_o$ to $\MM$ and the reach of $\MM_o$.}
For this subsection, we choose a new length scale so that $\tau = 1.$
 Let $r = C\sqrt{d}\sigma.$ Suppose that $dist(X_3, \MM) < \de:= Cr^2$, and $dist(\MM, X_3) < cr$. These are the parameters for the refined net.
\begin{theorem}
The Hausdorff distance between $\MM_o$ and $ \MM$ is less than $\frac{Cd\sigma^2}{\tau}$. 
\end{theorem}
\begin{proof}
Since $dist(X_3, \MM) < \de:= Cr^2$, and $dist(\MM, X_3) < cr$   the Hausdorff distance between $\bigcup_i D_i$ and $\MM$ is less than $\de = Cr^2$ by Subsection~\ref{ssec:fine}.
The Hausdorff distance between $\bigcup_i D_i$ and $\MM_o$ is less than $C\de$ by the quantitative implicit function theorem (Subsection~\ref{sec:quant}), applying Taylor's theorem together with (\ref{eq:119a}) and (\ref{eq:122}). Thus, by the triangle inequality, 
the Hausdorff distance between $\MM_o$ and $\MM$ is less than $Cr^2.$
\end{proof}
Next, we address the reach of $\MM_o$.
By Federer's criterion (Proposition~\ref{thm:federer}) we know that 
\beqs reach(\MM_o) = \inf\left\{|b-a|^{2}(2 dist(b, Tan(a)))^{-1}\big| \, a, b \in \MM_o, a \neq b\right\}.\eeqs
Let $ a, b \in \MM_o, a \neq b$.

 If $|a - b| > \frac{1}{Cd^6}$, then $|b-a|^{2}(2 dist(b, Tan(a)))^{-1} > \frac{1}{Cd^6},$ because 
\beqs |b-a| \geq dist(b, Tan(a)). \eeqs

Therefore, we may suppose that $|a-b| \leq \frac{1}{Cd^6}$.
By the bound on the Hausdorff distance between $\MM$ and $\MM_o$, the distances of $a$ and $b$ to their projections onto $\MM$, which we denote $a'$ and $b'$ respectively, are less than $C\de$. By the quantitative implicit function theorem (Subsection~\ref{sec:quant}) and the covering property of $\{U_i\}$, $\MM_o$ is a $C^2-$submanifold of $\R^n$. Therefore $Tan(a)$ is a $d-$dimensional affine subspace. By  (\ref{eq:201-Mar-10}), (\ref{eq:119a}) and (\ref{eq:122}) the Hausdorff distance between the two unit discs $Tan(a) \cap B(a, r)$ and $(Tan_\MM(a') \cap B(a',  r)) + (a - a')$ which are centered at $a$, is bounded above by $ Cd^3\de.$ Therefore,
the Hausdorff distance between the two unit discs $Tan(a) \cap B(a, 1)$ and $(Tan_\MM(a') \cap B(a',  1)) + (a - a')$ which are centered at $a$, is bounded above by $\frac{Cd^3\de}{r}.$

Then, we have the following.
\begin{obs}\lab{obs:7.1} $\MM_o$ and $\MM$ are $\de$ close in Hausdorff distance and 
$(\MM_o \cap B(a, 2|a-b|))$ and $(\MM \cap B(a',  2|a-b|))$ are $C(d^3\de/r)$ close in $C^1$ as graphs of functions over $(Tan(a) \cap B(a, 3|a-b|/2)).$ 
\end{obs}
Let these functions  be respectively $\widehat{f}_o$ and $\widehat{f}$.  Note that the range is $Nor(a)$, the fiber of the normal bundle at $a$; please see Lemma~\ref{lem:6} in  Section~\ref{sec:GeomPrelim}. 
We know that the  $C^1$ norm of $\widehat{f}$ on $(Tan(a) \cap B(a, 3|a-b|/2))$ is at most $\frac{Cd^3\de}{r} + C|a-b|$. 
Therefore, the  $C^1$ norm of $\widehat{f}_o$ on $(Tan(a) \cap B(a, 3|a-b|/2))$ is at most $\frac{Cd^3\de}{r} + C|a-b|.$
But using this and the Hessian bound of $C d^6$ from (\ref{eq:207}), we also know that the Hessian of $\widehat{f}_o$ is bounded above by $C d^6$. But now, by Taylor's theorem,  $dist(b, Tan(a)) \leq \sup \|Hess \widehat{f}_o\||a-b|^2/2,$ where the supremum is taken over 
$(Tan(a) \cap B(a, 3/2|a-b|))$. This, we know is bounded above by $Cd^6|a-b|^2$. Substituting this into Federer's criterion for the reach, we see that $reach(\MM_o) \geq \frac{1}{Cd^6}.$
Thus, we have just proved the following.
\begin{theorem}
The reach of $\MM_o$ is at least $\frac{1}{Cd^{6}}$.
\end{theorem}

Finally, we provide an estimate on the third derivatives of $\MM_o$. Since our guarantee about the true manifold is only that it is $\C^2$, it is inevitable that as the Hausdorff distance between the output manifold and the true manifold tends to zero, the guarantees on the third derivatives of $\MM_o$ viewed as the graph of a function tend to infinity. The inverse dependence on $\sigma$ in the following lemma reflects that fact.

\begin{proposition}\label{thm:C3}
Let $a \in \MM_o$ and $\widehat{f}_o$ be a function from $(Tan(a) - a)$ to $Nor(a)$ such that $(\MM_o - a)$ agrees with the graph of $\widehat{f}_o$ in a $\frac{1}{20}-$neighborhood $U$ of $a$. Then,  for any  unit vector $v$ in the domain, and any unit vector $w$ in the range, the third derivative $\langle\partial^3_v \widehat{f}_o, w \rangle$ at $x \in U$ satisfies 
\beqs \big|\langle\partial^3_v \widehat{f}_o(x), w \rangle\big| \leq \frac{Cd^9}{r} \leq \frac{Cd^{8.5}}{\sigma} . \eeqs 
\end{proposition}
\begin{proof}
This follows from Lemma~\ref{lem:19-May27} and the quantitative implicit function theorem from Subsection~\ref{sec:quant_imp_func}.
\end{proof}

\begin{proposition}\lab{prop:final-diff}
When $\dhaus(\MM, \MM_o) <  c(\min({reach}(\MM), {reach}(\MM_o)))$, the output manifold $\MM_o$ is $C^1-$diffeomorphic to $\MM$.
\end{proposition}
\begin{proof} 
Consider the projection map $\pi_o$ from a tubular neighborhood of $\MM_o$ of thickness $\frac{\mathrm{reach}(\MM_o)}{2}$  on to $\MM_o$.  This tubular neighborhood contains $\MM.$  It suffices to show that $\pi_o|_\MM:\MM \ra \MM_o$ is bijective and is $C^1$ and has a $C^1$ inverse that maps $\MM_o$ to $\MM.$ Let  ${reach}(\MM_o)$ be denoted $\tau_o$.
Due to the fact that $\MM_o$ is $C^2$, $\pi_o$ is $C^1$, and further by Theorem 4.8, part (8) of \cite{federer_paper} is $2\tau_o^{-1}$ Lipschitz. Considering Lemma~\ref{lem:6} and Proposition~\ref{thm:federer}, a short argument shows that $\pi_o$ restricted to $\MM$ is bijective and its differential at a point $a' \in \MM$ is a linear map from $Tan(a', \MM)$ to $Tan(\pi_0(a'), \MM_o)$ whose singular values lie in the interval $(c, 1]$. This completes the proof of this Proposition.
\end{proof}
\section{Concluding remarks}

We have studied the problem of reconstructing a compact embedded $d$ dimensional $C^2$ submanifold $\MM$ of $\R^n$ from random samples. These random samples are obtained from sampling the manifold independently and identically at random from some density $\mu$ and adding Gaussian noise having a spherically symmetric distribution where the standard deviation of any component is $\sigma.$ If that the noise is smaller than a parameter specified in (\ref{eq:sigma-May}), we developed an algorithm that uses $O(\sigma^{-d-4})$ samples and produces a manifold $\MM_o$ whose reach is no more than $Cd^6$ times the reach of $\MM$ and whose Hausdorff  distance to $\MM$ is at most $\frac{Cd\sigma^2}{\tau}.$ If $n\sigma^2 > c\tau^2$, our bounds imply that the number of samples used is $O(\sigma^{-d - 2})$, which is optimal in the case when $d= 0,$ and the manifold is a point.

\section*{Acknowledgements}
We are deeply grateful to the anonymous reviewers for many helpful comments, including Remark~\ref{rem:1.3}, regarding the $C^1-$diffeomorphism equivalence of the output manifold to the unknown manifold.

{Ch.F. was partly supported by the  
US-Israel Binational Science Foundation, grant number 2014055, AFOSR, grant DMS-1265524, and NSF, grant FA9550-
12-1-0425. S.I. was partly supported RFBR, grant 20-01-00070, 
M.L. was supported by Academy
of Finland, grants 273979 and 284715, and H.N. was partly supported by NSF
grant DMS-1620102 and a Ramanujan Fellowship and a Swarna Jayanti fellowship.}

 \bibliographystyle{ACM}

\begin{thebibliography}{10}



\bibitem{aamari2019}
{ Aamari, E., and Levrard, C.}
\newblock Nonasymptotic rates for manifold, tangent space and curvature
  estimation.
\newblock {\em Ann. Statist. 47}, 1 (02 2019), 177--204.


\bibitem{Anderson}
M. Anderson,
{\em Convergence and rigidity of manifolds under Ricci curvature bounds},
Invent. Math. {\bf 102}  (1990), 429--445. 

\bibitem{AKKLT}  M. Anderson, A. Katsuda, Y. Kurylev, M. Lassas, M. Taylor,
\emph{Boundary regularity for the Ricci equation, geometric convergence, and
Gel'fand's inverse boundary problem}, Invent. Math. {\bf 158} (2004),  261--321.

\bibitem{AizenbudSober} Aizenbud, Y., \& Sober, B. (2021). Non-Parametric Estimation of Manifolds from Noisy Data. ArXiv, {\em abs/2105.04754.}

\bibitem{Boissonnat}
{ Boissonnat, J., Guibas, L.~J., and Oudot, S.}
\newblock Manifold reconstruction in arbitrary dimensions using witness
  complexes.
\newblock {\em Discrete {\&} Computational Geometry 42}, 1 (2009), 37--70.



\bibitem{Boucheron2005}
{ Boucheron, S., Bousquet, O. and Lugosi, G.},
 {Theory of classification : a survey of some recent advances},
 {\em ESAIM: Probability and Statistics}, 9 (2005),  {323--375}.




\bibitem{BN} M. Belkin, P. Niyogi, \emph{Laplacian eigenmaps and spectral techniques
for embedding and clustering}, Adv. in Neural Inform. Process. Systems, {\bf 14}   (2001), 586--691.

\bibitem{BN2} 
M. Belkin, P. Niyogi, {\it Semi-Supervised Learning on Riemannian
Manifolds}, Machine Learning, {\bf 56} (2004), 209--239.


\bibitem{BN1} M. Belkin, P. Niyogi, \emph{Convergence of Laplacian eigenmaps},  Adv.
in Neural Inform. Process. Systems {\bf 19}   (2007), 129--136.

\bibitem{BerNik}
V. Berestovskij, I. Nikolaev, 
{\em Multidimensional generalized Riemannian spaces},
In: Geometry IV, Encyclopaedia Math. Sci. {\bf 70}, Springer, 1993, pp. 165--243.

\bibitem{Beretta}
E.  Beretta, M. de Hoop, L. Qiu,
{\it Lipschitz Stability of an Inverse Boundary Value Problem for a Schr\"odinger-Type Equation}
SIAM J. Math. Anal. {\bf 45} (2012), 679-699.

\bibitem{BMP} E. Bierstone, P. Milman, W. Paulucki, \emph{Differentiable
functions defined on closed sets. A problem of Whitney}, Invent. Math., {\bf 151}
(2003), 329--352.

\bibitem{witness}
{J.\ Boissonnat, L.\, Guibas,  S.\ Oudot,}
{\it  Manifold reconstruction in arbitrary dimensions using witness
  complexes},
\newblock { Discrete $\&$ Computational Geometry 42\/} (2009), 37--70.

\bibitem{Borcea1}
L. Borcea, V. Druskin,
 F. Guevara Vasquez, {\it Electrical impedance tomography with resistor networks}
Inverse Problems {\bf 24} (2008), 035013. 

\bibitem{Borcea2}
L. Borcea, V. Druskin, L. Knizhnerman, 
{\it On the continuum limit of a discrete inverse spectral problem on optimal finite difference grids},
Comm.  Pure Appl. Math. {\bf 58} (2005), 1231-1279.

%

\bibitem{mc}
{M. Brand,}
\newblock {\it Charting a manifold,}
\newblock {\em NIPS 15} (2002), 985--992.

\bibitem{BrHa}
M. Bridson, A. Haefliger, 
{\em Metric spaces of non-positive curvature},
Springer-Verlag, 1999.

\bibitem{Brom} S. Bromberg, {\it An extension in the class $C^1$,} Bol. Soc. Mat. Mex. II, Ser. {\bf 27}, (1982),
35--44.

\bibitem{Br1} Y. Brudnyi, \emph{On an extension theorem,} Funk. Anal. i Prilzhen. 4 (1970), 97--98; English
transl. in Func. Anal. Appl. {\bf 4} (1970), 252--253.
\bibitem{Br2}
 Y. Brudnyi, P. Shvartsman, \emph{The traces of differentiable functions to closed subsets of
$\R^n$,}  in Function Spaces (1989), Teubner-Texte Math. {\bf 120}, 206--210.
\bibitem{Br3}
 Y. Brudnyi, P. Shvartsman, {\it A linear extension operator for a space of smooth functions
defined on closed subsets of $\R^n$,} Dokl. Akad. Nauk SSSR {\bf 280} (1985), 268--270.
English transl. in Soviet Math. Dokl. {\bf 31}, No. 1 (1985), 48--51.
\bibitem{Br4}
 Y. Brudnyi, P. Shvartsman, {\it Generalizations of Whitney's extension theorem,} Int.
Math. Research Notices {\bf 3} (1994), 129--139.


\bibitem{BS1} Y. Brudnyi, P. Shvartsman, \emph{The traces of differentiable
functions to closed subsets of $\R^n$}, Dokl. Akad. Nauk SSSR {\bf 289} (1985),
268--270.

\bibitem{BS2} Y. Brudnyi, P. Shvartsman, \emph{The Whitney problem of
existence of a linear extension operator}, J. Geom. Anal. {\bf 7}(1997), 515--574.


\bibitem{Br5}
 Y. Brudnyi, P. Shvartsman, {\it Whitney's extension problem for multivariate $C^{1,\omega}$ functions,}
Trans. Amer. Math. Soc. {\bf 353} No. 6 (2001), 2487--2512.






\bibitem{Bernstien} 
M. Bernstien, V. de Silva, J. Langford, J. Tenenbaum, {\it Graph approximations to geodesics on embedded manifolds}. Technical Report,
 Stanford University,  2000. 


\bibitem{BIK} D. Burago, S. Ivanov, Y. Kurylev, \emph{A graph discretisation of the
Laplace-Beltrami operator}, 
J. Spectr. Theory {\bf 4} (2014), 675--714.



\bibitem{chen2015}
{ Chen, Y.-C., Genovese, C.~R., and Wasserman, L.}
\newblock Asymptotic theory for density ridges.
\newblock {\em Ann. Statist. 43}, 5 (10 2015), 1896--1928.

\bibitem{Cheng}
{ Cheng, S., Dey, T.~K., and Ramos, E.~A.}
\newblock Manifold reconstruction from point samples.
\newblock In {\em Proceedings of the Sixteenth Annual {ACM-SIAM} Symposium on
  Discrete Algorithms, {SODA} 2005, Vancouver, British Columbia, Canada,
  January 23-25, 2005\/} (2005), pp.~1018--1027.


\bibitem{CB} D. Chigirev and W. Bialek, \emph{Optimal Manifold Representation of Data:
An Information Theoretic Approach}. In: Advances in Neural Information Processing Systems {\bf 16,}
 Ed. S.\ Thrun et al, The MIT press, 2004, pp. 164--168.


\bibitem{diffusion1}
{R.\ Coifman, et al.}
\newblock {\it Geometric diffusions as a tool for harmonic analysis and structure
  definition of data } {P}art {II}: Multiscale methods.
 { Proc. of Nat. Acad. Sci.} {\bf 102} (2005), 7432--7438.


  \bibitem{diffusion2}
{R.\ Coifman, et al.}
R. R. Coifman, S. Lafon, A. B. Lee, M. Maggioni, B. Nadler, F. Warner, and S. W. Zucker
\newblock {Geometric diffusions as a tool for harmonic analysis and structure
  definition of data } {P}art {II}: Multiscale methods.
 {\it Proc. of Nat. Acad. Sci.} {102} (2005), 7432--7438.
 
  \bibitem{CoifmanLafon}
R. Coifman,  S. Lafon, Diffusion maps. {\it Appl.  Comp. Harm. Anal.} 21  (2006), 5-30.
 

\bibitem{mds}
{T. Cox, M.\ and Cox}
\newblock {\em Multidimensional Scaling.}
\newblock Chapman $\&$ Hall, London, (1994).

\bibitem{Dasgupta2}
{S.\ Dasgupta, Y.\ Freund,}
\newblock {\it Random projection trees and low dimensional manifolds.}
\newblock In { Proc. the 40th ACM symposium on Theory of
  computing} (2008), STOC '08~537--546.

\bibitem{Divol} {V. Divol}
\newblock Reconstructing measures on manifolds: an optimal transport approach. {\it Arxiv eprints} page arXiv:2102.07595

\bibitem{donoho}
{D.\ Donoho, D.\ Grimes,}
\newblock {\it Hessian eigenmaps: Locally linear embedding techniques for
  high-dimensional data,}
\newblock { Proceedings of the National Academy of Sciences, {\bf 100}}, 
   5591--5596.


\bibitem{Donoho1}
D. Donoho, C. Grimes,
When does geodesic distance recover the true hidden parametrization of families of articulated images?
{\it Proceedings of ESANN 2002}, Bruges, Belgium,  2002.



   
\bibitem{Donoho}
D. Donoho, C. Grimes,  Image Manifolds which are Isometric to Euclidean Space. {\it J.\ Math.\ Im.\ Vis} 23 (2005), 5.



\bibitem{Fan-Truong}
{ Fan, J., and Truong, Y.~K.}
\newblock Nonparametric regression with errors in variables.
\newblock {\em Ann. Statist. 21}, 4 (12 1993), 1900--1925.

\bibitem{federer_paper}
{ Federer, H.}
\newblock Curvature measures.
\newblock {\em Transactions of the American Mathematical Society 93\/} (1959).

\bibitem{federer_book}
{ Federer, H.}
\newblock {\em Geometric measure theory}.
\newblock Springer, 2014.


\bibitem{F1} Ch. Fefferman, \emph{A sharp form of Whitney's extension theorem},
Ann. of Math. {\bf 161} (2005), 509--577.

\bibitem{F2} Ch. Fefferman, \emph{Whitney's extension problem for $C^m$},
Ann. of Math. {\bf 164} (2006), 313--359.

\bibitem{F3} Ch. Fefferman, \emph{$C^m$-extension by linear operators},
Ann. of Math. {\bf 166} (2007), 779--835.



\bibitem{F5}
Ch. Fefferman, {\it A generalized sharp Whitney theorem for jets, Rev. Mat. Iberoam.} {\bf 21,} No. 2 (2005), 577--688.

\bibitem{F6}
Ch. Fefferman, Extension of $ C^{m,\omega}$  {\it smooth functions by linear operators,} Rev. Mat. Iberoam. {\bf 25,} No. 1 (2009), 1--48.


\bibitem{FK1} Ch. Fefferman, B. Klartag, \emph{Fitting $C^m$-smooth function to
data I}, Ann. of Math, {\bf 169} (2009), 315--346.


\bibitem{FK2} Ch. Fefferman, B. Klartag, \emph{Fitting $C^m$-smooth function to
data II}, Rev. Mat. Iberoam. {\bf 25} (2009), 49--273.


\bibitem{putative}
{ Fefferman, C., Ivanov, S., Kurylev, Y., Lassas, M., and Narayanan, H.}
\newblock Fitting a putative manifold to noisy data.
\newblock In {\em Proceedings of the 31st Conference On Learning Theory\/}
  (06--09 Jul 2018), S.~Bubeck, V.~Perchet, and P.~Rigollet, Eds., vol.~75 of
  {\em Proceedings of Machine Learning Research}, PMLR, pp.~688--720.

\bibitem{FIKLN}
{ Fefferman, C., Ivanov, S., Kurylev, Y., Lassas, M., and Narayanan, H.}
\newblock Reconstruction and interpolation of manifolds I: The geometric
  whitney problem.
\newblock {\em Foundations of
  Computational Mathematics\/}. Preprint arXiv:1508.00674.

\bibitem{FIKLN2}
{ Fefferman,  C., Ivanov, S., Lassas, M.,  and Narayanan, H.}
\newblock Reconstruction of a Riemannian Manifold from Noisy Intrinsic Distances
\newblock SIAM Journal on Mathematics of Data Science 2020 2:3, 770-808 

\bibitem{FMN}
{ Fefferman, C., Mitter, S., and Narayanan, H.}
\newblock Testing the manifold hypothesis.
\newblock {\em Journal of the American Mathematical Society 29}, 4 (2016),
  983--1049.

\bibitem{Fu88}
K. Fukaya,
{\em A boundary of the set of the Riemannian manifolds with bounded curvatures and diameters}
J. Differential Geom. {\bf 28} (1988), No. 1, 1--21. 

\bibitem{Wasserman}
{ Genovese, C.~R., Perone-Pacifico, M., Verdinelli, I., and Wasserman, L.}
\newblock Manifold estimation and singular deconvolution under hausdorff loss.
\newblock {\em Annals of Statistics 40}, 2 (2012).

\bibitem{Genovese:2012:MME:2188385.2343687}
{ Genovese, C.~R., Perone-Pacifico, M., Verdinelli, I., and Wasserman, L.}
\newblock Minimax manifold estimation.
\newblock {\em J. Mach. Learn. Res. 13\/} (May 2012), 1263--1291.

\bibitem{ridge}
{ Genovese, C.~R., Perone-Pacifico, M., Verdinelli, I., and Wasserman, L.}
\newblock Nonparametric ridge estimation.
\newblock {\em Ann. Statist. 42}, 4 (2014), 1511--1545.

\bibitem{G} G. Glaeser, \emph{Etudes de quelques algebres Tayloriennes},
J. d'Analyse {\bf 6} (1958), 1--124.

\bibitem{Gr} M. Gromov with appendices by M. Katz, P. Pansu, and S. Semmes,
 {\it Metric Structures for Riemannian and Non-Riemanian
Spaces.} Birkhauser (1999).


\bibitem{Hein}
{ Hein, M.,  Maier, M. }
\newblock Manifold denoising. 
\newblock {\em In Advances in neural information processing systems} (pp. 561-568).


\bibitem{PCA2}
{H.\ Hotelling,}
\newblock Analysis of a complex of statistical variables into principal
  components.
\newblock {\em Journal of Educational Psychology}  {\bf 24} (1933), 417--441,
  498--520.

\bibitem{Iversen}
E. Iversen, M. Tygel, B. Ursin, and M. V. de Hoop, {\it Kinematic
time migration and demigration of re
ections in pre-stack seismic data},
Geophys. J. Int. {\bf 189} (2012), 1635--1666.

\bibitem{heat}
{P.\ Jones, M., Maggioni, R.\ Schul,}
\newblock {\it Universal local parametrizations via heat kernels and eigenfunctions
  of the laplacian,}
\newblock { Ann. Acad. Scient. Fen. {\bf 35}} (2010), 1--44.

%




\bibitem{KKL} 
A. Katchalov, Y. Kurylev, M. Lassas:
{\it Inverse Boundary Spectral Problems}, Monographs and Surveys
in Pure and Applied Mathematics {\bf 123}, CRC-press, 2001,
xi+290 pp.


\bibitem{KatKL} A. Katsuda, Y. Kurylev, M. Lassas, \emph{Stability and Reconstruction in
Gel'fand Inverse Boundary Spectral Problem}, in: New analytic and geometric
methods in inverse problems. (Ed. K. Bingham, Y. Kurylev, and E.
Somersalo), 309--320, Springer-Verlag, 2003.



\bibitem{kim2015}
{ Kim, A. K.~H., and Zhou, H.~H.}
\newblock Tight minimax rates for manifold estimation under hausdorff loss.
\newblock {\em Electron. J. Statist. 9}, 1 (2015), 1562--1582.


\bibitem{Meila}  D. Perraul-Joncas, M. Meila,
Non-linear dimensionality reduction: Riemannian metric estimation and the problem of geometric discovery, arXiv:1305-7255, 2013.





\bibitem{Kress}
R. Kress,
\emph{Numerical analysis}.  Springer-Verlag, 1998. xii+326 pp.



\bibitem{LU} M.\ Lassas,  G.\
Uhlmann, {\it Determining Riemannian manifold
from boundary measurements,}
{Ann. Sci.
\'Ecole Norm. Sup.} {\bf 34} (2001),
771--787.


\bibitem{LeU} J.\ Lee, G.\ Uhlmann, {\it Determining
anisotropic real-analytic conductivities by boundary measurements,}
{Comm. Pure Appl. Math.} {\bf 42} (1989), 1097--1112.



\bibitem{MCT} L. Ma, M.\ Crawford, J.\ W.\ Tian,  \emph{Generalised supervised local
tangent space alignment for hyperspectral image classification}, Electronics Letters {\bf 46} (2010), 497.

\bibitem
{MS}
J. Mueller, S. Siltanen, {\it Linear and nonlinear inverse problems with practical applications}. SIAM, Philadelphia, 2012. xiv+351 pp.





\bibitem{Nash1}
J. Nash, {\it $C^1$-isometric imbeddings}, Ann. of Math. {\bf 60} (1954), 383--396.

\bibitem{Nash2}
J. Nash,  {\it The imbedding problem for Riemannian manifolds}, Ann. of Math. {\bf 63} (1956), 20--63.



\bibitem{Narasimhan}
{ Narasimhan, R.}
\newblock {\em Lectures on Topics in Analysis}.
\newblock Tata Institute of Fundamental Research, Bombay, 1965.

%



\bibitem{Ozertem11}
{ Ozertem, U., and Erdogmus, D.}
\newblock Locally defined principal curves and surfaces.
\newblock {\em Journal of Machine Learning Research 12\/} (2011), 1249--1286.


\bibitem{PSU}
G.\ Paternain, M.\ Salo, G.\ Uhlmann 
{\it Tensor Tomography on Simple Surfaces,}
Inventiones Math. {\bf 193} (2013), 229--247. 

\bibitem{PU}
L.\ Pestov, G.\ Uhlmann,
{\it Two Dimensional Compact Simple Riemannian manifolds are Boundary Distance Rigid,}
Ann. of Math. {\bf 161} (2005), 1089--1106.

\bibitem{PCA1}
{K.\ Pearson,}
\newblock {\it On lines and planes of closest fit to systems of points in space,}
\newblock {Philosophical Magazine {\bf 2}} (1901), 559--572.


\bibitem{Peters}
S. Peters,
{\it Cheeger's finiteness theorem for diffeomorphism classes of Riemannian manifolds},
J. Reine Angew. Math. {\bf 349} (1984), 77--82. 

\bibitem{Pe}
P. Petersen, {\it Riemannian geometry}. 2nd Ed. 
Springer, (2006), xvi+401 pp.



\bibitem{RBBK} G. Rosman, M. M. Bronstein, A. M. Bronstein, R. Kimmel,
\emph{Nonlinear Dimensionality Reduction by Topologically Constrained Isometric Embedding}, International Journal of Computer Vision, {\bf  89} (2010), 56--68.


\bibitem{RS} S. Roweis, L. Saul, \emph{Nonlinear dimensionality reduction by locally
linear embedding}, Science, {\bf 290} (2000), 2323--326.

\bibitem{llc}
{S.\ Roweis, L.\ Saul, G.\ Hinton,}
\newblock {\it Global coordination of local linear models,}
\newblock { Advances in Neural Information Processing Systems {\bf 14}} (2001)
  889--896.


\bibitem{Newtons}
V. Ryaben'kii, S. Tsynkov, {\em A Theoretical Introduction to Numerical Analysis}, CRC Press,
2006, 537 pp. 

\bibitem{Sakai}
T. Sakai, {\it Riemannian geometry.} AMS, (1996), xiv+358 pp.


\bibitem{kernel}
J.\ Shawe-Taylor, N.\ Christianini,
\newblock {\em Kernel Methods for Pattern Analysis,}
\newblock Cambridge University Press, (2004).




\bibitem{Shv1}
P. Shvartsman, {\it Lipschitz selections of multivalued mappings and traces of the Zygmund
class of functions to an arbitrary compact,} Dokl. Acad. Nauk SSSR {\bf 276} (1984), 559--562; English transl. in Soviet Math. Dokl. {\bf 29} (1984), 565--568.


\bibitem{Shv2}
P. Shvartsman, {\it On traces of functions of Zygmund classes,} Sibirskyi Mathem. J. {\bf 28} (1987), 203--215; English transl. in Siberian Math. J. {\bf 28} (1987), 853--863.


\bibitem{Shv3}
P. Shvartsman, {\it Lipschitz selections of set-valued functions and Helly's theorem,} J. Geom. Anal. {\bf 12}  (2002), 289--324.


%

\bibitem{Sober}
{ Sober B. and Levin, D.},
	{Manifold Approximation by Moving Least-Squares Projection (MMLS)},
{\em Constructive Approximation},
	{2019}





\bibitem{Tao2011AnEO}
{ Tao, T.}
\newblock An epsilon of room, ii: pages from year three of a mathematical blog.


\bibitem{TSL} J.\ Tenenbaum, V. de Silva,  J.\ Langford, \emph{A global geometric framework for nonlinear dimensionality reduction}, Science,  {\bf 290}  5500 (2000),  2319--2323.



 \bibitem{Verma1}	 N. Verma,
Distance preserving embeddings for general n-dimensional manifolds.
(aka An algorithmic realization of Nash's embedding theorem),
\emph{Journal of Machine Learning Research} 23 (2012) 32.1--32.28.
 
\bibitem{Vaidya}
{ Vaidya, P.~M.}
\newblock A new algorithm for minimizing convex functions over convex sets.
\newblock {\em Mathematical Programming 73}, 3 (1996), 291--341.

\bibitem{vershynin_2018}
{ Vershynin, R.}
\newblock {\em High-Dimensional Probability: An Introduction with Applications
  in Data Science}.
\newblock Cambridge Series in Statistical and Probabilistic Mathematics.
  Cambridge University Press, 2018.


\bibitem{MaximumVariance}
{K.\ Weinberger, L.\ Saul,}
\newblock {\it Unsupervised learning of image manifolds by semidefinite programming,}
\newblock { Int. J. Comput. Vision {\bf 70}}, 1 (2006), 77--90.

\bibitem{W1} H. Whitney, {\it Analytic extensions of differentiable
functions defined on closed  sets}, Trans. Amer. Math. Soc., {\bf 36} (1934), 63--89.

\bibitem{W11} H. Whitney, {\it Differentiable functions defined in closed sets I,} Trans. Amer. Math. Soc.
{\bf 36} (1934), 369--389.

\bibitem{W12} H. Whitney, { \it Functions differentiable on the boundaries of regions,} Ann. of Math. {\bf 35}
(1934), 482--485.


\bibitem{W2}
H. Whitney, {\it Differentiable manifolds}, Ann. of Math. {\bf 37} (1936), 645--680.


\bibitem{W3}
H. Whitney, J. Eells, D. Toledo, eds., \emph{The collected papers of Hassler Whitney}. Volumes I-II., Contemporary Mathematicians,  Birkh\"auser, (1992).

\bibitem{zhigang}
Zhigang Yao and Yuqing Xia, 
{\it Manifold Fitting under Unbounded Noise,}
https://arxiv.org/abs/1909.10228


   \bibitem{Zha}
H. Zha and Z. Zhang,
Continuum Isomap for manifold learnings.
{Comp. Stat. Data Anal.} 52 (2007), 184-200.


\bibitem{ZZ} Z. Zhang and H. Zha, \emph{Principal manifolds and nonlinear dimension
reduction via local tangent space alignment}, SIAM J. Sci. Computing,
{\bf 26} (2005), 313--338.

\bibitem{ZO1}
N. Zobin, {\it Whitney's problem on extendability of functions and an intrinsic metric}, Advances in Math. {\bf 133} (1998), 96--132.

\bibitem{ZO2}
N. Zobin, {\it Extension of smooth functions from finitely connected planar domains}, J.
Geom. Anal. {\bf 9}  (1999), 489--509.


\end{thebibliography}

\appendix
\section{Some basic lemmas}

\begin{lemma} \label{cl:g1sept}
Suppose that $\MM \in \G(d, m, V, \tau)$. Let \beqs U:= \{y\in \R^m\big||y-\Pi_xy| \leq \tau/4\} \cap  \{ y\in \R^m\big||x-\Pi_xy| \leq \tau/4\}.\eeqs 
Then, $$\Pi_x(U \cap \MM) = \Pi_x(U).$$
\end{lemma}
\begin{proof}
Without loss of generality, we will assume $\tau/2 = 1$, and $x = 0$, and $Tan(x) = \R^d$.
Let $\NN = U \cap \MM$. We will first show that $\Pi_0(\NN) = B_d$, where $B_d$ is the closed unit ball in $\R^d.$ Suppose otherwise, then let $\emptyset \neq Y := B_d \setminus \Pi_0(\NN)$. Note that $\NN$ is closed and bounded and is therefore compact. The image of a compact set under a continuous map is compact, therefore $\Pi_0(\NN)$ is compact. Therefore $\R^d \setminus \Pi_0(\NN)$ is open. Let $x_1$ be a point of minimal distance from $0 = \Pi_0(0) \subseteq \Pi_0(\NN)$ among all points in the closure $Z$ of $B_d \setminus \Pi_0(\NN)$. In order to prove this lemma, it suffices to show  \beq\label{eq:38-11am-May24-2019} |x_1|\geq \frac{1}{2}.\eeq Since $Y \neq \emptyset$ and  $B_d \setminus \Pi_0(\NN)$ is open relative to $B_d$, \beq \label{eq:38-May24-2019} |x_1| < 1.\eeq Since $Tan(0)= \RR^d$ and $\MM$ is a closed  imbedded $C^2-$submanifold, $0$ does not belong to $Z$. Therefore $x_1 \neq 0$. By Federer's criterion for the reach, (\ie Corollary~\ref{cor:1}) $\forall y_1 \in \Pi_0^{-1}(x_1) \cap \NN$, \beq dist(y_1, Tan(0)) \leq \frac{\|y_1\|^2}{4}.\eeq
Therefore, $\forall y_1 \in \Pi_0^{-1}(x_1) \cap \NN$, \beq dist(y_1, \R^d) \leq \frac{dist(y_1, \R^d)^2 + |x_1|^2}{4} .\eeq
Noting that $2 \geq 1 \geq dist(y_1, \R^d)$ and solving the above quadratic inequality, we see that \beq |y_1 - x_1|/2 & \leq & 1 - \sqrt{ 1 - \left(\frac{|x_1|}{2}\right)^2} \leq   \left(\frac{|x_1|}{2}\right)^2.\eeq

This implies that \beq\label{eq:54} |y_1 - x_1| \leq \frac{1}{8} < \frac{\tau}{8}. \eeq
Again by Federer's criterion, for any $z \in \Pi_0^{-1} (|x_1| B_d) \cap \NN$, 
\beq |z -\Pi_{y_1}(z)|/2 & \leq & 1 - \sqrt{1 - \left(\frac{|y_1 - \Pi_{y_1} z|}{2}\right)^2} \leq  \left(\frac{|y_1 - \Pi_{y_1} z|}{2}\right)^2. \eeq

By (\ref{eq:38-May24-2019}),
there exists no neighborhood $ V \subseteq \R^m$ of $y_1$, such that there exists an open set $U_0 \subseteq \R^d$ containing $x_1$ and a $\C^2$ function $F: U_0 \ra \R^{m-d}$ with $DF(u)$ of rank $d$ for all $u \in U_0$ such that 
\beq\label{eq:ast2.new} \NN \cap V = \{(u, F(u))\big| u \in U \cap \R^d\}. \eeq

 Therefore, we have the following.

\begin{claim}\label{cl:1} Let $y_1 \in \Pi_0^{-1}(x_1) \cap \NN$. Then there exists $v \in \partial B_d$ such that  if $y_1' \in Tan(y_1)$ then $\langle y_1'-y_1, v\rangle = 0$.
\end{claim}
The cause behind the preceding claim is that the only way $\NN$ is locally not the graph of a function over a neighborhood contained in the interior of $B_d$, is if such a vector $v$ exists.

Let $\ell = \{\la v| \la \in \R\}$ and let $\Pi_\ell$ denote the orthogonal projection on to $\ell$. Then, $$\Pi_\ell(|x_1| B_d) = \{\la v | \la \in [-|x_1|, |x_1|]\}.$$ By Claim~\ref{cl:1}, $\Pi_\ell(Tan(y_1))$ is the single point $\Pi_\ell(y_1)$. Let $\Pi_\ell(y_1) = \la_0 v$. Let $x_2 = |x_1| v$ if $\la_0 \leq 0$ and $x_2 = -|x_1|v$ if $\la_0 > 0$. Let $y_2 \in \Pi_0^{-1}(x_2) \cap \NN$. Note that $(y_2-x_2)$ and   $(x_1 - y_1)$ are both orthogonal to $(x_2 - x_1)$, which will be used to obtain (\ref{eq:51-May24-2019}) below.
Then,
\beq  |x_1| & \leq & |\Pi_\ell(y_1) - x_2|\\
 & \leq &  dist(y_2, Tan(y_1))\\
                            & \leq &  \frac{|y_2 - y_1|^2}{4}\\
                             & \leq & \frac{2|y_2 - x_2|^2 + |x_1 - x_2|^2 +2 |y_1 - x_1|^2}{4} \label{eq:51-May24-2019}\\
                            & \leq & \frac{2\left(\frac{|x_2|^2}{2}\right)^2 + 4 |x_1|^2 + 2\left(\frac{|x_1|^2}{2}\right)^2}{4}.\eeq

Therefore, $$\a := |x_1| \leq |x_1|^4/4 + |x_1|^2.$$
Therefore, $1 \leq  \a^3/4 + \a$. This implies that $|x_1|  >  \frac{1}{2},$ which proves Lemma~\ref{cl:g1sept} (see (\ref{eq:38-11am-May24-2019})).

\end{proof}

\begin{figure}\label{fig:H}
\centering
\includegraphics[scale=0.60]{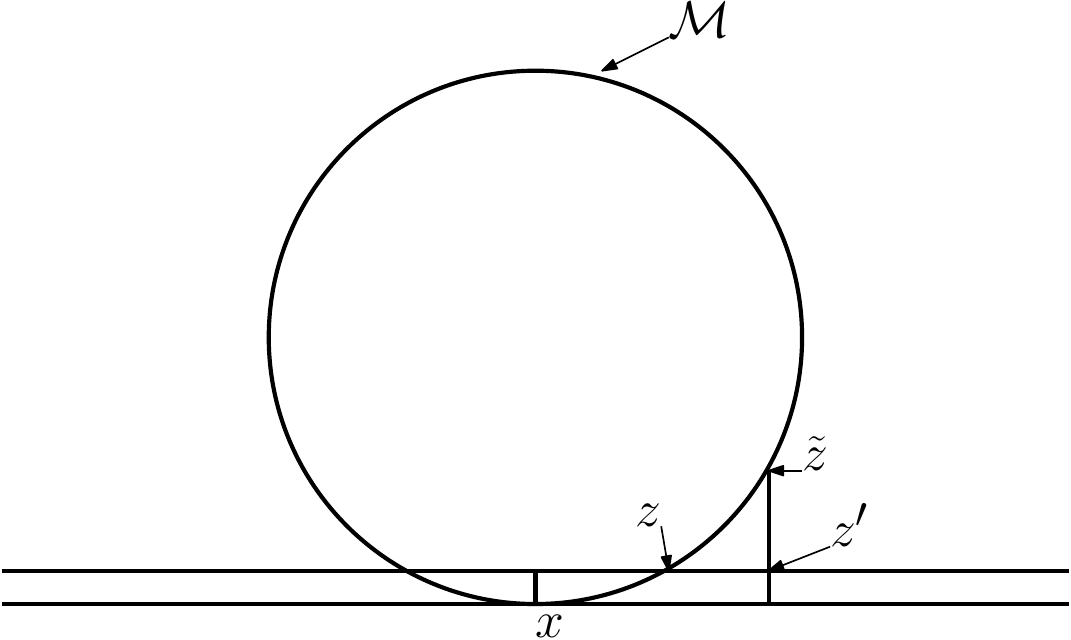}
\end{figure}
\begin{lemma}\label{lem:6}
Suppose that $\MM \in \G(d, m, V, \tau)$. Let $x \in \mathcal \MM$ and \beqs \widehat U:= \{y\in \R^m\big||y-\Pi_xy| \leq \tau/8\} \cap  \{y \in \R^m\big||x-\Pi_xy| \leq \tau/8\}.\eeqs 
There exists a $C^{2}$ function $F_{x, \widehat U}$ from $\Pi_x( \widehat U)$ to $\Pi_x^{-1}(\Pi_x(0))$ such that
\beqs  \{ y + F_{x, \widehat U}(y) \big | y \in \Pi_x(\widehat U)\} = \MM \cap \widehat U.\eeqs 
Secondly,  for $ \de \leq \tau/8$,
 let $z \in \MM \cap \widehat U$ satisfy $|\Pi_x(z) - x| = \de.$ Let $z$ be taken to be the origin and let the span of the first $d$ canonical basis vectors be denoted $\R^d$ and let $\R^d$ be a translate of $Tan(x)$. Let the span of the last $m-d$ canonical basis vectors be denoted $\R^{m-d}$. In this coordinate frame, let a point $z' \in\R^m$ be represented as  $(z'_1, z'_2)$, where $z'_1 \in \R^d$ and $z'_2 \in \R^{m-d}$. 
 By Lemma~\ref{cl:g1sept}, there exists an $(m-d) \times d$ matrix $A_z$ such that  \beq \label{eq:36.1} Tan(z) = \{(z'_1, z'_2)| A_z z'_1 - I z'_2 = 0\}\eeq where the identity matrix is $(m-d) \times (m-d)$. Let $z\in \MM \cap \{z\big||z-\Pi_xz| \leq \de\} \cap \{z\big||x-\Pi_xz| \leq \de\}$. Then $\|A_z\|_2 \leq 15 \de/\tau.$ Lastly, the following upper bound  on the second derivative of $F_{x, \widehat U}$ holds for  $y \in \Pi_x( \widehat U)$. $$\forall_{v \in \R^d} \forall_{w \in \R^{m-d}}\ \ \langle \partial_v^2 F_{x, \widehat U}(y), w\rangle \leq 
\frac{C|v|^2|w|}{\tau}.$$
\end{lemma}
\begin{proof}
We will first show that there exists a function $F_{x, \widehat U}$ that satisfies the given conditions and then show that it is $C^{2}$. Let $z \in \MM \cap \widehat U$ satisfy $|\Pi_x(z) - x| = \de.$ Let $z$ be taken to be the origin and let the span of the first $d$ canonical basis vectors be denoted $\R^d$ and let $\R^d$ be a translate of $Tan(x)$. Let the span of the last $m-d$ canonical basis vectors be denoted $\R^{m-d}$. In this coordinate frame, let a point $z' \in\R^m$ be represented as  $(z'_1, z'_2)$, where $z'_1 \in \R^d$ and $z'_2 \in \R^{m-d}$. By Lemma~\ref{cl:g1sept}, there exists a matrix $A$ such that  \beq \label{eq:36} Tan(z) = \{(z'_1, z'_2)| Az'_1 - I z'_2 = 0\}.\eeq
Further, a linear algebraic calculation shows that \beq \label{eq:57-May25} dist(z', Tan(z)) = \bigg|(I + A A^T)^{-1/2}(Az_1' - I z_2')\bigg|.\eeq Let $S^{d-1}_\de$ denote the $(d-1)$ dimensional sphere of radius $\de$ centered at the origin contained in $\R^d$. By Lemma~\ref{cl:g1sept},  for every $z' \in S^{d-1}_\de$ there is a point $\widetilde z \in \MM$, such that $\widetilde z \in U$, $\Pi_x \widetilde z = \Pi_x z'$ and 
\beq \left|\widetilde z - \Pi_x \widetilde z \right| \leq \frac{\left|x - \Pi_x \widetilde z \right|^2}{\tau} \leq \frac{4 \de ^2}{\tau}.\eeq The last inequality holds because $$\left|x - \Pi_x \widetilde z \right| \leq \left|x - \Pi_x z \right| + \left|\Pi_x z - \Pi_x \widetilde z\right| = 2\de.$$
Therefore, denoting $x$ by $(x_1, x_2)$, where $x_1 \in \R^d$ and $x_2 \in \R^{m-d}$, we have $|x_2| \leq \frac{|x_1|^2}{\tau} = \frac{\de^2}{\tau}$, and so \beq \left|\widetilde z - z'\right| = \left| \widetilde z - ((\Pi_x \widetilde z) -x_2)\right| \leq  \frac{4 \de ^2}{\tau} +  \frac{ \de ^2}{\tau} =  \frac{5 \de ^2}{\tau}. \eeq
Therefore, 
\beq dist(z', Tan(z)) & \leq & dist(\widetilde z, Tan(z)) +  \left|\widetilde z - z'\right| \\ & \leq & \frac{|z - \widetilde z|^2}{\tau} + \frac{5\de^2}{\tau}\\
& = & \frac{|z-z'|^2 + |z' - \widetilde z|^2}{\tau} + \frac{5\de^2}{\tau}\\ & \leq & \frac{\de^2 + (5\de^2/\tau)^2}{\tau} + \frac{5\de^2}{\tau}. \eeq
Therefore, for any $z_1' \in S^{d-1}_\de$, 
\beq \left|(I + A A^T)^{-1/2}(Az_1')\right| \leq \frac{\de^2}{\tau}\left(6 + \frac{25 \de^2}{\tau^2}\right)\leq \frac{\de^2}{\tau}\left(6 + \frac{64 \de^2}{\tau^2}\right). \eeq
Thus,
\beq\label{eq:65-May24} \left\|(I + A A^T)^{-1/2}A\right\|_2 \leq  \frac{\de}{\tau}\left(6 + \frac{64 \de^2}{\tau^2}\right) =:  \de'. \eeq
Therefore,  we see that
\beq  \left\|(I + A A^T)^{-1/2}AA^T (I + A A^T)^{-1/2}\right\|_2 \leq \de'^2.\eeq
Let $\|A\|_2 = \la$. We then see that $\la^2$ is an eigenvalue of $AA^T$. Therefore, 
$ \frac{\la^2}{1+\la^2} \leq \de'^2.$
This gives us $\la^2 \leq \frac{\de'^2}{1 - \de'^2 }$, which implies that
\beq\label{eq:la} \la \leq \frac{\de'}{\sqrt{1 - \de'^2}}. \eeq

We will use this to show that  $\Pi_x^{-1}(\Pi_x z) \cap \MM \cap U$ contains the single point $z$.
Suppose to the contrary, there is a point $\widehat z \neq z$ that also belongs to  $ \Pi_x^{-1}(\Pi_x z) \cap \MM \cap U$. Then, \beq dist(\widehat z, Tan(z)) \leq |\widehat z - z| \leq \frac{|\widehat z_2|^2}{\tau},\eeq where $|\widehat z_2| \leq |\Pi_x \widehat z| + |\Pi_x z| \leq 2 \de^2/\tau$. Thus, $$\frac{|\widehat z_2|}{\|I+ A A^T\|_2^{1/2}} \leq \frac{|\widehat z_2|^2 }{\tau}.$$ Therefore,
\beq 1 \leq |\widehat z_2|(1/\sqrt{(1 - \de'^2)})/\tau \leq \frac{2 \de^2}{\tau^2\sqrt{1 - \de'^2}}. \eeq
Hence, \beq 1-\de'^2 \leq \frac{2\de^2}{\tau^2},\eeq and so 
$ \de' \geq 1 - 2\de^2/\tau^2. $
Assuming $\de \leq \frac{\tau}{8}$ we infer from (\ref{eq:65-May24}) that 
\beq \label{eq:73-May24} \de' \leq \frac{7\de}{\tau}. \eeq

Therefore $$2 \de^2/\tau^2 + 7\de/\tau \geq 1.$$ This implies that $\de /\tau > 1/8$. This is a contradiction. This proves that $\Pi_x^{-1}(\Pi_x z) \cap \MM \cap U$ contains the single point $z$.
Further, (\ref{eq:la}) and (\ref{eq:73-May24}) together imply that $\|A_z\|_2 \leq \frac{56\de}{\sqrt{15}\tau} < \frac{15\de}{\tau}.$
Since $\MM$ is a $C^2$ submanifold of $\R^n$, by (\ref{eq:36}), $F_{x, \widehat U}$ is $C^2$. For $y\in \Pi_x (\widehat{U})$, we shall obtain the following upper bound  on the second derivative $$\forall_{v \in \R^d} \forall_{w \in \R^{n-d}} \ \ \langle \partial_v^2 F_{x, \widehat U}(y), w\rangle \leq 
\frac{C|v|^2|w|}{\tau}$$ below. Given $v \in \R^d$, let $z = y + F_{x, \widehat{U}}(y)$.

Let $\widetilde z = (y + \eps v, F_{x, \widehat{U}}(y+ \eps v))$. { Then,}
\beqs  & &\hspace{-2cm}\left|{F_{x, \widehat{U}}(y+ \eps v) - F_{x, \widehat{U}}(y) - \eps A_z v}\right|\eps^{-2} \\
& = & \left|(y + \eps v, F_{x, \widehat{U}}(y+ \eps v)) -  (y + \eps v, F_{x, \widehat{U}}(y) +  \eps A_z v)\right|\eps^{-2}\\
& = &  \|I+ A_z A_z^T\|_2^{1/2}|dist(\widetilde z, Tan(z))|\eps^{-2}\\
& \leq & \|I+ A_z A_z^T\|_2^{1/2}\left(\frac{|z - \widetilde z|^2}{2\tau}\right)\eps^{-2}\\
& \leq & \|I+ A_z A_z^T\|_2^{1/2}\left(\frac{\|A_z\|^2 |v|^2}{2\tau}\right) \leq  \frac{C |v|^2}{\tau}.
\eeqs
This yields
\beq \partial_v^2 F_{x, \widehat U}(y) = \lim_{\eps \ra 0} \left(\frac{F_{x, \widehat{U}}(y+ \eps v) - F_{x, \widehat{U}}(y) - \eps A_z v}{\eps^2}\right).\eeq
Therefore, 
$$\forall_{v \in \R^d} \left| \partial_v^2 F_{x, \widehat U}(y)\right| \leq 
\frac{C|v|^2}{\tau},$$ implying by Cauchy-Schwartz that
$$\forall_{v \in \R^d} \forall_{w \in \R^{m-d}}\ \  \langle \partial_v^2 F_{x, \widehat U}(y), w\rangle \leq 
\frac{C|v|^2|w|}{\tau}.$$
\end{proof}

\section{Proofs of Lemmas \ref{lem:find-disc}, \ref{lem:5.3-diff-eq}, \ref{lem:11-May27} and \ref{lem:12-May}}\lab{sec:subsidiary}

\begin{proof}[Proof of Lemma~\ref{lem:find-disc}]
Without loss of generality, let $x$ be the origin. Let $d(x, y)$ be used to denote $|x-y|$.
We will first show that for all $m \leq d-1$,
$$\max\left(|1-d(x, x_{m+1})|, \left | \left\langle \frac{x_1}{|x_1|}, x_{m+1}\right\rangle \right |, \dots, \left | \left\langle \frac{x_m}{|x_m|}, x_{m+1}\right\rangle \right | \right) < \delta.$$

To this end, we observe that the minimum over  $D_1(x)$ of
\beq\max\left(|1-d(x, y)|, \left | \left\langle \frac{(x_1)}{|x_1|}, y\right\rangle \right |, \dots, \left | \left\langle \frac{(x_m)}{|x_m|}, y\right\rangle \right | \right)\label{eq:Hari1}\eeq
is $0$, because the dimension of $D_1(x)$ is $d$ and there are only $m \leq d-1$ linear equality constraints. Also, the radius of $D_1(x)$ is $1$, so $|1 - d(x, z_{m+1})|$ has a value of $0$ where a minimum of (\ref{eq:Hari1}) occurs at $y = z_{m+1}$.
Since the Hausdorff distance between $D_1(x)$ and $X_0$ is less than $ \delta$ there exists a point $y_{m+1} \in X_0$ whose distance from $z_{m+1}$ is less than $\delta$. For this point $y_{m+1}$, we have
\beq \label{eq:prec}
\max\left(|1-d(x, y_{m+1})|, \left | \left\langle \frac{(x_1)}{|x_1|}, y_{m+1}\right\rangle \right |, \dots, \left | \left\langle \frac{(x_m)}{|x_m|}, y_{m+1}\right\rangle \right | \right)\leq \de.\eeq Since
$$\max\left(|1-d(x, x_{m+1})|, \left | \left\langle \frac{(x_1)}{|x_1|}, x_{m+1}\right\rangle \right |, \dots, \left | \left\langle \frac{(x_m)}{|x_m|}, x_{m+1}\right\rangle \right | \right)
$$ is no more than the corresponding quantity in (\ref{eq:prec}), we see that for each $m+1 \leq n$,
$$\max\left(|1-d(x, x_{m+1})|, \left | \left\langle \frac{(x_1)}{|x_1|}, x_{m+1}\right\rangle \right |, \dots, \left | \left\langle \frac{(x_m)}{|x_m|}, x_{m+1}\right\rangle \right | \right) < \delta.$$
Let $\widetilde V$ be an $D \times d$ matrix whose $i^{th}$ column is the column $x_i$. Let the operator $2$-norm of a matrix $Z$ be denoted $\|Z\|$.  For any distinct $i,j$ we have  $|\langle x_i , x_j \rangle|<\delta$, and for any $i$, $|\langle x_i, x_i\rangle - 1|<2\delta$, because $0 < 1-\de < |x_i| < 1$. Therefore,  $$\|{\widetilde V}^t {\widetilde V} - I\| \leq C_1 d\delta.$$
Therefore, the singular values of $\widetilde V$ lie in the interval $$I_C=(\exp(- C{d}\delta), \exp(C{d}\delta)) \supseteq (1 - C_1d\de, 1 + C_1d\de).$$
For each $i \leq n$, let $x'_i$ be the nearest point on $D_1(x)$ to the point $x_i$. Since the Hausdorff distance of $X_0$ to $D_1(x)$ is less than $\delta$, this implies that $|x'_i - x_i| < \delta$ for all $i \leq n$.
Let $\widehat{V}$ be an $D \times d$ matrix whose $i^{th}$ column is $ x'_i$. Since for any distinct $i,j$ we have $|\langle x'_i, x'_j\rangle|<3\delta +\de^2$, and for any $i$, $|\langle x'_i , x'_i \rangle - 1|<4\delta$,  $$\|\widehat{V}^t \widehat{V} - I\| \leq C{d}\delta.$$ This means that the singular values of $\widehat{V}$ lie in the interval $I_C$.

We shall now proceed to obtain an upper bound of $Cd\delta$ on the Hausdorff distance between $X_0$ and $\widetilde{D}_1(x)$. Recall that the unit $d$-disc $\widetilde{D}_1(x)$ is $\widetilde{A}_x \cap B_1(x)$.  By the triangle inequality, since the Hausdorff distance of $X_0$ to $D_1(x)$ is less than $\delta$, it suffices to show that the Hausdorff distance between $D_1(x)$ and $\widetilde{D}_1(x)$ is less than $Cd\delta.$



Let $x'$ denote a point on $D_1(x)$.  We will show that there exists a point $z' \in \widetilde{D}_1(x)$ such that 
$|x' - z'| < C d \delta.$ 

Let $\a \in \R^d$ be such that $\widehat{V}\alpha = x'$. By the bound on the singular values of $\widehat{V}$, we have $|\alpha| < \exp( C{d}\delta).$
Let $y' = \widetilde{V}\alpha$. Then, by the bound on the singular values of $\widetilde{V}$, we have $|y'| \leq \exp(C{d}\delta)$.
Let $z' = z' = \min(1-\delta,|y'|) |y'|^{-1}y' $. By the preceding two lines,
$  z'$ belongs to $\widetilde{D}_1(x).$ We next obtain an upper bound on $|x' -  z'|$
\begin{eqnarray} |x' - z'| & \leq & |x' - y'|\label{eq:1stterm}\\ &  & +|y' - z'|\label{eq:2ndterm}.\end{eqnarray}
We examine the  term in (\ref{eq:1stterm})
\begin{eqnarray*}  |x' -  y'|  & = &   |\widehat{V}\alpha - \widetilde{V}\alpha|\\
                    & \leq &   \sup_i |x_i - x'_i|(\sum_{j = 1}^d |\alpha_j|)\\
                                                               & \leq &   \delta {d}\exp(Cd\delta).
\end{eqnarray*}
We next bound the term in (\ref{eq:2ndterm}).
\begin{eqnarray*} 
|y' - z'| & \leq &   |y'|(1 - \exp(-C{d}\delta))\\
          & \leq & C{d}\delta. 
\end{eqnarray*}
Together, these calculations show that
$|x' - z'| < C {d}\delta.$
A similar argument shows that if $ z''$ belongs to $\widetilde{D}_1(x)$ then there is a point $p' \in D_1(x)$ such that 
$|p'-z''| < C {d}\delta$; the details follow.
Let $\widehat{V}\beta = z''$. From the bound on the singular values of $\widehat{V}$, $|\beta|<\exp(C{d}\delta).$
Let $q' :=  \widetilde{V}\beta$. Let $p' := \min(1-\delta,|q'|) |q'|^{-1}q'.$ Then,
\begin{eqnarray*} |p' - z''| & \leq & |q' - z''| + |p' - q'|\\
                                   & \leq & |\widetilde{V}\beta - V\beta| + |1 - \widetilde{V}\beta|\\
                                   & \leq &  \sup_i |x_i - x'_i|(\sum_{j = 1}^d|\beta_j|) + C\delta d\\
                                   & \leq & \delta{d}\exp(C{d}\delta) + C \delta{d}
                                   \leq   C \delta{d}.
\end{eqnarray*}
This proves that the Hausdorff distance between $X_0$ and $\widetilde{D}_1(x)$ is bounded above by $Cd\delta$ where $C$ is a universal constant.
\end{proof}

\begin{proof}[Proof of Lemma~\ref{lem:5.3-diff-eq}]
Let all lengths be rescaled so that $\sqrt{2\pi} \sigma = 1.$ Let $\widehat{x} = \frac{f(x)}{|f(x)|}$, and denote $|f(x)|$ by $T$. By (\ref{eq:de0}), for $x \in B_d(0, r_c),$
$ T  \leq   C \de. $
Then, $$g(x) = \frac{\left(\int\limits_{B_{D-d}(\frac{r_c}{2})} \langle z - T\widehat{x}, \widehat x \rangle \exp\left(- \pi\|z - T\widehat{x}\|^2 \right)\la_{D-d}(dz)\right)\widehat{x}}{\int\limits_{B_{D-d}(\frac{r_c}{2})}\exp(- \pi \|z - T\widehat{x}\|^2 )\la_{D-d}(dz)}.$$

Fot $t \in \R $, let 
$$h(t) := \int\limits_{B_{D-d}(\frac{r_c}{2})} \langle z - t\widehat{x}, \widehat x \rangle \exp\left(- \pi\|z - t\widehat{x}\|^2 \right)\la_{D-d}(dz).$$
Then,
\beqs \partial_t h(t) & = &  2\pi\int\limits_{B_{D-d}(\frac{r_c}{2})} \langle z - t\widehat{x}, \widehat x \rangle \langle z - t\widehat{x}, \widehat{x} \rangle \exp\left(- \pi\|z - t\widehat{x}\|^2 \right)\la_{D-d}(dz)\\
& & -    \int\limits_{B_{D-d}(\frac{r_c}{2})} \exp\left(- \pi\|z - t\widehat{x}\|^2 \right)\la_{D-d}(dz).\eeqs

Let $$J_r := 2\pi\int\limits_{\R^{D-d} \setminus B_{D-d}(r_c)} \langle z - t\widehat{x}, \widehat x \rangle \langle z - t\widehat{x}, \widehat{x} \rangle \exp\left(- \pi\|z - t\widehat{x}\|^2 \right)\la_{D-d}(dz)$$
Let $r' = \frac{r_c}{3} > \left(\frac{2}{3}\right)\sqrt{\sigma \tau D^{\frac{1}{2}}}.$ Note that $$B_{D-d}(0, \frac{r_c}{3}) \subseteq B_{D-d}(T\widehat{x}, \frac{r_c}{2}).$$ Then,
\beqs J_{r'} \exp(\pi r'^2/2)&  \leq & 2\pi\int_{\R^{D-d}} \frac{|x|^2}{D-d} \exp(-\pi |x|^2) \exp(\pi|x|^2/2)\la_{D-d}(dx) =  2^{ \frac{D-d}2}.\eeqs
Therefore $$J_{r'} \leq 2^{ \frac{D-d}{2}}\exp( - \frac{\pi r'^2}{2}).$$

 Let \beq I_{r'} := \int_{|x|>r'} \exp(-\pi|x|^2)\la_{D-d}(dx).\eeq
The left hand side $I_{r'}$ can be bounded above as follows. 
\beqs I_{r'} \exp(\pi r'^2/2) & \leq &  \int_{\R^{D-d}} \exp(-\pi |x|^2) \exp(\pi|x|^2/2)\la_{D-d}(dx)\\
& = & 2^{\frac{D-d}{2}}.\eeqs
Therefore $$I_{r'} \leq 2^{\frac{D-d}{2}} \exp(-\pi r'^2/2) \leq 2^{ \frac{D-d}{2}}\exp( - \frac{\pi r'^2}{2}).$$
We see that $J_{r'}$ and $I_{r'}$ are trivially non-negative. 
It follows that for $t \in [0, T]$, \beq |\partial_t h(t)| \leq 2^{ \frac{D-d}{2}}\exp( - \frac{\pi r'^2}{2}).
\eeq
Since $h(0) =0$, we see that $|h(T)| \leq 2^{ \frac{D-d}{2}}\exp( - \frac{\pi r'^2}{2})T.$
Therefore, $\forall x \in B_d(0, r)$,  \beq |g(x)| & \leq & \frac{2^{ \frac{D-d}{2}}\exp( - \frac{\pi r'^2}{2})T}{ 1 - 2^{ \frac{D-d}{2}}\exp( - \frac{\pi r'^2}{2})}\\
& \leq & 2^{ 2 + \frac{D-d}{2}}\exp( - \frac{\pi r'^2}{4\pi\sigma^2})T
 \leq \frac{\sigma^2}{\tau}.\eeq

\end{proof}

\begin{proof}[Proof of Lemma~\ref{lem:11-May27}]
\beq \nonumber \big|\Pi_{D-d} (e_y - z_{y, 2}) -f(y) \big|& = &  \bigg|\int_{x \in B_d(0, \frac{\tau}{20})} \E(f(x) - f(y)+ \widehat{G}^x_{D-d})  \mu''(dx)\bigg|\\
& \leq & \label{eq:254.2-Mar6}\bigg|\int_{x \in B_d(y, \frac{r_c}{2})} \E(f(x) - f(y) + \widehat{G}^x_{D-d})  \mu''(dx) \bigg| \\\label{eq:253.2-Mar5} & + & \bigg|\int_{x \in B_d(0, \frac{\tau}{20})\setminus B_d(y, \frac{r}{2})} \E(f(x) - f(y) + \widehat{G}^x_{D-d})  \mu''(dx)\bigg|. \eeq
Observe that the signed version of (\ref{eq:254.2-Mar6}) can be rewritten as follows.
 \beq & & \int_{x \in B_d(y, \frac{r_c}{2})} \E(f(x) - f(y) + \widehat{G}^x_{D-d})  \mu''(dx) \\& & =  \label{eq:253-Mar5} \left(\frac{1}{2}\right) \left(\int_{x \in B_d(y, \frac{r_c}{2})} \E(f(x) - f(y) - \partial_{x - y} f(y) + \widehat{G}^x_{D-d})\mu''(dx)\right) \\ & & +\left(\frac{1}{2}\right) \left( \int_{x \in B_d(y, \frac{r_c}{2})} \E(f(2y - x) - f(y) - \partial_{y-x} f(y) + \widehat{G}^{2y-x}_{D-d})  \mu''(d(2y-x))\right)\hspace{-1cm}\\
 &  &+ \label{eq:254-Mar5} 
\left(\frac{1}{2}\right)\left( \int_{x \in B_d(y, \frac{r}{2})}  \partial_{x-y} f(y)  (\mu''(dx) - \mu''(d(2y-x)) \right).
\eeq

We observe that twice the magnitude of  (\ref{eq:253-Mar5}) satisfies,
\beqs \bigg|\int_{x \in B_d(y, \frac{r_c}{2})} \E(f(x) - f(y) - \partial_{x - y} f(y) + \widehat{G}^x_{D-d})  \mu''(dx)\bigg|  \leq  \eeqs 
\beq \bigg|\int_{x \in B_d(y, \frac{r_c}{2})} (f(x) - f(y) - \partial_{x - y} f(y)) \mu''(dx)\bigg|  + \frac{\sigma^2}{\tau}  \leq \eeq 

\beqs \bigg|\int_{x \in B_d(y, \frac{r_c}{2})} \left(\frac{C |x - y|^2}{\tau}\right) \Gamma^{-1} \gamma_d(\Pi_d V_y - x) \gamma_{D-d}(\Pi_{D-d}V_y - f(x))\mu'(dx)\bigg|  + \frac{\sigma^2}{\tau}\leq \frac{Cd\sigma^2}{\tau}.\eeqs
By symmetry, The same bound applies to twice the next term. Now we need to bound (\ref{eq:254-Mar5}). To do so, observe that 
\beqs
& &\bigg|\int_{x \in B_d(y, \frac{r_c}{2})}  \partial_{x-y} f(y)  (\mu''(dx) - \mu''(d(2y-x)) \bigg|\\
& { \le}&  
\int_{x \in B_d(y, \frac{r_c}{2})}  |\partial_{x-y} f(y)| \cdot |(\mu''(dx) - \mu''(d(2y-x))| \\ 
&{ \le}&  
 \int_{x \in B_d(y, \frac{r_c}{2})} \left(\frac{C\de|x-y|}{r_c} + \frac{C |x - y|^2}{\tau}\right)\Gamma^{-1} \gamma_d(\Pi_d V_y - x) \left(\frac{|x - y|}{\tau}\right)\mu'(dx)\\
 &{ \le}& \left(\frac{Cd\sigma^2r_c}{\tau^2} + \frac{C\sigma^3d^{\frac{3}{2}}}{\tau^2}\right) < \frac{C\sigma^2}{\tau}.
\eeqs
In the last step, we used the fact that $r_c < \frac{\tau}{C d}.$
Lastly, from the tail decay of $\frac{d\mu''}{d\mu'}$ as expressed by (\ref{eq:113-oct3-2020}) and (\ref{eq:114-oct3-2020}) and the fact that $\frac{r_c}{2} \geq \big|\E(f(x) - f(y) + \widehat{G}^x_{D-d})\big|$, it follows that the term (\ref{eq:253.2-Mar5}) is bounded above by $\frac{\sigma^2}{\tau}.$
\end{proof}

\begin{proof}[Proof of Lemma~\ref{lem:12-May}] { We see that}
\beq \nonumber \big|\Pi_{n-D} (\ey - z^n_{y,2}) -\f(y) \big|& = &  \bigg|\int_{x \in B_d(0, \frac{\tau}{20})} \E(\f(x) - \f(y))  \mu''(dx)\bigg|\\
 &&\hspace{-2.4cm} \leq  \label{eq:254.2-Mar8}\bigg|\int_{x \in B_d(y, \frac{r_c}{2})} (\f(x) - \f(y))  \mu''(dx) \bigg| \\\label{eq:253.2-Mar8} &  & \hspace{-2cm}+\bigg|\int_{x \in B_d(0, \frac{\tau}{20})\setminus B_d(y, \frac{r_c}{2})} \E(\f(x) - \f(y))  \mu''(dx)\bigg|. \eeq
Observe that the signed version of (\ref{eq:254.2-Mar8}) can be rewritten as  \beq \int_{x \in B_d(y, \frac{r_c}{2})} (\f(x) - \f(y))  \mu''(dx)  = \eeq  
\beq \label{eq:253-Mar8} \left(\frac{1}{2}\right) \left(\int\limits_{x \in B_d(y, \frac{r_c}{2})} (\f(x) - \f(y) - \partial_{x - y} \f(y))\mu''(dx)\right)  +\eeq  \beqs \left(\frac{1}{2}\right) \left( \int_{x \in B_d(y, \frac{r_c}{2})} (\f(2y - x) - \f(y) - \partial_{y-x} \f(y))  \mu''(d(2y-x))\right) + \eeqs \beq\label{eq:254-Mar8} 
\left(\frac{1}{2}\right)\left( \int_{x \in B_d(y, \frac{r_c}{2})}  \partial_{x-y} \f(y)  (\mu''(dx) - \mu''(d(2y-x)) \right).
\eeq

We observe that twice the magnitude of  (\ref{eq:253-Mar8}) satisfies,
\beqs \bigg|\int_{x \in B_d(y, \frac{r_c}{2})} \E(\f(x) - \f(y) - \partial_{x - y} \f(y) + \widehat{G}_{n-D})  \mu''(dx)\bigg|  & = & \\
 \bigg|\int_{x \in B_d(y, \frac{r_c}{2})} (\f(x) - \f(y) - \partial_{x - y} \f(y)) \mu''(dx)\bigg| & \leq & \\
 \bigg|\int_{x \in B_d(y, \frac{r_c}{2})} \left(\frac{C |x - y|^2}{\tau}\right) \Gamma^{-1} \gamma_d(\Pi_d V_y - x) \gamma_{D-d}(\Pi_{D-d}V_y - \f(x))\mu'(dx)\bigg| & \leq&  \frac{Cd\sigma^2}{\tau}.\eeqs
By symmetry, the same bound applies to twice the next term. Now we need to bound (\ref{eq:254-Mar8}). To do so, observe that 
\beqs
& &\bigg|\int_{x \in B_d(y, \frac{r_c}{2})}  \partial_{x-y} \f(y)  (\mu''(dx) - \mu''(d(2y-x)) \bigg| \\
&{ \le}  &
\int_{x \in B_d(y, \frac{r_c}{2})}  |\partial_{x-y} \f(y)| \cdot |(\mu''(dx) - \mu''(d(2y-x))|\\
&  { \le}&  
 \int_{x \in B_d(y, \frac{r_c}{2})} \left(C |x-y| + \frac{C |x - y|^2}{\tau}\right)\Gamma^{-1} \gamma_d(\Pi_d V_y - x) \left(\frac{|x - y|}{\tau}\right)\mu'(dx)
 \\ &{ \le}& 
 \left(\frac{Cd\sigma^2}{\tau} + \frac{C\sigma^3d^{\frac{3}{2}}}{\tau^2}\right) < \frac{C\sigma^2}{\tau}.
\eeqs
Lastly, as we show below, from the tail decay of $\frac{d\mu''}{d\mu'}$, proceeding in a way that is analogous to 
(\ref{eq:251-Mar8.1}), it follows that the term (\ref{eq:253.2-Mar8}) is bounded above by $\frac{\sigma^2}{\tau}.$ Let us denote $2y-x$ by $z$.
{ Thus,}
\beqs
& & \bigg|\int\limits_{x \in B_d(y, \frac{\tau}{20}) \setminus B_d(y, \frac{r_c}{2})} (\f(x) - \f(y))  \mu''(dx)\bigg|  \\ 
&=& \left(\frac{1}{2}\right) \bigg|\int\limits_{x \in B_d(y, \frac{\tau}{20}) \setminus B_d(y, \frac{r_c}{2})} (\f(x) - \f(y) - \partial_{x - y} \f(y))\mu''(dx)\bigg| \\& & +\left(\frac{1}{2}\right) \bigg| \int\limits_{x \in B_d(y, \frac{\tau}{20}) \setminus B_d(y, \frac{r_c}{2})} (\f(z) - \f(y) - \partial_{y-x} \f(y))  \mu''(dz)\bigg|\\
& & + \
\left(\frac{1}{2}\right)\bigg| \int\limits_{x \in B_d(y, \frac{\tau}{20}) \setminus B_d(y, \frac{r_c}{2})}  \partial_{x-y} \f(y)  (\mu''(dx) - \mu''(dz) \bigg| \\
&{ \le}&
 \int\limits_{x \in B_d(y, \frac{\tau}{20}) \setminus B_d(y, \frac{r_c}{2})} \left(\frac{C|x - y|^2}{\tau}\right)\mu''(dx)\\
 & & +  \int\limits_{x \in B_d(y, \frac{\tau}{20}) \setminus B_d(y, \frac{r_c}{2})} \left(\frac{C|x - y|^2}{\tau}\right)\mu''(dz) \\
 &&+  \int\limits_{x \in B_d(y, \frac{\tau}{20}) \setminus B_d(y, \frac{r_c}{2})} \left(C|x-y| + \frac{C|x - y|^2}{\tau}\right)\mu''(dx) \leq \frac{\sigma^2}{\tau} .\eeqs

\end{proof}
\section{ A bound on the third derivative of $\Pi_x F(x)$}

We now proceed to obtain an upper bound on $\|\partial_v^3 \left(\Pi_x F(x)\right)\|.$
\begin{claim}
 \beq \|( \partial^3_v \ta_i(x))_{i \in[N_3]}\|_{\frac{d+k}{d+k-3}}  \leq  Cd^3.\eeq 
\end{claim}
\begin{proof} This follows from $c < \ta < C$, and the discussion below. Suppose $x$ belongs to the unit ball in $\R^D$. Then,
\beqs 
\partial_v^3 ( 1 - \|x\|^2)^{k+d}  & = & \partial^2_v((k+d)(1 - \|x\|^2)^{k+d-1}(2 \langle x, v \rangle))\\
& = & (k+d) \partial_v\bigg((k+d-1) (1 - \|x\|^2)^{k+d-2}(4 \langle x, v \rangle)^2\\
& &\hspace{2cm}+    (1 - \|x\|^2)^{k+d-1}(2 \langle v, v \rangle)\bigg)\\
& = & (k+d)((k+d-1)(\partial_v((1 - \|x\|^2)^{k+d-2}) (4 \langle x, v\rangle^2)) \\
& &+ (k+d)(\partial_v((1 - \|x\|^2)^{k+d-1}))(2 \langle v, v \rangle).
\eeqs
Therefore, 
 \beqs 
 & &\hspace{-1cm}\|( \partial^3_v \ta_i(x))_{i \in[N_3]}\|_{\frac{d+k}{d+k-3}} \\
 & \leq & C\left(d^3\ta + d^2\|( \partial_v \ta_i(x))_{i \in[N_3]}\|_{\frac{d+k}{d+k-3}} + d\|( \partial^2_v \ta_i(x))_{i \in[N_3]}\|_{\frac{d+k}{d+k-3}}\right)\\ & \leq &
C\left(d^3\ta + d^2\|( \partial_v \ta_i(x))_{i \in[N_3]}\|_{\frac{d+k}{d+k-1}} + d\|( \partial^2_v \ta_i(x))_{i \in[N_3]}\|_{\frac{d+k}{d+k-2}}\right)\\
& \leq & C\left(d^3\right).\eeqs
\end{proof}
As a consequence, we see the following.
\beq|( \partial^3_v \ta(x))| & \leq &   \|( \partial^3_v \ta_i(x))_{i \in[N_3]}\|_{\frac{d+k}{d+k-3}} \|(1)_{i \in [N_3]}\|_{\frac{d+k}{3}}\\
& \leq & Cd^6.\eeq

\begin{lemma} \label{lem:3-16}We have for any $v \in \R^n$ such that $|v| = 1$, and any $x \in \R^n$ such that $dist(x, \MM) \leq \frac{cr}{d}$, \beq   \|(\partial^3_v \a_i(x))_{i \in[N_3]}\|_{\frac{d+k}{d+ k -3}} \leq Cd^6. \eeq \end{lemma}
\begin{proof}
We see that 
\beq \partial^3_v (\a_i \ta) & = & (\partial^3_v \a_i) \ta + (3 \partial^2_v \a_i) (\partial_v \ta) + (3 \partial_v \a_i) (\partial^2_v \ta) + ( \a_i) \partial^3_v \ta. \eeq

Therefore,
\beqs  & & \ta(x) \|(\partial^3_v \a_i(x))_{i \in[N_3]}\|_{\frac{d+k}{d+ k -3}} \\
& = & \|\left(- \partial^3_v (\ta_i)+(3 \partial^2_v \a_i) (\partial_v \ta) + (3 \partial_v \a_i) (\partial^2_v \ta) + ( \a_i) \partial^3_v \ta\right)_{i \in [N_3]}\|_{\frac{d+k}{d+ k -3}}\eeqs

The right hand side above can be bounded above by 
\beqs  Cd^3 + \|((3 \partial^2_v \a_i) \partial_v \ta)_i\|_{\frac{d+k}{d+ k -3}} + \|((3 \partial_v \a_i) \partial^2_v \ta)_i\|_{\frac{d+k}{d+ k -3}} + \|(( \a_i) \partial^3_v \ta)_{i \in [N_3]}\|_{\frac{d+k}{d+ k -3}}.\eeqs
This is bounded above by \beqs Cd^3 + |\partial_v \ta|\|((3 \partial^2_v \a_i) )_i\|_{\frac{d+k}{d+ k -2}} + 
 |\partial^2_v \ta|\|((3 \partial_v \a_i) )_i\|_{\frac{d+k}{d+ k -1}} +  |\partial^3_v \ta|\|(( \a_i) )_{i \in [N_3]}\|_{1},
 \eeqs
which is in turn bounded above by 
\beqs Cd^3 + (Cd^2)(Cd^2) + (Cd^4)(Cd) + (Cd^6), \eeqs
in which the dominant term is $Cd^6$.

\end{proof}
\begin{lemma}\label{lem:19-May27}
\beq  \|\partial_v^3 \left(\Pi_x F(x)\right)\| \leq C d^9 \de. \eeq

\end{lemma}

\begin{proof}
{ We have }\beq \|\partial_v^3 \left(\Pi_x F(x)\right)\| & \leq & \| (\partial_v^3 \Pi_x) F(x)\|\label{eq:3-57}\\
& + & \|3 (\partial^2_v \Pi_x) \partial_v F(x)\|\label{eq:3-58} \\ 
& + & \|3 (\partial_v \Pi_x) \partial^2_v F(x)\|\label{eq:3-58.5} \\ 
&+ & \|\Pi_x \partial_v^3 F(x)\|.\label{eq:3-59}\eeq 

We first bound from above the right side of (\ref{eq:3-57}).
Let $A := zI - M(x)$ and $B := (zI - M(x))^{-1}.$ Then,
\beqs 0 & = & \partial^3_v (AB) \\
              & = & (\partial_v^3 A) B + 3 (\partial_v^2 A)(\partial_v B) + 3 (\partial_v A)(\partial_v^2 B) + A(\partial_v^3 B).\eeqs 
Thus, 
\beqs
-A (\partial_v^3 B) = (\partial_v^3 A) B + 3 (\partial_v^2 A)(\partial_v B) + 3 (\partial_v A)(\partial_v^2 B), \eeqs
and so,
\beqs
 \partial_v^3 B = -B(\partial_v^3 A) B - 3 B(\partial_v^2 A)(\partial_v B) - 3 B(\partial_v A)(\partial_v^2 B), \eeqs
Thus,
\beqs
 \|\partial_v^3 B \| \leq  C\left[\|\partial_v^3 A\| +  \|\partial_v^2 A\|\|\partial_v B\| + \|\partial_v A\|\|\partial_v^2 B\|\right]\eeqs
and
\beqs \|\partial_v^3 A\| & = & \|\partial_v^3 \sum_i \a_i \Pi_i\|\\
& = & \sum_i |\partial_v^3\a_i(\Pi_i - \Pi_1)|\\
& \leq & \|(\partial_v^3\a_i)_{i \in [N_3]}\|_{\frac{d+k}{d+k-3}}\|(\de)_{i \in [N_3]}\|_{\frac{d+k}{3}}\\
& \leq & Cd^9 \de.\eeqs


We already know by (\ref{eq:3-176}) that  \beqs\|\partial_v \Pi_x\| \leq Cd^3\de \eeqs and \beqs   \| \partial_v^2 \Pi_x \| 
                                                       & \leq & Cd^6\de.\eeqs
We have shown that  \beqs \|\partial_v F(x)\| 
& \leq & 1 + C d^3\de.\eeqs

We have also already shown in (\ref{eq:77}) that \beqs  \|\partial_v^2 F(x)\|
& \leq & Cd^6 \de.   
\eeqs
 
We proceed to get an upper bound on $\|\Pi_x \partial_v^3 F(x)\|$,

\beq  \|\Pi_x \partial_v^3 F(x)\| & \leq &  \|\partial_v^3 F(x)\|\\
& \leq & \|\partial_v^3 (F(x) - F_1(x))\|\\
                                                      & \leq & \sum_i |\partial_v^3 \a_i(x)|\| F_i(x) - F_1(x)\|\label{eq:3.1-75}\\
                                                        & & +  \sum_i 3| \partial_v^2 \a_i(x)| \|\partial_v F_i(x)- \partial_v F_1(x)\label{eq:3.1-76}\|\\
                                                        &&  +  \sum_i 3|\partial_v \a_i(x)| \|\partial_v^2 F_i(x)\|\label{eq:3.1-77}\\
&  &+   \sum_i | \a_i(x)| \|\partial_v^3 F_i(x)\|.\label{eq:3.1-78}
\eeq

For each $i$, $\partial_v^2 F_i$ is $0$, and so,  the above expression reduces to 
\beqs  \sum_i |\partial_v^3 \a_i(x)|\| F_i(x) - F_1(x)\| +  \sum_i 3| \partial_v^2 \a_i(x)| \|\partial_v F_i(x)- \partial_v F_1(x)\|.\eeqs

Here, we have
\beq \nonumber  \sum_i |\partial_v^3 \a_i(x)|\| F_i(x)- F_1(x)\| & \leq & \| (|\partial_v^3 \a_i(x)|)_{i\in[N_3]}\|_{\frac{d+k}{d+k-3}}\|(\| F_i(x)- F_1(x)\|)_{i \in [N_3]}\|_{\frac{d+k}{3}}\hspace{-15mm}\\
                                                                        & \leq & (Cd^6)(d^3 \de)
                                                                           =  Cd^9 \de, 
\eeq
and
\beqs \nonumber
& &\hspace{-1cm}\sum_i | \partial^2_v \a_i(x)| \|\partial_v F_i(x) - \partial_v F_1(x)\|\\
 & \leq & \|(|\partial^2_v \a_i(x)|)_{i\in [N_3]}\|_{\frac{d+k}{d+k-2}} \|(\|\partial_v F_i(x)- \partial_v F_1(x)\|)_{i\in[N_3]}\|_{\frac{d+k}{2}}\\
                                                                             & \leq & C d^4 (d^2\de)
                                                                                 =  C d^6\de. \eeqs

\end{proof}
\section{Quantitative implicit and inverse function theorems}\label{sec:quant}
 
In this subsection, we provide for the reader's convenience,  versions of the implicit and inverse function theorems with quantitative bounds on the derivatives that do not depend on the dimensions involved. We think it is very likely that such theorems exist in the literature, but are not aware of a specific reference. 

We begin with the inverse function theorem.

Let $g: \R^p \ra \R^p$ be a $\C^2$ function on whose derivatives the following bounds hold.

At any point $x \in B_p(0, 1)$, denoting by $Jac_g$ the Jacobian matrix of $g$, we have \beq\label{eq:11}\|Jac_g - I\| \leq \eps_1/4\eeq for some $\eps_1 \in [0, 1].$

For any non-zero vector $v$ and $x$ as before, \beq\label{eq:12}\left\|\frac{\partial^2 g(x)}{\partial v^2}\right\| \leq \left(\frac{\eps_2}{4}\right) |v|^2.\eeq

By (\ref{eq:11}), for any $x \neq x'$, both belonging to $B_p(0, 1)$, $$|g(x) - g(x') - (x - x')| \leq |x-x'| (1/4),$$ which implies that $g(x) \neq g(x')$. Applying the Inverse Function Theorem (\cite{Narasimhan}), 
there exists a function $f: g(B_p(0, 1)) \ra B_p(0, 1)$ such that 
$f(g(x)) = x$, for all $x \in B(0, 1)$. Let $\tf = w\cdot f$ for some fixed non-zero vector $w$. Let $g = (g_1, \dots, g_p)$, where each $g_i$ is a real-valued function. The Jacobian of the identity function is $I$. Therefore, by the chain rule,
\beq \label{eq:12.01} \left(\left(\frac{df_i}{dg_j}\right)_{i, j \in [p]}\right) Jac_g = I, \eeq implying by (\ref{eq:11}) that

\beq \label{eq:12.02} \left\| \left(\left(\frac{df_i}{dg_j}\right)_{i, j \in [p]}\right)\right\| \leq (1 - \eps_1/4)^{-1}.\eeq

The second derivative of a linear function is $0$ and so \beq 0 = \frac{\partial^2 \tf (g) }{\partial v^2} (x) =  \sum_{i, j} \frac{d^2\tf}{{dg_i}{dg_j}}\left(\frac{dg_i}{dv}\right)\left(\frac{dg_j}{dv}\right) + \sum_j \frac{d\tf}{dg_j}\left( \frac{d^2g_j}{dv^2}\right). \eeq
Therefore, 

 \beq  \sum_{i, j} \frac{d^2\tf}{{dg_i}{dg_j}}\left(\frac{dg_i}{dv}\right)\left(\frac{dg_j}{dv}\right) = (-1) \sum_j \frac{d\tf}{dg_j}\left( \frac{d^2g_j}{dv^2}\right), \eeq
and so by Cauchy-Schwartz { inequality}, 
\beq\label{eq:12.1}  \left|\sum_{i, j} \frac{d^2\tf}{{dg_i}{dg_j}}\left(\frac{dg_i}{dv}\right)\left(\frac{dg_j}{dv}\right) \right| \leq \left \| \left(\left(\frac{d\tf}{dg_j}\right)_{j \in [p]}\right)\right\| \left\| \left(\left( \frac{d^2g_j}{dv^2}\right)_{j \in [p]}\right) \right\|. \eeq

By (\ref{eq:11}) there exists a unit vector $\widetilde{v}$ such that 
\beq\label{eq:12.2} \left|\sum_{i, j} \frac{d^2\tf}{{dg_i}{dg_j}}\left(\frac{dg_i}{d\widetilde v}\right)\left(\frac{dg_j}{d \widetilde v}\right) \right| & = &  \left\|Hess\, \widehat{F}\right\|  \left\|\frac{dg}{d\widetilde v}\right\|^2\\ & \geq &  \left\|Hess\, \widehat{F}\right\|  \inf \limits_{\|v\| = 1}\left\|\frac{dg}{dv}\right\|^2. \eeq 

Together (\ref{eq:12}), (\ref{eq:12.02}), (\ref{eq:12.1}) and (\ref{eq:12.2}) imply that 

\beqs \left\|Hess\, \widehat{F}\right\|  \inf \limits_{\|v\| = 1}\left\|\frac{dg}{dv}\right\|^2 \leq \left\|\left(\left(\frac{df_i}{dg_j}\right)_{i, j \in [p]}\right)w\right\| \sup\limits_{\|v\|=1}\left(\frac{\eps_2}{4}\right) \|v\|^2 \leq \left(\frac{\eps_2}{4 - \eps_1}\right)\|w\|. \eeqs
It follows that 
\beq \left\|Hess\, \widehat{F}\right\|  \leq \left(\frac{\eps_2}{4 - \eps_1}\right)\|w\| \sup \limits_{\|v\| = 1}\left\|\frac{dg}{dv}\right\|^{-2} \leq \left(\frac{16\eps_2}{(4-\eps_1)^3}\right) \|w\|. \eeq

Next, consider the setting of the Implicit Function Theorem. Let $h: \R^{m+n} \ra \R^n$ be a $\C^2-$function,
$$ h:(x, y) \mapsto h(x, y).$$
Let $g:B_{m+n} \ra \R^{m+n}$ be defined by $$g:(x, y) \mapsto (x, h(x, y)).$$

 Suppose the Jacobian of  $g$, $Jac_g$ satisfies $$\left\|Jac_g - I\right\| < \eps_1/4$$ on $B_{m+n}$ and that for any vector $v \in \R^{m+n}$, $$\left\|\frac{\partial^2 g(x)}{\partial v ^2}\right\| \leq \left(\frac{\eps_1}{4}\right)\|v\|^2$$ where $\eps_0, \eps_1, \eps_2 \in [0, 1]$.
Suppose also that $\|g(0)\| < \frac{\eps_0}{20}$.

Let $p = m+n$. Then, applying the inverse function theorem, we see that defining $f$ and $\widehat F$ as before, and choosing $\|w\|= 1$, 
\beq\label{eq:15} \left\|Hess\, \widehat{F}\right\|  \leq \frac{16\eps_2}{(4-\eps_1)^3}. \eeq

\begin{lemma}\label{lem:4}
On the domain of definition of $f$, \ie $g(B_{m+n})$ 
$$f ((x, y)) = (x, e(x, y))$$ for an appropriate $e$ and in particular, for $\|x\| \leq \frac{\eta}{2}$, where $\eta \in [0, 1]$,
$$f((x, 0)) = (x, e(x, 0))$$ and $$\|(x, e(x, 0))\| \leq \frac{8}{5}\left( \frac{\eps_0}{20} + \frac{\eta}{2}\right).$$
 Finally, for any $w \in \R^n$ such that $\|w\| = 1$, 
\beq \|Hess (e\cdot w)\| \leq \frac{16\eps_2}{(4-\eps_1)^3}.\eeq
\end{lemma}

\begin{proof}
It suffices to prove that if $z = (x, y) \in \R^p$ and $\|z\| \leq \eta/2$, where $\eta \in [0, 1]$,
then there exists a point $\widehat z$, where $\|\widehat z\| \leq \frac{8}{5}\left( \frac{\eps_0}{20} + \frac{\eta}{2}\right), $ such that $g(\widehat z) = z.$
We will achieve this by analysing  Newton's method for finding a sequence $\widehat z_0, \dots, \widehat z_k, \dots $ converging to a point $\widehat z $ that satisfies  $g(\widehat z) = z$. We will start with $\widehat z_0 = 0.$ 

The iterations of Newton's method proceed as follows.

For $i \geq 0$, 
\beq \widehat z_{i+1} = \widehat z_i - J_g^{-1}(\widehat z_i)(g(\widehat z_i) - z). \eeq
\begin{claim}
{ For} any $i \geq 0$, $\|\widehat z_i\| \leq \frac{8}{5}\left( \frac{\eps_0}{20} + \frac{\eta}{2}\right).$ 
\end{claim}
\begin{proof}
Observe that \beq\|\widehat z_{i+1} - \widehat z_i\| = \|J_g^{-1}(\widehat z_i)(g(\widehat z_i) - z) \|.\eeq

For $i=0$, \beq \|g(\widehat z_i) - z\| \leq \frac{\eps_0}{20} + \frac{\eta}{2}.\eeq and since $\|J_g^{-1}(\widehat z_i)\| \leq \frac{1}{1-\eps_1/4} \leq 4/3$, therefore 
\beq\label{eq:19} \|\widehat z_{i+1} - \widehat z_i\| \leq \left(\frac{4}{3}\right)\left( \frac{\eps_0}{20} + \frac{\eta}{2}\right) .\eeq

Suppose $i \geq 1$. 
\beq \label{eq:20} g(\widehat z_i) - z =   g\left(\widehat z_{i-1} - J_g^{-1}(\widehat z_{i-1})(g(\widehat z_{i-1}) - z)\right) - z. \eeq
Using the integral form of the remainder in Taylor's theorem, the right hand side of (\ref{eq:20}) equals
$$g(\widehat z_{i-1}) + J_g(\widehat z_{i-1})\left( - J_g^{-1}(\widehat z_{i-1})(g(\widehat z_{i-1}) - z) \right) + \Lambda  - z,$$ which simplifies to $\Lambda$,
where $$\Lambda = \int_0^1 {(1-t)} (\widehat z_{i} - \widehat z_{i-1})^T Hess_g(\widehat z_{i-1} + t(\widehat z_{i} - \widehat z_{i-1}))(\widehat z_{i} - \widehat z_{i-1})dt.$$
The norm of $\Lambda$ is bounded above as follows. Note that by the induction hypothesis, $\|\widehat z_i\| \leq \frac{8}{5}\left( \frac{\eps_0}{20} + \frac{\eta}{2}\right),$ and $\|\widehat z_{i-1}\| \leq \frac{8}{5}\left( \frac{\eps_0}{20} + \frac{\eta}{2}\right)$, which places both $\widehat z_i$ and $\widehat z_{i-1}$ within the unit ball. Therefore $\|(\widehat z_{i} - \widehat z_{i-1})^T Hess_g(\widehat z_{i-1} + t(\widehat z_{i} - \widehat z_{i-1}))(\widehat z_{i} - \widehat z_{i-1})\| \leq (\eps_2/4)\|\widehat z_i - \widehat z_{i-1}\|^2$ for any $t \in [0, 1]$.
{ Moreover,}
$$\|\Lambda\| \leq  \int_0^1 {(1-t)} \|(\widehat z_{i} - \widehat z_{i-1})\|^2 (\eps_2/4) dt =   \left(\frac{\eps_2}{8}\right) \|(\widehat z_{i} - \widehat z_{i-1})\|^2.$$

Therefore  \beqs\|\widehat z_{i+1} - \widehat z_i\| = \|J_g^{-1}(\widehat z_i)(g(\widehat z_i) - z) \| \leq \eeqs \beqs \left(\frac{4}{3}\right)\left(\frac{\eps_2}{8}\right) \|(\widehat z_{i} - \widehat z_{i-1})\|^2  = \left(\frac{\eps_2}{6}\right) \|(\widehat z_{i} - \widehat z_{i-1})\|^2.\eeqs

By recursion,
\beq\label{eq:22} \|\widehat z_{i+1} - \widehat z_i\| \leq \left(\frac{\eps_2^{2i}}{6^i}\right)\|\widehat z_1 - \widehat z_0\|^{2^i}. \eeq
Therefore, \beqs  \|\widehat z_{i+1}\| = \|\widehat z_{i+1} - \widehat z_0\| \leq \sum_{j = 1}^i \|\widehat z_{j+1} - \widehat z_j\| \leq  \eeqs
\beqs \frac{\|\widehat z_1 - \widehat z_0\|}{1 - \frac{\eps_2}{6}} \leq \left(\frac{4}{3}\left( \frac{\eps_0}{20} + \frac{\eta}{2}\right)\right)\left(\frac{6}{5}\right) = \frac{8}{5}\left( \frac{\eps_0}{20} + \frac{\eta}{2}\right). \eeqs
\end{proof}
 Recall that 
$g:B_{m+n} \ra \R^{m+n}$ is given by $$g:(x, y) \mapsto (x, h(x, y)).$$ Since $g$ is injective, it  follows that on the domain of definition of $f$, \ie $g(B_{m+n})$ 
$$f ((x, y)) = (x, e(x, y))$$ for an appropriate $e$.  
By (\ref{eq:19}) and (\ref{eq:22}) $(\widehat z_0, \dots, \widehat z_i, \dots)$ is a Cauchy sequence, and therefore has a unique limit point. By the preceding Claim, this limit $\widehat z$ satisfies $\|\widehat z\| \leq \frac{22}{25} < 1$. Therefore any point in $B_m\times B_n$ of the form $(x, 0)$ where $\|x\| = \frac{\eta}{2} \leq \frac{1}{2}$ belongs to $g(B_{m+n})$. Further, 
$$\|f((x, 0))\| \leq \frac{8}{5}\left( \frac{\eps_0}{20} + \frac{\eta}{2}\right).$$ In particular, setting $\eta = 0$, we have \beq \|f((0, 0))\| \leq \frac{2\eps_0}{25}. \label{eq:119a} \eeq

By (\ref{eq:12.02}) the function $e$ satisfies, for $\|x\| \leq 1/2$, 
\beq \label{eq:De} \|D_x e\|^2 & = & \|D_x f\|^2 - 1\\
& \leq & (1 - \eps_1/4)^{-2} - 1\\
& \leq & \eps_1. \label{eq:122}\eeq

By (\ref{eq:15}) the function $e$  satisfies, for any $w \in \R^n$ such that $\|w\| = 1$, 
\beq \|Hess (e\cdot w)\| \leq \frac{16\eps_2}{(4-\eps_1)^3}.\eeq
\end{proof}
We next obtain bounds for the $m^{th}-$order derivatives. Our focus will be in the case of $m \geq 2$. 


 In the remainder of this section, all norms on Euclidean spaces $\R^N, \R^M, \R^D, \dots$ are Euclidean norms.
$$|(v_1, \dots, v_N)| = \left(\sum_{1}^N v_i^2\right)^\frac{1}{2}.$$

Note that the usual implicit function theorem gives the function $\psi(x, z)$ that solves the equation $F(x, \psi(x, z)) = z$ at the end of this section. The purpose of this section is to derive bounds for the derivatives of $\psi$ in terms of the derivatives of $F$. 

\subsection{Differentiating composed maps}

Let $y = (y_1, \dots, y_M) = \Phi(x_1, \dots, x_N)$.

$$ z = G(y) = G\circ \Phi(x).$$

Let $v_1, \dots, v_m$ be vectors in $\R^N$. Let $\partial_v$ denote the directional derivative in the direction $v$.
Then, 
$\partial_{v_1}\dots\partial_{v_m}(G\circ\Phi)(x)$ is a sum of terms 
$$\sum_{p_1, \dots, p_{\nu_{max}}}\prod_{\nu=1}^{\nu_{max}} (\partial_{w_{1, \nu}}\dots\partial_{w_{{s_\nu}, \nu}}y_{p_\nu})(x) \cdot (\partial_{y_{p_1}}\dots \partial_{y_{p_{\nu_{max}}}} G(y))\big|_{y = \Phi(x)},$$
where each $s_\nu \geq 1$, and the list $$w_{1, 1}, \dots w_{s_1, 1}, w_{1, 2}, \dots, w_{s_2, 2}, \dots, w_{1, \nu_{max}}, \dots, w_{s_{\nu_{max}}, \nu_{max}},$$ may be permuted into the list $v_1, \dots, v_m.$ This follows by induction on $m$. 

So
$$\partial_{v_1}\dots\partial_{v_m}(G \circ \Phi),$$ is a sum of terms 

\ben 

\item[(A)] $\partial_{\zeta_1}\dots\partial_{\zeta_{\nu_{max}}} G(y)\big|_{y = \Phi(x)},$ where $\zeta_\nu \in \R^M$ is the vector whose $p^{th}-$coordinate is $\zeta_{\nu, p} = \partial_{w_{1, \nu}} \dots \partial_{{w_{{s_\nu}, \nu}}} y_p.$ That is $\zeta_\nu = \partial_{w_{1, \nu}} \dots \partial_{w_{s_\nu, \nu}} y.$
\een

where, as before, the concatenated list of all the $w's$ may be permuted into the list of $v's$. 

The only term of the form $(A)$ in which $y$ is differentiated $m$ times is 
$$\partial_{[\partial_{v_1}\dots\partial_{v_m}y]}G(y) = \sum_p (\partial_{v_1} \dots \partial_{v_m} y_p)(\frac{\partial G}{\partial y_p}(y))\big|_{y = \Phi(x)}.$$

Suppose we know that 

\ben 
\item[(*1)] $|\partial_{w_1}\dots\partial_{w_s} y| \leq C_s$ for $s < m$, whenever $|w_1|, \dots, |w_s| \leq 1,$
\item[(*2)] $|\partial_{\zeta_1}\dots\partial_{\zeta_{\nu_{max}}}G| \leq C_m^*$ for $\nu_{max} \leq m$, all $|\zeta_{\nu}| \leq 1.$
\een
Then, $$|\partial_{\zeta_1}\dots\partial_{\zeta_{\nu_{max}}} G| \leq C_m^* |\zeta_1|\dots|\zeta_{\nu_{max}}|$$ for any $\zeta's$ provided $\nu_{max} \leq m$.

Then $(*1)$, $(*2)$ and our discussion of $\partial_{v_1}\dots\partial_{v_m}(G \circ \Phi)$ (see (A)) together imply that (for $ |v_1|, \dots, |v_m| \leq 1$) together imply that (for $|v_1|, \dots, |v_m| \leq 1$)
$$\partial_{v_1} \dots \partial_{v_m} (G \circ \Phi)(x) = \sum_p(\partial_{v_1}\dots\partial_{v_m} y_p)\left(\frac{\partial G}{\partial y_p}\right)\bigg|_{y=\Phi(x)} + \de_0,$$
where
$|\de_0|$ is less or equal to a constant determined by the $C_s$ and $C_m^*.$ We write $\overline{C}_m$ to denote any such constant. Write $$G(x) = (G_1(x), \dots, G_M(x)).$$ Then,
$$\partial_{v_1}\dots\partial_{v_m}(G_q \circ \Phi)(x) = \sum_p (\partial_{v_1}\dots \partial_{v_m} y_p) \left(\frac{\partial G_q}{\partial {y_p}}\right)\bigg|_{y = \Phi(x)} + |\overline{\de}_q|,$$ where 
$(\sum_q |\de_q|^2)^{\frac{1}{2}} \leq \overline{C}_m.$ 

Let $\Omega_q^r$ be the $M \times M$ matrix that inverts the matrix $\frac{\partial G_q}{\partial y_p}\big|_{y = \Phi(x)}.$ Assume that $(\Omega_q^r)$ has norm $\leq 10$ (say) as a linear map from $\R^M$ to $\R^M$. Then, we find that 
\ben
\item[(*3)] $\sum_q \Omega_q^r \partial_{v_1}\dots \partial{v_m} (G_q \circ \Phi)(x) = \sum_{p, q} (\partial_{v_1} \dots \partial_{v_m} y_p) \cdot \Omega_q^r \left(\frac{\partial G_q}{\partial y_p}\right)\bigg|_{y = \Phi(x)} + \widetilde{\de}_r,$
\een
which is equal to $\partial_{v_1} \dots \partial_{v_m} y_r(x) + \widetilde{\de}_r,$
where $(\sum_r |\widetilde{\de_r}|^2 )^{\frac{1}{2}} \leq \overline{C}_m$. Here is what that means:

Let $\Phi:\R^N \ra \R^N$, and let $G$ be the inverse function of $\Phi$ in some neighborhood.

Suppose we have bounds on 

\ben
\item[(!1)] $|\partial_{v_1} \dots \partial_{v_k} G| $ for $k \leq m$ and $|v_1|, \dots, |v_k| \leq 1.$

\een
\ben
\item[(!2)] $|\partial_{v_1} \dots \partial_{v_k} \Phi| $ for $k \leq m-1$ and $|v_1|, \dots, |v_k| \leq 1.$

\een

Suppose the inverse of the Jacobian $\nabla_x G$ has norm $\leq 10$ as a matrix (\ie as a bounded linear operator on $\R^N$).

Then, we obtain bounds on 
\ben
\item[(!3)] $|\partial_{v_1} \dots \partial_{v_k} \Phi| $ for $k = m$ and $|v_1|, \dots, |v_k| \leq 1.$

\een

This holds for $m \geq 2$. Our bounds for $(!3)$ depend only on $m$, and our bounds for $(!1), (!2)$.

Starting from $m=2$, we may now use induction on $m$ to obtain the following result.

\subsection{Quantitative Inverse Function Theorem.}

Let $m \geq 2$, and let $G, \Phi$ be inverse images of each other in a neighborhood of a point in $\R^N$.

Suppose 
\begin{itemize}
\item $|(\nabla_x G)^{-1}(v)| \leq 10|v|$ for all values $v \in \R^N$.

\item $|\partial_{v_1} \dots \partial_{v_k} G| \leq C$ for $k \leq m$, $|v_1|, \dots, |v_k| \leq 1.$

\end{itemize}

Then, $|\partial_{v_1}\dots \partial_{v_k} \Phi| \leq C'$ for $k \leq m$, $|v_1|, \dots, |v_k| \leq 1,$ where $C'$ depends only on $C$ and $m$. In particular, $C'$ does not depend on $N$ (unless $C$ does).

Now for $x \in \R^N, y \in \R^D$, let $G(x, y)$ take values in $\R^D$. 

We want to solve the equation $$G(x, y) = z,$$ for the unknown $y$. Say the solution is $y = \Psi(x, z)$.
So $$G(x, \Psi(x, z)) = z.$$

Then, the following maps from $\R^N \times \R^D$ to itself are inverses of each other.

$$(x, y) \mapsto (x, G(x, y)),$$

$$(x, z) \mapsto (x, \Psi(x, z)).$$

Applying the quantitative inverse function theorem to these two maps, we obtain the following.

\subsection{Quantitative Implicit Function Theorem}\label{sec:quant_imp_func}

Let $m \geq 2$. Let $$G(x, y) = (G_1(x, y), \dots, G_D(x, y)),$$ for $x = (x_1, \dots, x_N), y = (y_1, \dots, y_D)$.

Suppose $y = \Psi(x, z)$ solves the equation $G(x, y) = z.$ Assume that 

$$|\partial_{v_1} \dots \partial_{v_k} G| \leq C$$ for $k \leq m$, $v_1, \dots, v_k \in \R^{N+D}$ of length $\leq 1$. Assume that the inverse of the matrix $$\left(\frac{\partial G_p}{\partial y_q}\right),$$ has norm 
at most 10 as a linear map from $\R^D$ to itself. Then also
$$|\partial_{v_1} \dots \partial_{v_k} \Psi| \leq \overline{C},$$ for $k \leq m$, and $v_1, \dots, v_k \in \R^{N+D},$ of length $\leq 1$, where $\overline C$ is determined by $C$ and $m$.

\end{document}